\documentclass[11pt]{amsart}

\pdfoutput=1

\usepackage{amsmath}
\usepackage[mathscr]{eucal}
\usepackage{tikz-cd}
\usepackage{enumerate}
\usepackage{tensor}
\usepackage{color}
\usepackage{graphicx}
\usepackage[top=1.1in, bottom=1in, left=1in, right=1in]{geometry}
\usepackage{hyperref}
\usepackage{cleveref}

\numberwithin{equation}{section}
\newtheorem{thm}{Theorem}[section]
\newtheorem{lemma}[thm]{Lemma}
\newtheorem{prop}[thm]{Proposition}
\newtheorem{cor}[thm]{Corollary}
\newtheorem{conj}[thm]{Conjecture}

\theoremstyle{definition}
\newtheorem{defn}[thm]{Definition}

\theoremstyle{remark}
\newtheorem{rmk}[thm]{Remark}

\newtheorem{eg}[thm]{Example}

\newcommand{\C}{{\mathbb C}}
\newcommand{\Z}{{\mathbb Z}}
\newcommand{\N}{{\mathbb N}}
\newcommand{\R}{{\mathbb R}}
\newcommand{\K}{{\mathbb K}}
\renewcommand{\H}{{\mathbb H}}
\newcommand{\tild}{\widetilde}
\newcommand{\ol}{\overline}
\newcommand{\ssigma}{{\boldsymbol\sigma}}
\newcommand{\glu}[2]{\tensor[_{#1}]{\#}{_{#2}}}
\newcommand{\into}{\hookrightarrow}
\newcommand{\Fuk}{\mathrm{Fuk}}
\newcommand{\W}{\mathcal W}
\newcommand{\Ws}{\mathcal W_\ssigma}
\newcommand{\Wsp}{\mathcal W_{\ssigma'}}
\newcommand{\Bs}{{\mathcal B(\sigma)}}
\newcommand{\Bss}{{\mathcal B_\ssigma(\sigma)}}
\newcommand{\Bsps}{{\mathcal B_{\ssigma'}(\sigma)}}

\title{On partially wrapped Fukaya categories}
\author{Zachary Sylvan}

\begin{document}

\begin{abstract}
We define a new class of symplectic objects called ``stops'', which roughly speaking are Liouville hypersurfaces in the boundary of a Liouville domain. Locally, these can be viewed as pages of a compatible open book. To a Liouville domain with a collection of disjoint stops, we assign an $A_\infty$-category called its partially wrapped Fukaya category. An exact Landau-Ginzburg model gives rise to a stop, and the corresponding partially wrapped Fukaya category is meant to agree with the Fukaya category one is supposed to assign to the Landau-Ginzburg model. As evidence, we prove a formula that relates these partially wrapped Fukaya categories to the wrapped Fukaya category of the underlying Liouville domain. This operation is mirror to removing a divisor.
\end{abstract}

\maketitle

\section{Introduction}

\subsection{Stops}

Our basic goal is to use geometric data to enhance the wrapped Fukaya category of a Liouville domain to the so-called ``partially wrapped Fukaya category''. Recall that the wrapped Fukaya category \cite{Abou-Seid_osa} of a Liouville domain $(M,\lambda)$ is the $A_\infty$ category $\W(M)$ whose objects are properly embedded exact Lagrangians, and whose morphism spaces are Floer cochain complexes generated, roughly speaking, by interior intersections and positive-time Reeb chords on the boundary. For this paper we take all Floer complexes to have coefficients in a field $\K$ of characteristic 2, though in general $\W(M)$ can be defined over $\Z$. The wrapped Fukaya category was introduced to satisfy mirror symmetry for open Calabi-Yau manifolds, where it has proven extremely successful and notoriously difficult to compute.

To enhace it, we introduce the notion of a \emph{stop}. Roughly speaking, a stop $\sigma$ in a Liouville domain $(M,\lambda)$ is a hypersurface with boundary of the boundary $\partial M$ such $(\sigma,\lambda|_\sigma)$ is itself a Liouville domain. For an appropriate presentation of $\W(M)$, the intersection number of a Reeb chord with $\sigma$ gives a filtration on the wrapped Floer cochain complexes, and we refer to the zero-filtered part as the \emph{partially wrapped} Floer cochain complex. The \emph{partially wrapped Fukaya category} of $(M,\lambda,\sigma)$, denoted $\W_\sigma(M)$, is then the unital, non-full $A_\infty$ subcategory of $\W(M)$ consisting of those objects which avoid $\sigma$ and those morphisms which have intersection number zero with $\sigma$. 

The main result of this paper, Theorem \ref{thm:quotient}, characterizes the inclusion functor $\W_\sigma(M)\to\W(M)$. Here, the key technical assumption is that a stop, treated as a Liouville domain, is \emph{strongly nondegenerate}. For a Liouville domain $F$, this means that the following two conditions hold.
\begin{enumerate}
 \item For each connected component $F_i\subset F$, the symplectic cohomology $SH^*(F_i)$ is not zero.
 \item There is an element $e\in HH_*(\mathcal W(F))$ of action zero such that $\mathcal{OC}(e)$ is the unit $\mathbf{1}\in SH^*(F)$, where 
 \[
  \mathcal{OC}\colon HH_*(\mathcal W(F))\to SH^{*+n}(F)
 \]
 is the open-closed map. Here, the action filtration on Hochschild homology is induced from the action filtration on Floer cochain groups; for a precise definition, see Section \ref{sec:Hochschild action}.
\end{enumerate}
In particular, any punctured Riemann surface other than $\C$ and any cotangent bundle is strongly nondegenerate.

\begin{rmk}
 A nondegenerate Liouville domain $F$ is one for which some element $e\in HH_*(\mathcal W(F))$ satisfies $\mathcal{OC}(e)=\mathbf{1}$. This condition is Abouzaid's generation criterion \cite{Abou_gcgfc}, which says that such an $F$ admits a finite collection of Lagrangians which split-generate $\mathcal W(F)$. Ganatra further proved that in this case, $\mathcal{OC}$ is an isomorphism \cite{Gana_scdwc}, which in particular implies that $e$ is unique.
\end{rmk}

An approximate version of Theorem \ref{thm:quotient} can be stated as follows:

\begin{thm}\label{thm:quotient-intro}
 Let $M$ be a Liouville domain, and let $\sigma$ be a strongly nondegenerate stop in $M$. Let $\mathcal B\subset\W_\sigma(M)$ be the full subcategory of objects supported near $\sigma$. Then the inclusion $\W_\sigma(M)\to\W(M)$ induces a fully faithful functor
 \[
  \W_\sigma(M)\big/\mathcal B\to\W(M),
 \]
 where the quotient is a quotient of an $A_\infty$ category by a full subcategory in the sense of Lyubashenko-Ovsienko \cite{Lyub-Ovsi}.
\end{thm}

At the most basic level, Theorem \ref{thm:quotient-intro} shows that the partially wrapped Fukaya category, together with $\mathcal B$, really knows at least as much as the wrapped Fukaya category. That is, restricting to the zero-filtered part but remembering the stop doesn't lose information. At an intuitive level, this happens because the ``nice'' presentation of $\W(M)$ giving rise to the stop filtration uses a contact form with a large number of canceling Reeb chords. These chords live at different levels in the filtration, so passing to the zero-filtered part results in them no longer canceling. In a more minimal presentation, these chords could be eliminated geometrically.

The key ingredient of the proof of Theorem \ref{thm:quotient-intro} is an auxiliary filtration on $\mathcal W(M)$, presented as the trivial quotient $\mathcal A=\mathcal W(M)\big/\mathcal B$. The benefit of this quotient presentation is that it naturally contains the category $\mathcal A_0=\mathcal W_\sigma(M)\big/\mathcal B(\sigma)$ as the minimally filtered part, which makes it possible to build a homotopy which retracts $\mathcal A$ onto $\mathcal A_0$. The homotopy itself requires a filtered version of the annulus trick, which was introduced in \cite{Abou_gcgfc} and extended in \cite{Gana_scdwc} and \cite{Abou-Gana}. Specifically, one factors the identity operation as a composition of a product and a coproduct, where the coproduct is required to have one component land in the partially wrapped complex.

\begin{rmk}
  The condition that $\sigma$ is strongly nondegenerate, rather than just nondegenerate in the usual sense, is used in two ways. First, the zero-action condition allows us to confine certain holomorphic curves in $M$ to a neighborhood of $\sigma$ without leaving the world of contact-type Floer data (Lemma \ref{lem:no escape with small action}). Second, the condition that $SH^*(\sigma)\ne0$ implies that zero is not a representative of $\mathbf{1}\in SH^*(\sigma)$, so that the sum of the Morse minima is the unique low-action representative of the unit. Together with the action condition, this implies that $\mathcal{OC}(e)$ is the sum of the Morse minima at chain level. This allows us to avoid a large-energy homotopy which would fall outside the scope of Lemma \ref{lem:no escape with small action}. In the presence of a better confinement lemma, such as Lemma 4.11 of \cite{GPS_gen}, one could probably weaken the condition on $\sigma$ from strong nondegeneracy to ordinary nondegeneracy.
  
  On the other hand, it is unlikely that the condition that $\sigma$ is nondegenerate could be completely eliminated. That said, very little is currently known about degenerate Liouville domains, so producing a counterexample would be difficult at the moment.
\end{rmk}

\subsection{Discussion}

\begin{figure}
 \def\svgwidth{4cm}
 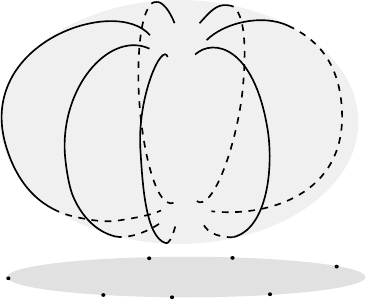
 \caption{The prototypical pumpkin domain. $\partial M$ is equipped with a finite collection of pages from an open book.}\label{fig:pumpkin}
\end{figure}

\subsubsection{Multiple stops} In fact, we consider not just a single stop $\sigma$ but a collection $\ssigma$ of disjoint stops. We refer to the triple $(M,\lambda,\ssigma)$ as a \emph{pumpkin domain} (see Figure \ref{fig:pumpkin}), and assign to it as above its partially wrapped Fukaya category $\Ws(M)$. This is the subcategory of $\W(M)$ consisting of objects which avoid all the stops in $\ssigma$, and whose morphism complexes are generated by chords and intersections which avoid all the stops in $\ssigma$. While this is the same category as $\W_{\bar\sigma}(M)$, where
\[
 \bar\sigma=\bigcup_{\sigma\in\ssigma}\sigma,
\]
the analog of Theorem \ref{thm:quotient-intro} is more refined. In this case, again assuming that $\sigma\in\ssigma$ is strongly nondegenerate, it states that the canonical functor
\[
 \Ws(M)\big/\mathcal B\to\W_{\ssigma\setminus\{\sigma\}}(M)
\]
is fully faithful.

\subsubsection{Mirror symmetry}\label{sec:mirror symmetry} Stops are meant to be a symplectic way of encoding the mirror of an anticanonical divisor. To see how, recall that there are two constructions of the mirror of a toric Fano variety $X$. The first \cite{Kont_98 notes,Hori-Vafa} produces a Landau-Ginzburg model $((\C^*)^n,W)$, where $W$ is a Laurent polynomial depending on the fan of $X$. This comes with a Fukaya-Seidel category, defined in \cite{Seid_plt} whenever $W$ is a Lefschetz fibration. Mirror symmetry in various forms has been proved via this approach in \cite{Give_mirthm}, \cite{Abou_hmstv}, and many others.

The second approach is less direct and assigns to $X$ a singular Lagrangian skeleton $\Lambda\subset(\C^*)^n$. The pair $((\C^*)^n,\Lambda)$ then has two flavors of Fukaya category, called partially wrapped and infinitesimally wrapped, which are meant to be equivalent to the dg-categories $Coh(X)$ and $Perf(X)$, respectively. The equivalence proceeds by combining the coherent-constructible correspondence \cite{FLTZ_Morelli} and the Nadler--Zaslow correspondence \cite{Nadl-Zasl}. This approach is explored by Fang--Liu--Treumann--Zaslow in \cite{FLTZ_hmstv}.

In good cases, each of these mirror objects gives rise to a stop in $M=(\C^*)^n$. Indeed, if $W\colon M\to\C$ is a superpotential with a compact set of critical values, then the fiber $W^{-1}(p)$ for $|p|\gg1$ projects along the Liouville vector field to a stop $\sigma_W$. Similarly, if $M$ is equipped with a singular Lagrangian skeleton $\Lambda\subset M$ with $\partial\Lambda$ a smooth Legendrian, then we may apply the Legendrian neighborhood theorem to thicken $\partial\Lambda\subset\partial M$ in the contact directions and obtain a stop $\sigma_\Lambda$. In the situations where $W$ and $\Lambda$ are mirror to smooth toric Fano varieties, $\partial\Lambda$ is meant to be a Lagrangian skeleton for the generic fiber of $W$, so the stops $\sigma_W$ and $\sigma_\Lambda$ are isotopic.

Going further, in the Landau-Ginzburg picture, Theorem \ref{thm:quotient-intro} can be thought of as a characterization of the acceleration functor
\[
 A\colon\Fuk(M,W)\to\Fuk(M,0).
\]
This characterization is dual to that of \cite{Abou-Seid_lft} and can be thought of as extending Abouzaid and Seidel's result to more general Landau-Ginzburg models. In fact, one could dream of a situation in which the theory of pumpkin domains has been extended to intersecting stops. In this case, a theorem analogous to Theorem \ref{thm:quotient-intro} would give a strong refinement of the acceleration functor:
\begin{conj}\label{conj:acceleration}
 Suppose $W=\sum_{i=1}^dW_i$ is a sum of monomials. Then $\sigma_{W_i}$ and $\sigma_{W_j}$ are generically expected to intersect. However, if partially wrapped Fukaya categories are developed for intersecting stops, one expects
 \[
  \mathcal W_{\sigma_W}(M)\cong\mathcal W_{\{\sigma_{W_1},\dotsc,\sigma_{W_d}\}}(M).
 \]
 In this case, deleting a stop corresponds to deleting a monomial from $W$. For $(M,W)$ mirror to a toric variety $X$, this in turn corresponds to deleting a toric divisor from $X$.
\end{conj}

\subsubsection{Other frameworks} When $M$ is equipped with a Lefschetz fibration $W$, one expects that
\[
Tw^\pi\mathrm{FS}(W)\cong Tw^\pi\W_{\sigma_W}(M),
\]
where $\mathrm{FS}(W)$ is the Fukaya-Seidel category of the Lefschetz Fibration. The machinery of Fukaya-Seidel categories can be extended at least partially to more general Landau-Ginzburg models \cite{Abou-Auro,Abou-Gana}, in which case one still expects the above to hold when $W$ has a compact set of critical points. When the critical locus is noncompact, then one can produce two Fukaya-Seidel type categories which differ on whether fiberwise noncompact Lagrangians are allowed as objects. The partially wrapped Fukaya category is meant to be the one in which they are.

When $M$ is instead equipped with a Lagrangian skeleton $\Lambda$, Nadler defines a category of``wrapped microlocal sheaves'' \cite{Nadl_wmls} as the full subcategory of compact objects a large category of microlocal sheaves. It is expected that when $\partial\Lambda$ is the skeleton for a stop $\sigma$, then Nadler's category will agree with $\W_\sigma(M)$.

\subsubsection{Generation} The primary application of Theorem \ref{thm:quotient-intro} is to give generation statements for partially wrapped Fukaya categories:
\begin{cor}\label{cor:generation}
 Under the assumptions of Theorem \ref{thm:quotient-intro}, suppose $\mathcal A_0\subset\W(M)$ and $\mathcal I\subset\mathcal B$ are split-generating full subcategories. Let $\mathcal A\subset\W_\sigma(M)$ be a collection of objects such that $A(\mathcal A)=\mathcal A_0$, where $A\colon\W_\sigma(M)\to\W(M)$ is the inclusion. Then the full subcategory $\mathcal A\cup\mathcal I$ of $\W_\sigma(M)$ split-generates.
\end{cor}

In fact, in future work we will show that the subcategory $\mathcal B$ is the image of an $A_\infty$ functor $\imath_\sigma\colon\W(\sigma)\to\W_\sigma(M)$, which is the partially wrapped version of Orlov's functor
\[
 \Fuk(F)\to\mathrm{FS}(W)
\]
for a Lefschetz fibration $W$ with smooth fiber $F$. In the mirror picture, this corresponds to the pushforward functor from sheaves on a divisor to sheaves on the total space. The upshot is that a split-generating subcategory $\mathcal I$ can be found as the image under $\imath_\sigma$ of a split-generating subcategory $\tild{\mathcal I}\subset\W(\sigma)$. This reduces the problem of split-generation of partially wrapped Fukaya categories with strongly nondegenerate stops to that of fully wrapped Fukaya categories, for which \cite{GPS_gen} gives a good answer.

Concretely, consider the Landau-Ginzburg model $(M,W)=\left(\C^3,xyz\right)$, which is mirror to the pair of pants. Recall from Section \ref{sec:mirror symmetry} that this gives rise to a stop $\sigma_W$ which is symplectomorphic to the generic fiber of $W$. In this case the generic fiber is $(\C^*)^2$, whose wrapped Fukaya category is generated by the single Lagrangian $\tild{L}=(\R_+)^2$. Because $\mathcal B$ is the image of $\imath_{\sigma_W}$, we can replace $\mathcal B$ with a single Lagrangian $L=\imath_{\sigma_W}(\tild{L})$, which geometrically is the parallel transport of $\tild L$ over the arc that curves around $\sigma_W$. This expresses the trivial category $\mathcal W(\C^3)$ as a quotient of $\mathcal W_\sigma(M)$ by $L$, which means $L$ split-generates $\mathcal W_\sigma(M)$. This result is predicted by mirror symmetry in \cite{Abou-Auro}, where Abouzaid and Auroux compute the endomorphism algebra of $L$ in a ``fiberwise wrapped'' Fukaya category which is expected to agree with ours.

\subsection{Outline of the paper}

In Section \ref{ch:geom setup}, we define stops and pumpkin domains in terms of local models. We prove Proposition \ref{prop:stops exist}, which justifies that definition, and we use it to give basic examples of pumpkin domains. We then describe how to glue pumpkin domains along stops. This is used in a basic way to control the geometry in Sections \ref{ch:nondeg stops} and \ref{ch:stop removal}.

In Section \ref{ch:PWFC}, we define partially wrapped Fukaya categories. We then state some invariance properties in Proposition \ref{prop:pwfc properties}, which are proven in the expanded version \cite{Sylv_pwfc-big}.

Section \ref{ch:nondeg stops} begins by constructing the action filtration on Hochschild chains which goes into the definition of a strongly nondegenerate stop. This allows us to state the precise version of Theorem \ref{thm:quotient-intro}. We then introduce a chain-level nondegeneracy condition for the pair $(M,\sigma)$ and prove that it holds whenever $\sigma$ is strongly nondegenerate (Proposition \ref{prop:nondegeneracy}). This is the key technical ingredient we use to pass from data in the stop to data on the ambient pumpkin domain.

Finally, in Section \ref{ch:stop removal}, we prove the main theorem. We begin by constructing the coproduct operation. Then we construct a sequence of smaller homotopies which interpolate between the composition of product with coproduct and a projection to the partially wrapped part. The key observation here is that every time a long $X_H$-chord intersects a stop $\sigma$, it does so by first entering a neighborhood of $\sigma$, then intersecting $\sigma$, and then leaving the neighborhood. Thus, by carefully choosing incidence conditions with the boundary of a neighborhood of $\sigma$, we construct in Section \ref{sec:last htopy} an operation which looks like the identity but is homotopic to zero.

\subsection*{Acknowledgments}
Above all, I'd like to thank my advisor Denis Auroux for his support, guidance, and patience. A substantial fraction of the good ideas in this paper are his, and it would be a sad husk of a text without them. I am also grateful to Mohammed Abouzaid, Sheel Ganatra, David Nadler, and Katrin Wehrheim for their feedback and for many valuable discussions. Finally, the idea of Section \ref{ch:nondeg stops} sprang from discussions with Paul Biran and Dietmar Salamon, whom I would like to thank as well.

This work was partially supported by NSF grants DMS-0838703, DMS-1007177, and DMS-1264662.

\section{Geometric setup}\label{ch:geom setup}

\subsection{Liouville domains}

Our basic objects of study will be Liouville domains $(M,\lambda_M)$, which are compact manifolds with boundary such that $\omega_M:=d\lambda_M$ is symplectic, and such that the Liouville vector field $Z_M$ defined by $\imath_{Z_M}\omega_M=\lambda_M$ points outward along the boundary. This implies that $\alpha=\lambda_M|_{\partial M}$ is a contact form, and flowing along $-Z_M$ gives a collar
\[
(U,\lambda_M)\cong((0,1]\times\partial M,r\alpha).
\]
Attaching the rest of the symplectization of $\partial M$ gives the \textbf{completion} $\hat M:=M\cup_{\partial M}[1,\infty)\times\partial M$, which comes with a natural 1-form $\hat\lambda_M$, symplectic form $\hat\omega_M$, and Liouville vector field $\hat Z_M$.

A good class of mappings between Liouville domains $F$ and $M$ is that of \textbf{Liouville maps}, which are proper embeddings $\phi\colon\hat F\hookrightarrow\hat M$ such that
\begin{equation}
\label{eq:liouville map}
\begin{aligned}
  &\phi^*\hat\lambda_M=\hat\lambda_F+df \quad &&\text{for some compactly supported $f$, and}\\
  &\phi_*\hat Z_F=\hat Z_M &&\text{away from a compact set.}
\end{aligned}
\end{equation}
Note that the second condition is redundant for codimension zero maps. In general, it can be rephrased as saying that the symplectic orthogonal of the image of $\phi$ lies in the kernel of $\hat \lambda_M$. A Liouville isomorphism, then, is just a Liouville map that is a diffeomorphism. Note that we are viewing these as maps of Liouville domains, even though geometrically they are only defined at the level of completions.

A version of Moser's lemma holds in this setting \cite{Ciel-Elia_swb}:

\begin{lemma}\label{lem:moser}
Let $(M,\lambda_M^t)$ be a smooth family of Liouville domains parametrized by $t\in[0,1]$. Suppose $K\subset M$ is a codimension zero closed subset which near $\partial M$ is invariant under the Liouville vector field for all $t$, and that $\lambda_M^t$ is $t$-independent on $K$. Then there is a family of Liouville isomorphisms $\phi_t\colon(M,\lambda_M^0)\to(M,\lambda_M^t)$ which is the identity on $K$.
\qed
\end{lemma}
\begin{cor}\label{cor:moser}
Let $(F,\lambda_F^t)$ and $(M,\lambda_M^t)$ be smooth families of Liouville domains for $t\in[0,1]$. Suppose there exists a Liouville map $\phi\colon(F,\lambda_F^0)\to(M,\lambda_M^0)$. Then $\phi$ extends to an isotopy of Liouville maps $\phi^t\colon(F,\lambda_F^t)\to(M,\lambda_M^t)$.
\qed
\end{cor}

Occasionally we will use the stronger notion of an \textbf{isomorphism of exact symplectic manifolds}, which is a diffeomorphism $\phi\colon M\to M'$ of exact symplectic manifolds (not necessarily Liouville domains or their completions) satisfying $\phi^*\lambda_{M'}=\lambda_M$.

Given two Liouville domains $M$ and $M'$, one can attempt to form their product. The result is an exact symplectic manifold with corners. One can non-canonically round the corners to obtain a Liouville domain. The result completes to $(\hat M\times\hat M',\hat \lambda_M+\hat\lambda_{M'})$, so the product is at least well-defined up to completion, and hence up to isomorphism. We'll use ${M\times M'}$ to denote the resulting Liouville domain for any choice of boundary.

If we additionally have a Liouville map $\phi\colon F\to M$, then the product
\[
 \phi\times\mathrm{id}_{M'}\colon F\times M'\to M\times M'
\]
is not quite a Liouville map, but it becomes one if we replace $\lambda_F$ by $\phi^*\lambda_M$. By definition, this doesn't change the Liouville isomorphism class of $F$, and in the sequel we'll often make such compactly supported changes implicitly when talking about products.

\subsection{Stops}

Symplectic manifolds often come with additional data, such as a global meromorphic function or a distinguished collection of Lagrangians. For Floer theoretic purposes, this data can often be encoded as a set of framed complex hypersurfaces.

\begin{defn}\label{defn:stop}
Let $(M^{2n},\lambda_M)$ and $(F^{2n-2},\lambda_D)$ be Liouville domains. For $\rho>0$, denote by $\H_\rho$ the set $\{z\in\C\mid\Re(z)\ge-\rho\}$ with the standard exact symplectic structure coming from $\C$, i.e. $\lambda_\C=\frac12\left(xdy-ydx\right)$. A \textbf{stop} of width $\rho$ in $M$ with fiber $F$ is a proper embedding $\sigma\colon\hat F\times\H_\rho\to\hat M$ satisfying
\[
  \sigma^*\hat\lambda_M=\hat\lambda_F+\lambda_{\H_\rho}+df
\]
for some compactly supported $f$. If $\sigma$ is a stop, then $D_\sigma:=\sigma|_{\hat F\times\{0\}}$ is a Liouville map, which we'll call its \textbf{divisor}. In what follows, we'll often identify $D_\sigma$ with its image.
\end{defn}

The requirement that a stop be a proper map is important. It means that all of the data lives on the boundary, which will be needed to obtain well behaved gluing operations. The notion of width, on the other hand, is just a notational convenience. Specifically, if $\rho'=t\rho$, then $\H_\rho$ and $\H_{\rho'}$ are isomorphic as exact symplectic manifolds via 
\[
(x,y)\mapsto(tx,t^{-1}y).
\]
We will also sometimes wish to narrow a stop, that is to embed $\H_\rho$ into some enlarged angular sector
\[
S_{\rho,s}=\bar D^2_\rho\cup\left\{re^{i\theta}\in\C\mid r>0, |\theta|\le s\right\}
\]
with $\rho>0$ and $s\in(0,\frac\pi2)$. While this can't be done in a way that strictly preserves the Liouville form, it can be done in a way that only modifies the Liouville form in some small annulus around zero and fixes the positive real axis. For this, one can take the large time flow of a Hamiltonian which, outside of the annulus, takes the form $r^2\sin\theta$. Crossing with the fiber, this might cause the Liouville vector field to fail to point outward along $\partial F\times D^2_\rho$. However, since the modification to $\lambda_\C$ is bounded, we replace $F$ with a larger piece of $\hat F$ so that $\lambda_\C$ is small compared to $\lambda_F$, and hence outward pointingness will be preserved at this new boundary. This shows

\begin{lemma}\label{lem:narrowing}
Let $(M,\lambda_M)$ be a Liouville domain and $\sigma_0\colon\hat F\times S_{\rho,s}\to\hat M$ be a proper codimension zero embedding with
\[
  \sigma_0^*\hat\lambda_M=\hat\lambda_F+\lambda_{S_{\rho,s}}+df
\]
for some compactly supported $f$. Then there is a new Liouville form $\hat\lambda_M'=\hat\lambda_M+dg$, where $g$ is supported in a small tube around $\sigma_0(\hat F\times\{0\})$, such that as a map into $(\hat M,\hat\lambda_M')$, $\sigma_0|_{(\hat F\times\{0\})}$ extends to a stop $\sigma$ with $\sigma(\hat F\times\R_+)=\sigma_0(\hat F\times\R_+)$.
\qed
\end{lemma}

\begin{defn}\label{defn:narrow stop}
A map satisfying the properties of $\sigma_0$ above will be called a \textbf{narrow stop}.
\end{defn}

A stop also constrains the behavior of the Liouville form near its divisor. This too will be needed for gluing, though it is not hard to modify a given Liouville map to look like the divisor of a stop. In fact, we have the following:

\begin{prop}\label{prop:stops exist}
Let $(M^{2n},\lambda_M)$ be a Liouville domain, and let $P\subset\partial M$ be a compact hypersurface with boundary such that $(P,\lambda_M|_P)$ is a Liouville domain. Choose $f\colon P\to[\frac12,1]$ to be a continuous function such that
\begin{enumerate}
\item $f$ is smooth and less than 1 on the interior of $P$.
\item $f|_{\partial P}=1$.
\item\label{cd:contact collar for new stops} $f^{-1}(r)$ is transversely cut out and contact for $r>\frac12$.
\item\label{cd:Z invariance for new stops} $F=\mathrm{graph}(f)\subset M$ is a smooth submanifold that is parallel to $Z$ to infinite order along its boundary. See Figure \ref{fig:page to divisor}.
\end{enumerate}
Then $(F,\lambda_M|_F)$ is a Liouville domain, and its inclusion into $M$ extends to a Liouville map $\phi$. Moreover, one can construct a new Liouville form $\lambda_M'=\lambda_M+dh$ such that, after moving $\partial M$ out, $\phi$ becomes the divisor of a stop in $(M,\lambda_M')$ with fiber $F$.
\end{prop}

\begin{figure}
 \def\svgwidth{9cm}
 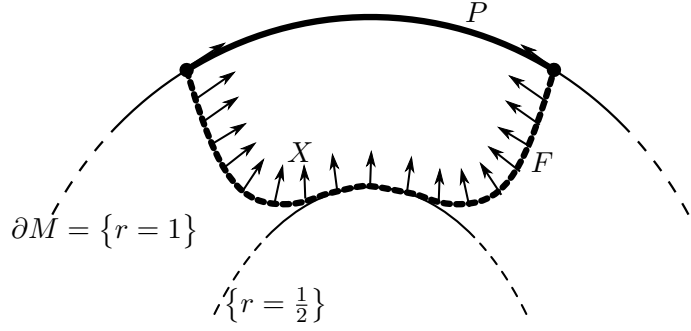
 \caption{The divisor $F$ and its framing vector field $X$.}\label{fig:page to divisor}
\end{figure}

Note that $h$ was not required to vanish in a neighborhood of $\partial M$. Of course, by Corollary \ref{cor:moser}, one can arrange that it does, at the expense of an isotopy of $F$.

\begin{proof}
To see that $F$ is a Liouville domain, it suffices to show that $\omega_M|_F$ is symplectic, since then outward-pointingness is clear from condition \ref{cd:Z invariance for new stops}. In fact, this is automatic near the boundary, since there $F$ is close to the symplectization of the contact manifold $\partial P$. Thus, we can consider only what happens away from the boundary, which allows us to transport the question to $P$. For this, let $\tild f\colon P\to M$ be the graph map $p\mapsto(f,p)$, so that we are interested in whether $\tild f^*\omega_M$ is symplectic on the interior of $P$. Then $\tild f^*\lambda_M=f\lambda_M|_P$, so that $\tild f^*\omega_M=df\wedge\lambda_M|_P+f\omega_M|_P$. We compute
\[
 (\tild f^*\omega_M)^{n-1}=(\omega_M^{n-1})|_P+df\wedge\lambda_M|_P\wedge(\omega_M^{n-2})|_P.
\]
The first term is positive because $(P,\lambda_M|_P)$ is a Liouville domain, while the second term is nonnegative by condition \ref{cd:contact collar for new stops}. This implies that $F$ is a Liouville domain, and it follows from the definitions that $\phi$ is a Liouville map.

Our next step then is to exhibit $F$ locally as the divisor of a stop. To do this, we will use Moser's argument to modify the Liouville form on $M$ in a neighborhood of $F$. For that to be effective, we will want to frame $F$ so that, when we try to extend $\phi$ to a stop, it will know which way points out.

Choose a nonvanishing vector field $X\in\Gamma(T\hat M|_{\hat F})$ that is symplectically orthogonal to $\hat F$ and, in the symplectization coordinates $(r,p)$ on $(0,\infty)\times\partial M$, is of the form $X=(g\frac\partial{\partial r},X_\partial)$, where $g\ge0$ and $X_\partial$ is tangent to $P$. The choice is unique up to scaling by a positive function. Next, pick a second vector field $Y\in\Gamma(T\hat M|_{\hat F})$, also orthogonal to $\hat F$, such that the radial component of $Y$ vanishes identically and $\omega_M(X,Y)=1$. By the symplectic neighborhood theorem on a compact part of $\hat M$, we can find a number $\rho>0$ and a symplectic embedding $\psi\colon F_2\times D^2_\rho\to\hat M$, where $F_2$ is the part of $\hat F$ with $r\le2$, such that
\begin{enumerate}[(i)]
  \item $\psi|_{F_2\times\{0\}}=\phi$\\
  \item $\psi_*\frac\partial{\partial x}=X$ along $F_2$ \label{cd:framing 1}\\
  \item $\psi_*\frac\partial{\partial y}=Y$ along $F_2$. \label{cd:framing 2}
\end{enumerate}
where $x=\Re(z)$ and $y=\Im(z)$ are the coordinates on $D^2_\rho$.

It is time to change $\lambda$. Let $\theta=\lambda_{F_2\times D^2_\rho}-\psi^*\hat\lambda_M$. Then $\theta$ is closed and $\theta|_{F_2}=0$, so we can find a primitive $h_0$ of $\theta$ on a neighborhood of $F_2$ with $h_0|_{F_2}=0$. Shrinking $\rho$, we can assume that $h_0$ is defined on all of $F_2\times D^2_\rho$. Consider a family of cutoff functions $\kappa_t\colon F_2\times D^2_\rho\to[0,1]$ indexed by $t\in(0,\rho)$ and satisfying the following conditions:
\begin{enumerate}[(i)]
\setcounter{enumi}{3}
  \item $\kappa_t$ is independent of the $F$ component and is rotationally invariant and radially nonincreasing in the $D^2$ component when $r\le1$ \label{cd:cutoff isotropy a}\\
  \item $\kappa_t=1$ when $|z|\le\frac t3$ and $r\le1$ \label{cd:cutoff isotropy b}\\
  \item $\kappa_t=0$ when $|z|\ge\frac{2t}3$ or $r\ge\frac32$ \label{cd:cutoff support}\\
  \item $|d\kappa_t|<\frac4t$ with respect to some fixed $t$-independent product metric on $F_2\times D^2_\rho$ which is Euclidean on the $D^2_\rho$ factor \label{cd:cutoff niceness}\\
  \item $\kappa_{t_0}(p,t_0z)=\kappa_{t_1}(p,t_1z)$ for all $(p,z)\in F_2\times D^2$ and $t_i\in(0,\rho)$. \label{cd:cutoff scaling}
\end{enumerate}
We can rephrase this last condition as saying that shrinking $t$ corresponds to conjugation by a rescaling of the $D^2$ component.

We will see that the function $h$ in the statement of the lemma can be taken to be $\psi_*(\kappa_th_0)$ for sufficiently small $t$. For now, let us denote that function by $h_t$. The first thing to notice is that for $t$ sufficiently small, the Liouville vector field $Z_M^t$ associated to $\lambda_M^t=\lambda_M+dh_t$ points out along the boundary of $M$, so that $(M,\lambda_M^t)$ is a Liouville domain. To see this, note that $\lambda_M$ vanishes on the symplectic orthogonal to $[1,2]\times\partial F$, where $[1,2]\subset(0,\infty)$ is the symplectization component, so $\theta$ does as well. Thus, $h_0$ vanishes quadratically on $[1,2]\times\partial F$. This, combined with conditions \eqref{cd:cutoff support} and \eqref{cd:cutoff niceness}, implies that $dh_t$ has magnitude $O(t)$. Since the condition that $Z$ points outward is open, this gives the desired conclusion.

It remains to find some $t$ for which $\phi$ extends to a stop in $(M,\lambda_M^t)$. By Lemma \ref{lem:narrowing}, it is enough to extend $\phi$ to a narrow stop. By \eqref{cd:cutoff isotropy b}, we can find the disk part of a narrow stop, so we need only find an angular sector over which we can finish extending $\phi$. The naive solution here is to just pick a small angular sector and flow out via $Z$, which works, but one needs to ensure that this doesn't get snagged on some interesting piece of $M$. This is where the framing of $\psi$ becomes important.

\begin{figure}
 \def\svgwidth{7cm}
 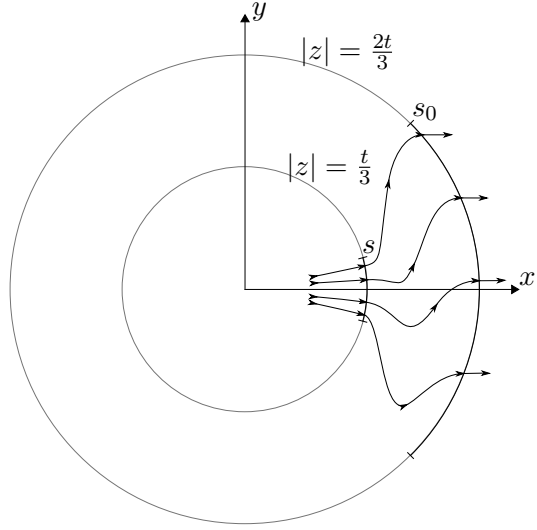
 \caption{One possibility for the image of the narrow stop.}\label{fig:finding the narrow stop}
\end{figure}

For convenience, let's now identify $\hat F\times D^2_\rho$ with its image under $\psi$. Let $\delta>0$ be such that if $r>1-\delta$, then for all sufficiently small $t$ the flow of $Z_M^t$ escapes to infinity. Because of \eqref{cd:framing 1}, for $t$ sufficiently small and $r\le1-\delta$, the original Liouville vector field $\hat Z_M$ has an $x$-component $(\hat Z_M)_x$ which is positive and bounded away from zero along $\hat F\times\{\frac{2t}3\}$. By \eqref{cd:cutoff support}, the same is true of $\hat Z_M^t$. By openness, we can find some angle $s_0\in(0,\frac\pi2)$ such that $\hat Z_M^t$ points out of $\hat F\times D^2_{\frac{2t}3}$ when $r\le1-\delta$ and the angular $D^2$-coordinate $\theta$ belongs to $[-s_0,s_0]$. Since $(\hat Z_M^t)_y$ vanishes on $F$, $s_0$ can be chosen to be independent of $t$. Indeed, as $t\to0$, $s_0$ could be taken to increase to $\frac\pi2$. We want to show that there exists some smaller $s$ such that $\hat Z_M^t$ takes
\[
\textstyle \hat F\times\left\{\frac t3e^{i\theta}\mid-s\le\theta\le s\right\}
\]
through
\[
\textstyle \left(\hat F\times\left\{\frac{2t}3e^{i\theta}\mid-s_0\le\theta\le s_0\right\}\right)\cup\big(\partial M\times(1-\delta,\infty)\big),
\]
as in Figure \ref{fig:finding the narrow stop}.

If we can, then we're done, since we know that the flow of $\hat Z_M^t$ starting anywhere in the latter set escapes to infinity. To accomplish this, it is enough to show that and  there exist positive constants $C$ and $D$, with $D$ small (strictly less than $\frac{c-1}{2^{c-1}-1}\tan s$), such that
\begin{equation}\label{cd:stays narrow}
\left|\frac{(\hat Z_M^t)_y}{(\hat Z_M^t)_x}\right|<C\cdot\left|\frac yx\right|+D\quad\text{ and }\quad(\hat Z_M^t)_x>0
\end{equation}
when $t$ is small, $\theta\in[-s_0,s_0]$, and $r\le1-\delta$. It turns out we can take $C=1+\varepsilon$ and $D=\varepsilon$ for arbitrary small $\varepsilon>0$. Specifically, remember that $Z_M^t$ is the vector field dual to
\[
\lambda_M^t=\kappa_t\cdot\left(\hat\lambda_F+\frac{xdy-ydx}2\right)+(1-\kappa_t)\lambda_M+h_0d\kappa_t.
\]
Note we've switched to the interior part of $M$, since that is where all our current problems live. Note also that $\left|(Z_M^t)_y\right|$ is given by $\left|\lambda_M^t(\frac\partial{\partial x})\right|$ and similarly with $x$ and $y$ switched. In the above formula, the $\hat\lambda_F$ term doesn't affect \eqref{cd:stays narrow}, so it can be ignored. Furthermore, to lowest order, the $h_0d\kappa_t$ term is strictly beneficial from the perspective of \eqref{cd:stays narrow}. To see this, note that there is a positive function $E\in C^\infty(F)$ such that $dh_0=-Edy+O(|z|)$ near $F\times\{0\}$, so that $h_0=-Ey+O(|z|^2)$. Thus, using conditions \eqref{cd:cutoff niceness}, we see that
\[
 h_0d\kappa_t=-E\frac{xy}{|z|}\frac{\partial\kappa_t}{\partial|z|}dx-E\frac{y^2}{|z|}\frac{\partial\kappa_t}{\partial|z|}dy+O(t).
\]
Now $\frac{\partial\kappa_t}{\partial|z|}$ is nonpositive by assumption, and $dx$ is dual to $-\frac\partial{\partial y}$, so the first term leads to a negative $y$-component when $y$ is positive and a positive $y$-component when $y$ is negative. This means that its contribution to $Z_M^t$ will never enlarge $|\theta|$. Similarly, the second term is dual to a nonnegative function times $\frac\partial{\partial x}$, so it too will never enlarge $|\theta|$. Hence, it suffices to show that \eqref{cd:stays narrow} can be satisfied with $h_0$ replaced by $\tild h_0=h_0+Ey$, i.e. after discarding the leading order term. 

For this, we will need to separately consider two pieces. Pick $w\in(\frac13,\frac23)$ to be such that $\kappa_t(p,wt)\ne1$ and, if $|z|<wt$, $\theta\in[-s_0,s_0]$, and $r\le1-\delta$, then
\begin{align}\label{cd:small z error 1}
\|\tild h_0d\kappa_t(p,z)\|&<\frac{\varepsilon\kappa_t}4x\\
\big(1-\kappa_t(p,wt)\big)\cdot\left\|d\left(\lambda_M\left({\textstyle \frac\partial{\partial x}}\right)\right)(p,0)\right\|&<\frac{\varepsilon\kappa_t}6.\label{cd:small z error 2}
\end{align}
When $|z|\ge wt$, there is a positive $t$-independent lower bound on $(1-\kappa_t)\lambda_M(\frac\partial{\partial y})$, whereas after replacing $h_0$ by $\tild h_0$ every other term of $\lambda_M^t(\frac\partial{\partial x})$ and $\lambda_M^t(\frac\partial{\partial y})$ tends uniformly to zero as $t\to0$. Thus, we can satisfy \eqref{cd:stays narrow} as long as we can reach $|z|=wt$. But when $|z|<wt$, we can again shrink $t$ so that \eqref{cd:small z error 2} implies
\begin{equation}\label{cd:small z error 3}
 (1-\kappa_t)|\lambda_M({\textstyle \frac\partial{\partial x}})|<\frac{\varepsilon\kappa_t}5x.
\end{equation}
But now, using the positivity of $\lambda_M(\frac\partial{\partial y})$, we have
\begin{align*}
 \left|\frac{(\hat Z_M^t)_y}{(\hat Z_M^t)_x}\right|&=\left|\frac{\frac12\kappa_ty-(1-\kappa_t)\lambda_M({\textstyle \frac\partial{\partial x}})-\tild h_0d\kappa_t({\textstyle \frac\partial{\partial x}})}{\frac12\kappa_tx+(1-\kappa_t)\lambda_M({\textstyle \frac\partial{\partial y}})+\tild h_0d\kappa_t({\textstyle \frac\partial{\partial y}})}\right|\\
 &\le\frac{|y|+\frac25\varepsilon x+\frac12\varepsilon x}{x-\frac12\varepsilon x}\\
 &<(1+\varepsilon)\cdot\left|\frac yx\right|+\varepsilon
\end{align*}
as desired.
\end{proof}

\begin{eg}\label{eg:boundary fibers}
Let $W\colon M\to\C$ be an exact Lefschetz fibration \cite{Seid_plt}, and let $\gamma\colon[0,\infty)\to\C$ be a properly embedded, asymptotically radial ray that avoids the critical values of $W$. Then we can modify the Liouville structure on $M$ in a neighborhood of $W^{-1}(\gamma)$ to obtain a stop modeled on $W^{-1}(\gamma(0))$.

More generally, we can do the above for any holomorphic fibration. This is the construction that we will use to define the partially wrapped Fukaya category of a Landau-Ginzburg model.
\end{eg}

\begin{eg}\label{eg:Legendrians}
Let $\Lambda\subset\partial M$ be a smooth closed Legendrian submanifold. By the Legendrian neighborhood theorem, $\Lambda$ has a tubular neighborhood $U\subset\partial M$ which is isomorphic as a contact manifold with 1-form to a convex neighborhood of the zero section in the 1-jet bundle $J^1\Lambda$. Identifying $U$ with that neighborhood, the submanifold
\[
 T^*\Lambda\cap U\subset U\subset M
\]
is then a Liouville hypersurface of $\partial M$, and applying Proposition \ref{prop:stops exist} gives a stop $\sigma_\Lambda$.
\end{eg}

\subsection{Pumpkin domains}

\begin{defn}\label{defn:ldds}
 From here on, the basic geometric object we will deal with is a \textbf{Liouville domain with disjoint stops}, or \textbf{pumpkin domain} for short. This is a triple $(M,\lambda_M,\ssigma)$ with $(M,\lambda_M)$ a Liouville domain and $\ssigma=\{\sigma_1,\dotsc,\sigma_k\}$ a collection of stops in $M$ such that the images of $\sigma_i$ and $\sigma_j$ are disjoint for $i\ne j$. For technical simplicity, we make the additional assumption that every stop $\sigma_i$ strictly preserves the Liouville form on $\hat M\setminus M$. Since a pumpkin domain only has finitely many stops, this can always be achieved by moving $\partial M$ out.

An \textbf{equivalence} of pumpkin domains $(M,\lambda_M,\ssigma)$ and $(M',\lambda_{M'},\ssigma')$ is a homotopy of collections of disjoint stops $\ssigma^t=\{\sigma_1^t,\dotsc,\sigma_k^t\}$, where $\sigma_i^t$ is a stop in $M$ with fiber $(F_i,\lambda_{F_i}^t)$, together with a Liouville isomorphism $\psi\colon M\to M'$ such that $\ssigma^0=\ssigma$ and $\psi\circ\sigma_i^1=\sigma_i'$. Here, $M$ and $M'$ are required to have the same number of stops.
\end{defn}

We'll usually abuse notation and use $M$ to refer to the pumpkin domain $(M,\lambda_M,\ssigma)$.

\begin{eg}\label{eg:C_n}
Fix a positive integer $n$, and consider the map $u_n\colon\C\to\C$ given by 
\[
 u_n(z)=z^{n+1}-1.
\]
Then $u^*\lambda_\C$ is almost a Liouville form on $\C$, except that its derivative vanishes at the origin. Choose a cutoff function $\kappa\colon\R\to[0,1]$ with $\kappa(x)=1$ for $x\le\frac14$ and $\kappa(x)=0$ for $x\ge\frac12$. Next, choose $\epsilon>0$ such that $u^*\omega_\C+\epsilon d(\kappa(|z|)\lambda_\C)$ is symplectic. Let $\ssigma^n=\{\sigma^n_0,\dotsc,\sigma^n_n\}$ be the set of $u_n$-lifts of the inclusion $\H_{\frac12}\into\C$ ordered counterclockwise with $\sigma^n_0$ specified by $\sigma^n_0(0)=1$. Then this data describes a pumpkin domain
\[
\C_n:=(\C,u^*\lambda_\C+\epsilon\kappa(|z|)\lambda_\C,\ssigma^n).
\]
Its underlying Liouville domain is Liouville isomorphic to $\C$, and it has $n+1$ stops, all with fiber the point. Note that $i\R\subset\C_1$ is invariant under the flow of the Liouville vector field.
\end{eg}

\begin{defn}\label{defn:stab}
Let $M$ be a Liouville domain. Then the $stabilization$ of $M$ is the pumpkin domain
\[
\Sigma M=M\times\C_1.
\]
As a Liouville domain, $\Sigma M$ is just isomorphic to the product $M\times\C$. As for the stops, there are two of them, both with fiber $M$, and their divisors sit over $1$ and $-1$.
\end{defn}

\begin{rmk}
Though we will not deal with them, one is sometimes given manifolds with stops that intersect. In this situation, it is reasonable to ask that the stops are \textbf{orthogonal}: it should be the case that if $\sigma_i$ has fiber $F_i$ and width $\rho_i$, then there is a Liouville splitting
\[
 \mathrm{image}(\sigma_1)\cap\mathrm{image}(\sigma_2)=(D_{\sigma_1}\cap D_{\sigma_2})\times \H_{\rho_1}\times \H_{\rho_2}
\]
that induces the splittings given by each of the stops individually. This gives rise to a more natural setting of Liouville domains with stops, not necessarily disjoint, and here the fiber of a stop will again be a Liouville domain with stops. Well definedness is achieved by induction on dimension. This approach has the advantage of being closed under products; in particular it admits arbitrary stabilizations.
\end{rmk}

\subsection{Hamiltonians for stops}

Let $M$ be a pumpkin domain. To obtain an invariant Floer theory, we will need to find a class of Hamiltonians on $M$ that is well adapted to the pumpkin structure. To state a compatibility condition, we need a convention for Hamiltonian vector fields, which we set as $dH=-\imath_{X_H}\omega$.

\begin{defn}\label{defn:compatible H}
A \textbf{compatible Hamiltonian} on $(M,\lambda_M,\ssigma)$ is a function $H\in C^\infty(\hat M)$ such that
\begin{enumerate}
\item $H$ is strictly positive.
\item $dH(\hat Z_M)=2H$ outside of a compact set. In other words, $H$ is quadratic in the symplectization coordinate. \label{cd:quadratic H}
\item $X_H$ is tangent to $D_\sigma$ for each stop $\sigma\in\ssigma$.\label{cd:Hamiltonian tangency}
\item For each stop $\sigma\in\ssigma$, $d\theta(X_H)$ is nowhere negative on a neighborhood of $\sigma(\hat F\times\R_+)$. Here, $\theta$ is the angular coordinate on the right half plane.\label{cd:Hamiltonian positivity}
\end{enumerate}
In particular, this last condition says that any integral curve for $X_H$ has only positive intersections with $\sigma(\hat F\times\R_+)$.
\end{defn}

It's worth noting that the space of Hamiltonians compatible with a given pumpkin domain forms a convex cone. Additionally, if one thinks of a Liouville domain $M$ as a pumpkin domain with no stops, then a compatible Hamiltonian on $M$ is just a positive quadratic Hamiltonian, as usual.

\begin{lemma}\label{lem:compatible H exist}
Every pumpkin domain admits a compatible Hamiltonian.
\end{lemma}
\begin{proof}
We need to show that conditions \eqref{cd:Hamiltonian tangency} and \eqref{cd:Hamiltonian positivity} can be achieved. For this, let $(M,\lambda_M,\ssigma)$ be a pumpkin domain, and assume without loss of generality that $\ssigma$ has only one element $\sigma$ with fiber $F$ and width $\rho$. Fix a compatible Hamiltonian $g$ on $F$. We want to extend this compatibly to all of $\hat F\times \H_\rho$, since then we can just patch it into $M$. For that, choose a nondecreasing smooth function $a\colon\R_{\ge0}\to[0,1]$ with $a|_{[0,1]}=0$ and $a|_{[2,\infty]}=1$, and set $f(z)=|z|^4a(|z|^4)$ as a function on $\H_\rho$. Define $h\colon\R_{\ge0}\to\R_{\ge0}$ by $h(x)=\frac12a(x)\log x$. Our candidate Hamiltonian is given by
\begin{equation}\label{eq:sum H}
H_\text{cand}(p,z)=e^{2h(f(z))}g\bigl(\phi_F(-h(f(z)),p)\bigr)+e^{2h(g(p))}f\bigl(\phi_\C(-h(g(p)),z)\bigr)
\end{equation}
where $\phi_F(t,\cdot)$ and $\phi_\C(t,\cdot)$ are the time $t$ flows of the Liouville vector fields of $F$ and $\C$, respectively. One readily checks that $H_\text{cand}$ satisfies conditions (1)-(3), so we need only to find some condition under which it satisfies (4). Now, since $f$ is rotationally invariant, condition (4) is equivalent to the requirement that $\frac\partial{\partial x}H_\text{cand}\ge0$ for $z=x\in\R_+$. This clearly holds for the second term, and for the first we compute
\[
	\frac\partial{\partial x}e^{2h\circ f}g\bigl(\phi_F(-h(f(x)),p)\bigr)
	=e^{2h\circ f}(h\circ f)'(x)\cdot\bigl(2g-dg(\hat Z_F)\bigr).
\]
Since $h\circ f$ is nonnegative and nondecreasing, it is enough to require $dg(\hat Z_F)\le2g$ globally. This can be achieved by just making $g$ bigger on the interior of $F$.
\end{proof}

\begin{rmk}
We will usually have not one, but a family of compatible Hamiltonians parametrized by some space $\Sigma$. In this situation, we require that the compact set in condition \ref{cd:quadratic H} in Definition \eqref{defn:compatible H} can be chosen $\Sigma$-independently.
\end{rmk}

\subsection{Geometric gluing}

Let $(M,\lambda_M,\ssigma)$ and $(M',\lambda_{M'},\ssigma')$ be pumpkin domains. Let $F_i$ and $F_j'$ be the fibers of $\sigma_i$ and $\sigma_j'$, respectively. When $F_i$ and $F_j'$ are isomorphic, we would like to form a new pumpkin domain $M\glu{\sigma_i}{\sigma_j'}M'$. To do this, let us fix an isomorphism $\phi\colon F_i\to F_j'$. Replacing the 1-form $\lambda_{F_j'}$ by $\phi^*\lambda_{F_i}$, we can assume $\phi$ is an isomorphism of exact symplectic manifolds. Due to the noncompact $\H_\rho$ factor in the domain of $\sigma_j'$, this causes $\sigma_j'$ to cease being a stop. To fix that, modify $\hat\lambda_{M'}$ so that it agrees with the new $\hat\lambda_{F_j'}+\lambda_{\H_\rho}$ on $\sigma_j'(\hat F\times\H_\rho)$, where $\rho$ is half the width of $\sigma_j'$, and is unchanged outside the image of $\sigma_j'$. By choosing the modification to be exact and translation-invariant in the imaginary direction, we can ensure that it is globally bounded, so that the new 1-form is still outward-pointing near infinity. Reducing the width $\sigma_j'$ by half turns it back into an honest stop. At this point $\sigma_i$ and $\sigma_j'$ are stops with the {\em same} fiber $F$, and so we can make one last modification of $\lambda_M$ and $\lambda_{M'}$, this one compactly supported and exact, to assume that $\sigma_i$ and $\sigma_j'$ themselves strictly preserve the Liouville forms.

That done, we can write down the gluing. Pick a positive number $a$ that is smaller than the widths $\rho_i$ and $\rho_j'$ of $\sigma_i$ and $\sigma_j'$. With this data, we can define the underlying Liouville domain of $M\glu{\sigma_i}{\sigma_j'}M'$ as
\[
	\bigl(\hat M\setminus\sigma_i(\hat F\times\{\Re(z)\ge a\})\bigr)\amalg\bigl(\hat M'\setminus\sigma_j'(\hat F\times\{\Re(z)\ge a\})\bigr)\big/\sim
\]
where $\sim$ is the identification
\[
	\sigma_i(\hat F\times\{-a<\Re(z)<a\})=\sigma_j'(\hat F\times\{-a<\Re(z)<a\})
\]
via $(p,z)\mapsto(p,-z)$. The stops are just
\[
	\ssigma_{M\glu{\sigma_i}{\sigma_j'}M'}:=(\ssigma\setminus\{\sigma_i\})\amalg(\ssigma'\setminus\{\sigma_j'\}),
\]
which makes sense since the stops are disjoint.

Suppose now that we had a 1-parameter family of pumpkin domains $(M,\lambda_M,\ssigma^t)$, that is an equivalence between $(M,\lambda_M,\ssigma^0)$ and $(M,\lambda_M,\ssigma^1)$. Then the diffeomorphism type of $M\glu{\sigma_i^t}{\sigma_j'}M'$ is independent of $t$, and by Moser's lemma we get a family of Liouville isomorphisms
\[
 \Psi_t\colon M\glu{\sigma_i^0}{\sigma_j'}M'\to M\glu{\sigma_i^t}{\sigma_j'}M'.
\]
Pulling back the stops in $M\glu{\sigma_i^t}{\sigma_j'}M'$ via $\Psi_t$, we see that our homotopy of stops in $M$ results only in a homotopy of stops in the gluing. Repeating this on the $M'$ side, we obtain

\begin{lemma}\label{lem:geom gluing}
Gluing descends to an operation on equivalence classes of pumpkin domains, and at this level it depends only on the triple $(\sigma_i,\sigma_j',[\phi])$. Here, $[\phi]$ is the connected component that $\phi$ belongs to in the space of Liouville isomorphisms from $F_i$ to $F_j'$.
\qed
\end{lemma}

We will often want to study objects which live ``near'' a stop but do not intersect its divisor. To give this a precise meaning, we will often make use of the following construction.

\begin{defn}\label{defn:trivial gluing}
 Let $(M,\lambda_M,\ssigma)$ be a pumpkin domain, and let $\sigma\in\ssigma$ be a stop with fiber $F$. Then the \textbf{trivial gluing} at $\sigma$, written $M[\sigma]$, is the pumpkin domain $M\glu{\sigma}{\sigma_0}\Sigma F$.
\end{defn}

Trivial gluing effectively replaces $\sigma$ with $\sigma_1$ and doesn't change the pumpkin equivalence class of $M$. Indeed, it can be achieved by homotoping $\lambda_M$ in the class of Liouville forms and moving $\sigma$ out. The benefit of trivial gluing is that it gives rise to the globally $\hat Z_M$-invariant hypersurface $\hat F\times i\R$.

\section{Partially wrapped Fukaya categories}\label{ch:PWFC}

\subsection{Lagrangian Floer cohomology}

For convenience of notation, we'll assume everything in sight is graded. Specifically, we require that all of our Liouville domains satisfy $2c_1(M)=0$, and further that they come with a choice of fiberwise universal cover $\tild{LGr}(M)$ of their bundle of unoriented Lagrangian Grassmannians. Given two Liouville domains $M_1$ and $M_2$, their product is graded in the unique way that extends $\tild{LGr}(M_1)\times_\Z\tild{LGr}(M_2)$. All codimension zero symplectic embeddings will be assumed to preserve these covers.

\begin{defn}\label{defn:pumpkin Lag}
  Given a pumpkin domain $M$, a \textbf{Lagrangian} $L\subset M$ is an exact, properly embedded Lagrangian submanifold of $\hat M$ which is parallel to $\hat Z_M$ outside of a compact set. It is required to be graded in the standard sense, namely that it is equipped with a lift to $\tild{LGr}(M)$ of the natural section $L\to LGr(M)$. For compatibility with the pumpkin structure, we require that $L$ does not intersect any $\sigma_i(\hat F\times\R_{\ge0})$.
 
 An \textbf{interior Lagrangian} is a Lagrangian which completely avoids the images of the stops. It is easy to see that any Lagrangian is isotopic to an interior Lagrangian.
\end{defn}

Given a compatible Hamiltonian $H$ on $M$, we want to consider a class $\mathcal J(M,H)$ of almost complex structures which are adapted to $H$. An element $J\in\mathcal J(M,H)$ is a smooth almost complex structure on $\hat M$ which is compatible with $\hat\omega_M$ and satisfies the following three conditions. First, there is some $c>0$ such that
\begin{gather}\label{eq:compatible J}
	dH\circ J=-cH\hat\lambda_M
\end{gather}
outside of a compact set. Second, the restriction $J|_{\ker dH\cap\ker\hat\lambda_M}$, i.e. the contact portion of $J$, is asymptotically $\hat Z_M$-invariant: there is some compatible $J'$ which is $\hat Z_M$-invariant outside a compact set such that the associated metric $g_J$ approaches $g_{J'}$ in the $C^0$ norm defined by $g_{J'}$. Third, for each stop $\sigma\in\ssigma$, we require that the projection to $\H_\rho$ is holomorphic along $D_\sigma$. In other words, the divisor of each stop is required to be an almost complex submanifold, and the restriction of $J$ to its symplectic orthogonal coincides with multiplication by $i$ in the base.

\begin{lemma}\label{lem:Js exist}
For any pumpkin domain $(M,\lambda_M,\ssigma)$ and compatible Hamiltonian $H$, the space $\mathcal J(M,H)$ is contractible and non-empty.
\end{lemma}
\begin{proof}
  We prove only the last part, for which it is enough to construct such an almost complex structure near the divisor of a stop $\sigma$. Let $F$ be the fiber of $\sigma$, and pick an almost complex structure $J_F\in\mathcal J(F,H|_{\hat F})$. Since the symplectic orthogonal to $D_\sigma$ lies in the kernel of both $dH$ and $\lambda_M$ outside of a compact set, there is no obstruction to extending $J_F$ to $T\hat M|_{D_\sigma}$ while satisfying Equation \eqref{eq:compatible J}. Now just extend to the rest of $\hat M$.
\end{proof}

To endow $\mathcal J(M,H)$ with the structure of a complete metric space, one needs to fix the compact set for \eqref{eq:compatible J}. This prevents the existence of a sequence of almost complex structures which satisfies \eqref{eq:compatible J} only outside of ever larger compact sets, so that the limit satisfies it nowhere. To obtain transversality results, we will require that $H$ is quadratic and all Lagrangians are $\hat Z_M$-invariant outside the compact set. We choose the compact sets implicitly as part of the data of $H$, for example to equal $H^{-1}((-\infty,r+1])$, where $H^{-1}((-\infty,r])$ is the smallest sublevel set of $H$ outside of which it is strictly quadratic and the Lagrangians are strictly conical.

We will in fact need time-dependent, or more generally domain-dependent almost complex structures. For this, suppose $\Sigma$ is a smooth manifold, possibly with boundary or corners, and that we've chosen a $\Sigma$-parametrized family of compatible Hamiltonians $H$. Denote by $\mathcal J^\Sigma(M,H)$ the set of smooth maps $J\colon\Sigma\to\mathcal J(\hat M,\hat \omega_M)$ satisfying
\[
	J(z)\in\mathcal J(M,H(z))
\]
for all $z\in\Sigma$, and such that \eqref{eq:compatible J} holds pointwise outside of a $\Sigma$-independent compact subset of $\hat M$. Here, $\mathcal J(\hat M,\hat \omega_M)$ is the space of all $\hat\omega_M$-compatible almost complex structures. Likewise, for families of domain-dependent almost complex structures, we require that the compact set can be chosen uniformly for the family. In practice, we will choose the compact set implicitly to be a sublevel set for the family of Hamiltonians.

Now suppose $L_0$ and $L_1$ are Lagrangians in $M$, and $H$ is a compatible Hamiltonian. $H$ is called \textbf{nondegenerate} for the pair $(L_0,L_1)$ if $\phi(L_0)$ is transverse to $L_1$, where $\phi$ is the time 1 flow of $X_H$. If $H$ is nondegenerate, set $\mathscr X(L_0,L_1;H)$ to be the set of time 1 $X_H$-chords starting on $L_0$ and ending on $L_1$. Since everything was graded, chords $\gamma\in\mathscr X(L_0,L_1)$ are equipped with a degree $\deg(\gamma)$ given by topological intersection number with the Maslov cycle.

For $J\in\mathcal J^{[0,1]}(M,H)$, we consider maps
\[
	u\colon\R\times[0,1]\to\hat M
\]
mapping $\R\times\{0\}$ to $L_0$ and $\R\times\{1\}$ to $L_1$. For fixed $\gamma_+$ and $\gamma_-$ in $\mathscr X(L_0,L_1,H)$, let $\tild{\mathcal R}(\gamma_+;\gamma_-)$ be the collection of such maps satisfying Floer's equation
\begin{equation}\label{eq:Floer}
	\partial_s u+J(t)(\partial_t u-X_H)=0
\end{equation}
with $s$ and $t$ the coordinates on $\R$ and $[0,1]$, respectively, and such that $\lim_{s\to\pm\infty}u(s,\cdot)=\gamma_\pm$. The transversality arguments in \cite{Floe-Hofe-Sala} show that for generic $J\subset\mathcal J^{[0,1]}(M,H)$, $\tild{\mathcal R}(\gamma_+;\gamma_-)$ is transversely cut out for every $\gamma_\pm\in\mathscr X(L_0,L_1)$. In this case, it has dimension $\deg(\gamma_-)-\deg(\gamma_+)$, and the translation $\R$-action on $\tild{\mathcal R}(\gamma_+;\gamma_-)$ is free if and only if $\gamma_+\ne\gamma_-$.

Fixing a regular $J\subset\mathcal J^{[0,1]}(M,H)$, i.e. which achieves transversality, define $\mathcal R(\gamma_+,\gamma_-):=\tild{\mathcal R}(\gamma_+;\gamma_-)/\R$ whenever $\deg(\gamma_-)>\deg(\gamma_+)$. This is again a smooth manifold, and it has a Gromov-Floer compactification $\ol{\mathcal R}(\gamma_+;\gamma_-)$ given by adding broken Floer trajectories \cite{Floe_mtli}.

In particular, if $\deg(\gamma_-)-\deg(\gamma_+)=1$, then $\ol{\mathcal R}(\gamma_+;\gamma_-)=\mathcal R(\gamma_+;\gamma_-)$ is a finite set, and if $\deg(\gamma_-)-\deg(\gamma_+)=2$, then $\ol{\mathcal R}(\gamma_+;\gamma_-)$ is a compact 1-manifold with boundary, where
\[
 \partial\ol{\mathcal R}(\gamma_+;\gamma_-)=\hspace{-.7cm}\coprod_{\substack{
   \tild\gamma\in\mathscr X(L_0,L_1)\\ 
   \deg(\tild\gamma)=\deg(\gamma_+)+1}}\hspace{-.8cm}
  \bigl(\mathcal R(\gamma_+;\tild\gamma)\times\mathcal R(\tild\gamma;\gamma_-)\bigr).
\]
Note that this sum is finite for action reasons.

Let $\K$ be a field of characteristic 2. We define a graded vector space $CW^*(L_0,L_1)$ by degree as
\[
	CW^k(L_0,L_1)=\hspace{-.3cm}\bigoplus_{\substack{\gamma\in\mathscr X(L_0,L_1)\\
				\deg(\gamma)=k}}\hspace{-.4cm}\K\gamma.
\]
The wrapped Floer differential
\[
	\delta\colon CW^k(L_0,L_1)\to CW^{k+1}(L_0,L_1)
\]
is given by
\[
	\delta\gamma_+=\hspace{-.8cm}\sum_{\deg(\gamma_-)-\deg(\gamma_+)=1}\hspace{-.9cm}
			\#\mathcal R(\gamma_+;\gamma_-)\cdot\gamma_-,
\]
where $\#\mathcal R(\gamma_+;\gamma_-)$ is the mod-2 count of elements of $\mathcal R(\gamma_+;\gamma_-)$. Now, $\delta^2$ counts broken trajectories connecting chords of index difference 2, which are precisely elements of the boundary of some one-dimensional moduli space of the above type. This implies $\delta^2=0$ as usual, so $\big(CW^*(L_0,L_1),\delta\big)$ is a cochain complex, called the \textbf{wrapped Floer cochain complex} of $L_0$ with $L_1$.

Each stop $\sigma\in\ssigma$ induces a filtration by $\N$ on $CW^*(L_0,L_1)$ as follows: Condition \eqref{cd:Hamiltonian positivity} in Definition \ref{defn:compatible H} means that for any $\gamma\in\mathscr X(L_0,L_1)$, the intersections of $\gamma$ with $\sigma(\hat F\times\R_+)$ are all positive. Denote the number of such intersections $n_\sigma(\gamma)$.

\begin{lemma}\label{lem:intersections filter}
The Floer differential $\delta$ never increases $n_\sigma$. In other words, $n_\sigma$ induces a filtration on wrapped Floer cochain complexes.
\end{lemma}
\begin{proof}
  Suppose $u\in\mathcal R(\gamma_+;\gamma_-)$. Since our Lagrangians avoid $\sigma(\hat F\times\R_{\ge0})$, the winding number of $\partial u$ about $D_\sigma$ coincides with the difference ${n_\sigma}(\gamma_+)-n_\sigma(\gamma_-)$. By definition, this winding number also gives the topological intersection number of $u$ with $D_\sigma$. Thus, it is enough to show that $u$ has only positive intersections with $D_\sigma$.

Recall Gromov's trick, which interprets $H$-perturbed holomorphic curves as unperturbed holomorphic sections of $\R\times[0,1]\times\hat M$, for a special choice of almost complex structure. Since $J_t$ fixes $D_\sigma$ and $X_H$ is tangent to $D_\sigma$, Gromov's trick will present $\R\times[0,1]\times D_\sigma$ as an almost complex submanifold of $\R\times[0,1]\times\hat M$. This means its intersections with the section given by $u$ are all positive \cite[Theorem 2.88]{Wend_lhc}, and linear algebra shows that the same holds for the original intersections.
\end{proof}

Combining the above for all the stops $\sigma_i\in\ssigma$, we get a filtration on $CW^*(L_0,L_1)$ by $\N^{|\ssigma|}$.

\begin{defn}\label{defn:pw Floer complex}
The \textbf{partially wrapped Floer cochain complex} of $L_0$ with $L_1$, denoted $CW_\ssigma^*(L_0,L_1)$, is the 0-filtered part of $CW^*(L_0,L_1)$. In other words, it is the subcomplex generated by those $H$-chords which don't traverse any of the stops.
\end{defn}

\subsection{$A_\infty$ categories}

We'll now construct the Fukaya $A_\infty$-categories that enhance the above Floer complexes. To begin, we establish some notation for associahedra and strip-like ends.

Following \cite{Seid_plt}, for $d\ge2$, let $\mathcal R^{d+1}$ denote the space of disks with $d+1$ boundary punctures, labeled $\zeta_0$ to $\zeta_d$ and ordered counterclockwise, modulo conformal equivalence. $\mathcal R^{d+1}$ lives naturally as interior of the $d$'th Stasheff associahedron $\ol{\mathcal R}^{d+1}$, where the boundary faces are products of lower dimensional associahedra indexed by irreducible rooted trees with $d$ ordered leaves. To be explicit, by irreducible we mean that the root vertex has valency at least two and the internal vertices have valency at least three.

Associated to the associahedra are their dg-operad of top cells, and an \textbf{$A_\infty$-category} is a category over this operad. Explicitly, an $A_\infty$-category $\mathcal A$ consists of
\begin{enumerate}
 \item A collection of objects $\mathrm{Ob}\mathcal A$.
 \item For each pair of objects $a_0,a_1\in\mathrm{Ob}\mathcal A$, a graded $\K$-vector space $\hom(a_0,a_1)$.
 \item For $k\ge1$ and all sequences of $k$ objects $L_0,\dotsc,L_k$, a map of degree $2-k$
 \begin{equation}\label{eq:Ainfty domains}
\mu^k\colon\hom(L_{k-1},L_k)\otimes\dotsm\otimes\hom(L_0,L_1)\to\hom(L_0,L_k)
\end{equation}
 satisfying the \textbf{$A_\infty$ associativity relations}
 \begin{equation}\label{eq:Ainfty rels}
  \sum_{k=1}^d\sum_{i=1}^k\mu^k(\gamma_d,\dotsc,\gamma_{i+d-k+1},
  \mu^{d-k+1}(\gamma_{i+d-k},  \dotsc,\gamma_i),\gamma_{i-1},\dotsc,\gamma_1)=0.
\end{equation}
\end{enumerate}
For an detailed treatment of $A_\infty$-categories, we refer the reader to chapter 1 of \cite{Seid_plt}.

For $\Sigma$ a boundary-punctured Riemann surface and $\zeta\in\ol{\Sigma}$ a boundary puncture, a \textbf{positive strip-like end} is a holomorphic embedding
\begin{equation}\label{eq:pos sle}
  \epsilon\colon\R_{\ge0}\times[0,1]\to\Sigma
\end{equation}
sending $\R_{\ge0}\times\{0\}$ and $\R_{\ge0}\times\{1\}$ to $\partial\Sigma$, and satisfying
\[
	\lim_{s\to\infty}\epsilon(s,t)=\zeta.
\]
Similarly, a \textbf{negative strip-like end} for $\zeta$ is a holomorphic embedding
\begin{equation}\label{eq:neg sle}
  \epsilon\colon\R_{\le0}\times[0,1]\to\Sigma
\end{equation}
sending $\R_{\le0}\times\{0\}$ and $\R_{\le0}\times\{1\}$ to $\partial\Sigma$, and satisfying
\[
	\lim_{s\to-\infty}\epsilon(s,t)=\zeta.
\]
If $\Sigma^+$ has a positive strip-like end $\epsilon^+$ and $\Sigma^-$ has a negative strip-like end $\epsilon^-$, then we can glue $\Sigma^+$ and $\Sigma^-$ with length $\ell>0$ by removing $\epsilon^+([\ell,\infty)\times[0,1])$ and $\epsilon^-((-\infty,-\ell])\times[0,1])$ and identifying, for $s\in(0,\ell)$, $\epsilon^+(s,t)$ with $\epsilon^-(s-\ell,t)$. The resulting glued surface inherits any data on $\Sigma^\pm$ supported away from the images of $\epsilon^\pm$.

A \textbf{boundary-punctured Riemann surface with strip-like ends} is a boundary-punctured Riemann surface $\Sigma$, along with a choice of a positive or negative strip-like end for each boundary puncture, such that the images of the strip-like ends are pairwise disjoint. For a disk $\Sigma^{d+1}\in\mathcal R^{d+1}$, we require this to be a choice of strip-like end $\epsilon_i$ for each $\zeta_i$, where $\epsilon_i$ is positive for $i>0$ and negative for $i=0$. Seidel has shown that we can make a \textbf{universal and consistent choice} of strip like ends: we can choose, for all $d\ge2$, a collection of strip-like ends for each $\Sigma^{d+1}$ varying smoothly over $\mathcal R^{d+1}$, and such that near $\partial\ol{\mathcal R}^{d+1}$ they agree with the strip-like ends induced by gluing. See \cite{Seid_plt} for details.  

A universal and consistent choice of strip-like ends gives rise to a \textbf{thick-thin decomposition} of each $\Sigma^{d+1}\in\mathcal R^{d+1}$, which we modify slightly from Seidel's convention. Namely, for a strip-like end $\epsilon$, define its \textbf{$m$-shift} $\epsilon^m$ by
\begin{equation}\label{eq:shifted sle}
\epsilon^m(s,t)=\begin{cases}
		  \epsilon(s+m,t)&\text{if $\epsilon$ is a positive strip-like end}\\
		  \epsilon(s-m,t)&\text{if $\epsilon$ is a negative strip-like end}.
		\end{cases}
\end{equation}
Similarly, if $S\in\Sigma^{d+1}$ is a finite-length strip obtained as the overlap from gluing $\epsilon^+$ and $\epsilon^-$ with length $\ell$, then $S^m\subset S$ is the possibly empty finite-length strip obtained as the overlap from gluing $(\epsilon^+)^m$ and $(\epsilon^-)^m$ with length $\ell-2m$. Now we can declare the thin part of $\Sigma^{d+1}$ to be the union of the images of all 3-shifts of strip-like ends and all 3-shifts of gluing regions, and the thick part to be its complement. This is the 3-shift of Seidel's thick-thin decomposition.

In everything that follows, we will assume that that we've fixed a universal and consistent choice of strip-like ends. 

Next, we recall Abouzaid's rescaling trick from \cite{Abou_gcgfc}. Departing slightly from our earlier notation, let $\phi^\tau$ be the diffeomorphism of $\hat M$ given by the time $\log\tau$ flow of the Liouville vector field. Note that pullback by $\phi^\tau$ sends Lagrangians to Lagrangians, compatible Hamiltonians to compatible Hamiltonians, and preserves equation \eqref{eq:compatible J}. Suppose then that we've fixed Lagrangians $L_0$ and $L_1$, along with a nondegenerate Hamiltonian $H$ and regular almost complex structure $J$. Then we get a natural bijection between solutions to \eqref{eq:Floer} with boundary conditions $(L_0,L_1)$ and solutions to
\begin{equation}\label{eq:pullback Floer}
	\partial_s u+(\phi^\tau)^*J(t)(\partial_t u-(\phi^\tau)^*X_H)=0
\end{equation}
with boundary conditions $\bigl((\phi^\tau)^*L_0,(\phi^\tau)^*L_1\bigr)$. The identity
\[
	(\phi^\tau)^*X_H=X_{\frac1\tau(\phi^\tau)^*H}
\]
lets us rewrite \eqref{eq:pullback Floer} as
\begin{equation}\label{eq:rescaled Floer}
	\partial_s u+J_\tau(t)(\partial_t u-X_{H_\tau}),
\end{equation}
where
\[
(J_\tau,H_\tau)=\left((\phi^\tau)^*J,\frac1\tau(\phi^\tau)^*H\right)
\]
again satisfies equation \eqref{eq:compatible J}. Since $H_\tau=\tau H$ near infinity, we can make $H_\tau$ bigger than any other given compatible Hamiltonian by taking $\tau$ sufficiently large.

Let us fix, for {\em each} pair of Lagrangians $(L_i,L_j)$ in $M$, a nondegenerate Hamiltonian $H^{i,j}$, along with a regular almost complex structure $J^{i,j}\in\mathcal J^{[0,1]}(M,H^{i,j})$. The pair $(H^{i,j},J^{i,j})$ is known as a Floer datum for $(L_0,L_1)$, and it singles out well defined wrapped and partially wrapped Floer complexes. In the sequel, this choice will be usually be implicit, and we will write, e.g., $\mathscr X(L_i,L_j)$ instead of $\mathscr X(L_i,L_j,H^{i,j})$.

For $d\ge2$ and a $d+1$-tuple of Lagrangians $(L_0,\dotsc,L_d)$, we wish to define a family of maps
\begin{equation}\label{eq:mu^d domains}
	\mu^d\colon CW^*(L_{d-1},L_d)\otimes\dotsm\otimes CW^*(L_0,L_1)
			\to CW^*(L_0,L_d)
\end{equation}
of degree $2-d$ which satisfy an analog of Lemma \ref{lem:intersections filter}. Let $\Sigma\in\mathcal R^{d+1}$. From our consistent and universal choice, $\Sigma$ is equipped with a collection of strip-like ends. Let $\partial_i\Sigma$ be the edge of $\Sigma$ between $\zeta_i$ and $\zeta_{i+1}$, or in the case $i=d$ between $\zeta_d$ and $\zeta_0$, and label $\partial_i\Sigma$ with the Lagrangian $L_i$.

The following definition is important to the present situation, but we state it in enough generality that we won't need to rewrite it too many times.

\begin{defn}\label{defn:Floer datum}
A \textbf{Floer datum} on a boundary-punctured Riemann surface $\Sigma$ with strip-like ends and Lagrangian labels consists of
\begin{enumerate}
	\item A positive real number $w_i$ for each puncture $\zeta_i$.
	\item A 1-form $\beta$ on $\Sigma$ satisfying $d\beta\le0$, $\beta|_{\partial\Sigma}=0$, and $(\epsilon_i^1)^*\beta=w_idt$ for all $i$.
	\item A $\Sigma$-parametrized compatible Hamiltonian $H$ on $M$ satisfying the following conditions.
	\begin{enumerate}
		\item\label{cd:Floer data monotonicity} $d^\Sigma H\wedge\beta\le0$ outside of a compact set. Here we view $H$ as a function on $\Sigma\times\hat M$, and $d^\Sigma H$ is the component of $dH$ in the $\Sigma$-direction. Moreover, $d^\Sigma H$ vanishes on outward normal vectors at $\partial\Sigma$, and $d\beta$ is strictly negative and bounded away from zero on the support of $d^\Sigma H$.
		\item\label{cd:Ham pos end} For each positive strip-like end $\epsilon_i$, let $L_0$ and $L_1$ be the Lagrangians assigned to the boundary components of $\Sigma$ containing $\epsilon_i(\R_{\ge0}\times\{0\})$ and $\epsilon_i(\R_{\ge0}\times\{1\})$, respectively. Then there is a scaling constant $\tau_i>0$ such that
		\[
		w_iH=H^{0,1}_{\tau_i}
		\]
		on the image of $\epsilon_i$.
		\item\label{cd:Ham neg end} For each negative strip-like end $\epsilon_i$, let $L_0$ and $L_1$ be the Lagrangians assigned to the boundary components of $\Sigma$ containing $\epsilon_i(\R_{\le0}\times\{0\})$ and $\epsilon_i(\R_{\le0}\times\{1\})$, respectively. Then there is a scaling constant $\tau_i>0$ such that
		\[
		w_iH=H^{0,1}_{\tau_i}
		\]
		on the image of $\epsilon_i$.
	\end{enumerate}
	\item A $\Sigma$-parametrized almost complex structure $J\in\mathcal J^\Sigma(M,H)$ such that
	\begin{enumerate}
	 \item\label{cd:Adapted ACS} For each strip-like end as above, $J$ satisfies
	\[
	J=J_{\tau_i}^{0,1}
	\]
	on the image of $\epsilon_i^2$.
	 \item\label{cd:ACS rescaling restriction} Let $c\colon\Sigma\to\R_+$ be the constant in the compatibility condition \eqref{eq:compatible J}. Then the support of $dc$ is disjoint from the support of $d^\Sigma H$.
	\end{enumerate}
	\item A smooth function $\tau_E\colon \partial\Sigma\to(0,\infty)$ such that $\tau_E(z)=\tau_i$ for all ends $\zeta_i$ and all points $z\in\partial\Sigma\cap\text{image}(\epsilon_i)$.
\end{enumerate}
For a family of Floer data, we strengthen condition \ref{cd:Floer data monotonicity} to require that $d\beta$ is uniformly bounded away from zero on the support of $d^\Sigma H$.

A Floer datum for a boundary-punctured Riemann surface without Lagrangian labels consists of a Floer datum for every Lagrangian labeling of that Riemann surface.
\end{defn}

\begin{lemma}\label{lem:Floer data exist}
Let $\Sigma$ be a boundary-punctured Riemann surface with strip-like ends and Lagrangian labels. Then the space of Floer data on $\Sigma$ is nonempty and contractible.
\end{lemma}
\begin{proof}
For existence, choose $\beta$, which determines $w_i$. Then choose $H$ to be $\Sigma$-independent outside of the strip-like ends. A choice of $H$ determines $\tau_i$, and from there we can fill in choices of $J$ and $\tau_E$.

For contractibility, we choose data in the order $w_i$, then $\beta$, then $\tau_i$, then $H$, $J$, and $\tau_E$. Each space of choices forms a contractible set depending on the previous choices.
\end{proof}

Following Abouzaid, we consider conformal rescalings of Floer data. Namely, we say that the Floer data $(\beta,H,J,\tau_E)$ and $(\beta',H',J',\tau'_E)$ are conformally equivalent if there are constants $C,W>0$ such that
\begin{equation}\label{eq:conf Equiv Floer data}
	\beta=W\beta',\quad H=\frac1W(H')_C,\quad J=(J')_C,\quad\tau_E=C\tau'_E.
\end{equation}
If $\Sigma^+$ has a positive strip-like end $\epsilon_i^+$ and $\Sigma^-$ has a negative strip-like end $\epsilon_j^-$, and the corresponding Lagrangian labels agree, then Floer data on $\Sigma^+$ and $\Sigma^-$ can be glued by rescaling one and patching together the data. Specifically, one chooses $C$ and $W$ so that $\tau_i^+=C\tau_j^-$ and $w_i^+=Ww_j^-$ and uses those constants in \eqref{eq:conf Equiv Floer data} to define a new Floer datum on $\Sigma^+$. The precise Floer datum obtained by iterated gluing depends on the order of the gluings, but its conformal equivalence class does not.

We now specialize back to the disks with which we will construct the $A_\infty$ structure. For that we will need coordinate charts near $\partial\mathcal R^{d+1}$, which we choose as in \cite{Seid_plt}, but with the exponential gluing profile. In other words, if $S\subset\partial\ol{\mathcal R}^{d+1}$ is a boundary stratum corresponding to a rooted tree $T$ with labeled leaves, then a chart for $\ol{\mathcal R}^{d+1}$ near $\Sigma\in S$ is
\begin{equation}\label{eq:Stasheff chart}
  \prod_{\text{internal vertices }v}\hspace{-2mm}U_v\times\prod_{\text{internal edges }e}[0,a_e).
\end{equation}
Here, $U_v$ is a subset of the space $\mathcal R^{m+1}$ corresponding to the vertex $v$, and $[0,a_e)$ is an interval of gluing parameters corresponding to the edge $e$, where gluing parameter $\rho$ corresponds to the length $\ell=e^\frac1\rho$. The identity map from such a chart to one obtained from the logarithmic gluing profile $\ell=\frac{-1}\pi\log\rho$ is smooth, and hence any smooth data on the classical associahedra can be pulled back to smooth data in these charts.

\begin{defn}\label{defn:consistent Floer data}
  A \textbf{universal and conformally consistent} choice of Floer data for $\mathcal R^{d+1}$ consists of, for all $d\ge2$, a Floer datum $\mathbf K(\Sigma)=(\beta,H,J,\tau_E)$ for each $\Sigma$ varying smoothly over $\mathcal R^{d+1}$, and such that near $\partial\ol{\mathcal R}^{d+1}$ it satisfies the following consistency condition.
\begin{enumerate}
  \item For $\Sigma$ sufficiently close to the boundary of $\mathcal R^{d+1}$, $\mathbf K(\Sigma)$ coincides on the thin part up to a conformal rescaling with the Floer datum induced by gluing.
  \item In a chart of the form \eqref{eq:Stasheff chart}, we can consider the restriction of $\mathbf K(\Sigma)$ to each piece $\Sigma_i\in\mathcal R^{m+1}$ from which $\Sigma$ is glued. This gives a family of Floer data on $\Sigma_i$ parametrized by
  \[
    U\times\prod_e(0,a_e)\times E,
  \]
  where $U\subset\mathcal R^{m+1}$ is a neighborhood of $\Sigma_i$, the intervals consist of the gluing parameters for gluing regions adjacent to $\Sigma_i$, and $E$ contains all the remaining terms in \eqref{eq:Stasheff chart}. We require that this family extends smoothly to 
  \[
    U\times\prod_e[0,a_e)\times E,
  \]
  and that on $U\times\prod_e\{0\}\times E$ it agrees up to a family of conformal rescalings with the family of Floer data that was chosen for $\mathcal R^{m+1}$.
\end{enumerate}
Write $\mathcal K(M)$ for the space of universal and conformally consistent choices of Floer data for $\mathcal R^{d+1}$.
\end{defn}

Though our situation is slightly different from Abouzaid's, Lemma 4.3 from \cite{Abou_gcgfc} still holds, namely

\begin{lemma}\label{lem:consistent FD exist}
Universal and conformally consistent choices of Floer data exist. Moreover, if $\mathbf K_0$ is such a choice and $K_\Sigma$ is another Floer datum on some $\Sigma\in\mathcal R^{d+1}$, then $K_\Sigma$ can be extended to a universal and asymptotically consistent choice that agrees with $\mathbf K_0$ on $\mathcal R^{m+1}$ for all $m<d$.
\qed
\end{lemma}

Let $(L_0,\dotsc,L_d)$ be a $(d+1)$-tuple of Lagrangians, and let
\begin{equation}\label{eq:gamma_i location}
\gamma_i\in\begin{cases}
	    \mathscr X(L_{i-1},L_i)&i\ne0\\
	    \mathscr X(L_0,L_d)&i=0.
	   \end{cases}
\end{equation}
Given a Floer datum $\mathrm K=(\beta,H,J,\tau_E)$ on some $\Sigma\in\mathcal R^{d+1}$, we can consider maps $u\colon\Sigma\to\hat M$ satisfying the generalized Floer equation
\begin{equation}\label{eq:gen Floer}
J\circ(du-X_H\otimes\beta)=(du-X_H\otimes\beta)\circ j
\end{equation}
and such that $u(\partial_i\Sigma)\subset(\phi^{\tau_E})^*L_i$ and $u(\zeta_i)=(\phi^{\tau_i})^*\gamma_i$. More, given a universal and conformally consistent choice $\mathbf K$, we can consider $\mathcal R^{d+1}(\gamma_d,\dotsc,\gamma_1;\gamma_0)$, the space of such maps as $\Sigma$ varies in $\mathcal R^{d+1}$ and $\mathrm K$ varies with $\Sigma$. Note that as with solutions to the Floer's equation on strips, a conformal rescaling of $\mathbf K$ induces a {\em canonical} identification of the corresponding versions of $\mathcal R^{d+1}(\gamma_d,\dotsc,\gamma_1;\gamma_0)$. In the sequel, we will usually make this identification implicitly.

In the expanded version \cite{Sylv_pwfc-big} we prove the following maximum principle.

\begin{lemma}\label{lem:maximum principle}
 For any compact family $\mathcal R$ of boundary punctured Riemann surfaces with strip-like ends, Lagrangian labels, and Floer data, there is some $R>0$ such that the following holds. For any holomorphic curve $u\colon\Sigma\to\hat M$ with $\Sigma\in\mathcal R$, the function $H\circ u$ has no local maximum with value greater than $R$.
\end{lemma}
\begin{proof}[Remark on proof]
 The proof is essentially standard. The extra needed ingredient is that one can bound the terms coming from $d^\Sigma H$ by those coming from $d\beta$.
\end{proof}

This shows that the maps $u$ as above are constrained to take values in some compact part of $\hat M$, so that the Gromov compactness theorem applies. This says that $\mathcal R^{d+1}(\gamma_d,\dotsc,\gamma_1;\gamma_0)$ has a natural compactification $\ol{\mathcal R}^{d+1}(\gamma_d,\dotsc,\gamma_1;\gamma_0)$ obtained by adding in broken configurations similar to those for Floer trajectories. We enumerate those broken configurations with exactly two components:
\begin{subequations}\label{eq:Ainfty binary breaks}
\begin{equation}\label{eq:Ainfty bdry breaks}
\begin{gathered}
  \mathcal R^{m+1+1}(\gamma_d,\dotsc,\gamma_{i+d-m+1},\tild\gamma,\gamma_{i},\dotsc,\gamma_1;\gamma_0)\\
  \times\mathcal R^{d-m+1}(\gamma_{i+d-m},\dotsc,\gamma_{i+1};\tild\gamma)
\end{gathered}
\hspace{.4cm}
\begin{gathered}
  \mbox{\footnotesize $1\le m\le d-2$}\\
  \mbox{\footnotesize$0\le i\le m$}\\
  \mbox{\footnotesize$\tild\gamma\in\mathscr X(L_{i},L_{i+d-m})$}
\end{gathered}
\end{equation}
\begin{equation}\label{eq:Ainfty Floer prebreaks}
\begin{gathered}
 \mathcal R^{d+1}(\gamma_d,\dotsc,\gamma_{i+1},\tild\gamma,\gamma_{i-1},\dotsc,\gamma_1;\gamma_0)\\
  \times\mathcal R(\gamma_i;\tild\gamma)
\end{gathered}
\hspace{.8cm}
\begin{gathered}
  \mbox{\footnotesize$1\le i\le d$}\\
  \mbox{\footnotesize$\tild\gamma\in\mathscr X(L_{i-1},L_i)$}
\end{gathered}
\end{equation}
\begin{equation}\label{eq:Ainfty Floer postbreaks}
\mathcal R(\tild\gamma;\gamma_0)\times\mathcal R^{d+1}(\gamma_d,\dotsc,\gamma_1;\tild\gamma)
\hspace{2cm}
\mbox{\footnotesize$\tild\gamma\in\mathscr X(L_0,L_d)$}
\end{equation}
\end{subequations}

The first kind occur when a sequence of curves has domains approaching $\partial\mathcal R^{d+1}$, and the other two occur when energy escapes through one of the strip-like ends. The configurations with more than two components are in general some combination of the above, but since they don't show up in the construction of Fukaya categories, we won't worry about them. As with Floer trajectories, there are only finitely many intermediate chords $\tild\gamma$ for which at least one of the above products is nonempty.

The key analytic ingredient is

\begin{lemma}\label{lem:Ainfty transversality}
  For generic universal and conformally consistent choices $\mathbf K\subset\mathcal K(M)$, the moduli space $\mathcal R^{d+1}(\gamma_d,\dotsc,\gamma_1;\gamma_0)$ is a transversely cut out smooth manifold of dimension
  \[\deg(\gamma_0)-\sum_{i=1}^d\deg(\gamma_i)+d-2\]
  for every $d\ge 2$, every $L_0,\dotsc,L_d$, and every $\gamma_i$ as in \eqref{eq:gamma_i location}. In this case, Gromov compactness holds and says the following.
  \begin{enumerate}
    \item If $\deg(\gamma_0)-\sum_{i=1}^d\deg(\gamma_i)=2-d$, then $\mathcal R^{d+1}(\gamma_d,\dotsc,\gamma_1;\gamma_0)$ is compact.
    \item If $\deg(\gamma_0)-\sum_{i=1}^d\deg(\gamma_i)=3-d$, then $\ol{\mathcal R}^{d+1}(\gamma_d,\dotsc,\gamma_1;\gamma_0)$ is a compact topological 1-manifold with boundary, and its boundary is the union of all binary broken curves \eqref{eq:Ainfty binary breaks}.
  \end{enumerate}
\end{lemma}
\begin{proof}
  The proof is explained in \cite{Seid_plt}, Section 9.
\end{proof}

Fix a regular $\mathbf K\in\mathcal K(M)$, and hence a moduli space $\mathcal R^{d+1}(\gamma_d,\dotsc,\gamma_1;\gamma_0)$ for all $d$ and all $\gamma_i$. We can now define what will be the $A_\infty$ operations $\mu^d$. Namely, one sets $\mu^1$ to be the Floer differential $\delta$, and
\[
\mu^d(\gamma_d,\dotsc,\gamma_1)=\hspace{-1.2cm}\sum_{\substack{\gamma_0\in\mathscr X(L_0,L_d)\\
    \deg(\gamma_0)-\sum_{i=1}^d\deg(\gamma_i)=2-d}}\hspace{-1.3cm}
 \#\mathcal R^{d+1}(\gamma_d,\dotsc,\gamma_1;\gamma_0)\cdot\gamma_0
\]
if $d\ge2$. One can check that, when they're zero-dimensional, the products in \eqref{eq:Ainfty binary breaks} encode all possible ways of composing two $\mu^d$'s with the given inputs and output. Since these make up the boundary of a compact 1-manifold, the total number of elements is even, so the $\mu^k$ satisfy the $A_\infty$ relations \eqref{eq:Ainfty rels}.

\begin{defn}\label{defn:wrapped cat}
  The \textbf{wrapped Fukaya category} of a pumpkin domain $(M,\lambda_M,\ssigma)$, denoted $\W(M,\lambda_M)$, is the $A_\infty$-category whose objects are Lagrangians in $M$, in the sense of Definition \ref{defn:pumpkin Lag}, and such that $\hom(L_0,L_1)=CW^*(L_0,L_1)$. The $A_\infty$ structure is given by the $\mu^d$ described above.

  The \textbf{interior wrapped Fukaya category} $\W^{int}(M,\lambda_M)$ is the full subcategory of the wrapped category containing only the interior Lagrangians of $M$. It is of course quasi-equivalent to $\W(M,\lambda_M)$, but being able to pass between these two versions will be useful for technical purposes in the sequel.
\end{defn}

\begin{rmk}
  Seidel has observed that while our wrapped Fukaya category of a pumpkin domain embeds as a full subcategory of the wrapped Fukaya category of the underlying Liouville domain, the latter category can have strictly more objects. These take the form of Lagrangian submanifolds which intersect the stops in an essential way. This is, however, impossible when the fiber $F$ is a Weinstein domain.
\end{rmk}

As with the Floer differential, the disks defining the higher compositions have only isolated positive intersections with the divisors of the stops. The result is that the $A_\infty$ operations preserve the intersection filtrations induced by the stops.

\begin{lemma}\label{lem:compositions preserve filtration}
  Let $\sigma\in\ssigma$ be a stop. Then, for any $d\ge1$ and composable sequence of morphisms $\gamma_1,\dotsc,\gamma_d$, we have
  \[
  n_\sigma\left(\mu^d(\gamma_d,\dotsc,\gamma_1)\right)\le\sum_{i=1}^dn_\sigma(\gamma_i).
  \]
\qed
\end{lemma}

In particular, this says that the $A_\infty$ operations preserve the partially wrapped complexes, so we can define

\begin{defn}\label{defn:partially wrapped cat}
  For $(M,\lambda_M,\ssigma)$ a pumpkin domain, its \textbf{partially wrapped Fukaya category} $\Ws(M,\lambda_M)$ is the subcategory of $\W(M,\lambda_M)$ with all the same objects and such that
  \[
   \hom_{\Ws}(L_0,L_1)=CW_\ssigma^*(L_0,L_1)
  \]

  The \textbf{interior partially wrapped Fukaya category} $\Ws^{int}(M,\lambda_M)$ is the full subcategory of the partially wrapped category containing only the interior Lagrangians of $M$.
\end{defn}

These categories have a number of important properties, which we summarize below. For proofs, see the expanded version \cite{Sylv_pwfc-big}.

\begin{prop}\label{prop:pwfc properties}
 The following hold.
 \begin{enumerate}
  \item $\Ws(M,\lambda_M)$ is a cohomologically unital subcategory of $\W(M,\lambda_M)$.
  \item Isotopic Lagrangians are isomorphic in $\W(M)$ and $\Ws(M)$. In particular, the inclusion $\Ws^{int}(M,\lambda_M)\hookrightarrow\Ws(M,\lambda_M)$ is a quasi-equivalence.
  \item Given two presentations $\Ws^1(M,\lambda_M)$ and $\Ws^2(M,\lambda_M)$ of the partially wrapped Fukaya category, there is a quasi-equivalence $F\colon\Ws^1(M,\lambda_M)\to\Ws^2(M,\lambda_M)$ called a \textbf{continuation functor} which is canonical up to homotopy. These continuation functors are obtained by counting holomorphic disks similar to those which appear in the $A_\infty$ structure.
  \item Equivalent pumpkin domains have quasi-equivalent partially wrapped Fukaya categories.
 \end{enumerate}
 \qed
\end{prop}
 
\section{Nondegenerate stops}\label{ch:nondeg stops}

\subsection{Closed strings}\label{sec:closed strings}

The proof of Theorem \ref{thm:quotient-intro} relies on the existence of a closed string version of partially wrapped Floer homology. This will take the form of a filtration on the symplectic homology chain complex, which we now explain.

Let $(M,\lambda_M,\ssigma)$ be a pumpkin domain. Pick a compatible Hamiltonian $\tild H$ and an $S^1$-family of perturbing Hamiltonians $P_t\colon\hat M\to\R_{\ge0}$, where $t\in\R/\Z\cong S^1$, which satisfy the following conditions.
\begin{equation}
 \parbox{14cm}{$P_t$ is bounded, and $\|X_{P_t}\|$ decays exponentially in the symplectization coordinate $\sqrt{\tild H}$ for any metric of the form $\hat\omega_M(\cdot,J\cdot)$ with $J\in\mathcal J(M,H)$.}
\end{equation}
\begin{equation}
 \parbox{14cm}{$H_t:=\tild H+P_t$ is \textbf{nondegenerate} in the sense that for any 1-periodic orbit $x$ of the time-dependent vector field $X_{H_t}$, the linearized Poincar\'e return map of $x$ does not have 1 as an eigenvalue.}
\end{equation}
\begin{equation}\label{cd:SH Ham nicely split}
 \parbox{14cm}{For each $\sigma\in\ssigma$, $\tild H$ is of the form \eqref{eq:sum H} near $D_\sigma$, with $f(z)=c|z|^2$ and $c>0$. Similarly, $P_t$ is independent of the $\H_\rho$-coordinate near $D_\sigma$, and $X_{H_t}$ satisfies condition \eqref{cd:Hamiltonian positivity} of Definition \ref{defn:compatible H}.}
\end{equation}
Note that perturbing Hamiltonians $P$ can be constructed by taking sums of functions of the form $\kappa\circ H$, where $H$ is a compatible Hamiltonian and $\kappa\colon\R_+\to\R_+$ is a positive nondecreasing function which is eventually constant.

Let $\mathscr X(H_t)$ be the space of 1-periodic orbits of $H_t$. Because $H_t$ is nondegenerate, this is a discrete space. The symplectic cochain complex $SC^*(M)$ is the graded $\K$-vector space generated by $\mathscr X(H_t)$ with grading given by the cohomological Conley-Zehnder index, see \cite{Abou_shbook}. For this situation, we switch to a different convention for almost complex structures. Namely, we follow Ganatra \cite{Gana_scdwc} and define the space of almost complex structures of \textbf{rescaled contact type} $\mathcal J^{S^1}_{resc}(M,H_t)$. This consists of $S^1$-families of almost complex structures $J_t$ which are $\hat\omega_M$-compatible and satisfy the following three conditions. First, there is some $t$-independent constant $c_{resc}>0$ such that, for all $t$,
\begin{equation}\label{eq:rescaled J}
 d\tild H\circ J_t=-c_{resc}\hat\lambda_M
\end{equation}
outside of a $t$-independent compact set. Second, the restriction $J_t|_{\ker dH\cap\ker\hat\lambda_M}$ is asymptotically $\hat Z_M$-invariant. Third, for each stop $\sigma\in\ssigma$, the projection to $\H_\rho$ is holomorphic along $D_\sigma$. Given $J_t\in\mathcal J^{S^1}_{resc}(M,H_t)$, we can consider maps $u\colon\R\times S^1\to\hat M$ satisfying Floer's equation
\[
 \partial_su+J_t(\partial_tu-X_{H_t})=0
\]
and asymptotic as $s\to\pm\infty$ to orbits $x_\pm\in\mathscr X(H_t)$. The moduli space of such maps, denoted $\tild{\mathcal Q}(x_+,x_-)$, satisfies the usual transversality and compactness statements, i.e.
\begin{lemma}\label{lem:SH transversality}
 For generic choices of $P_t$ and $J_t\in\mathcal J^{S^1}_{resc}(M,H_t)$, $\tild{\mathcal Q}(x_+;x_-)$ is transversely cut out of dimension $\deg(x_-)-\deg(x_+)$ for every $x_\pm\in\mathscr X(H_t)$. In this case, the translation $\R$-action on $\tild{\mathcal Q}(x_+;x_-)$ is free if and only if $x_+\ne x_-$. Writing $\mathcal Q(x_+;x_-)$ for the quotient $\tild{\mathcal Q}(x_+;x_-)/\R$, we have
 \begin{enumerate}
  \item If $\deg(x_-)-\deg(x_+)=1$, then $\mathcal Q(x_+;x_-)$ is compact.
  \item If $\deg(x_-)-\deg(x_+)=2$, then $\mathcal Q(x_+;x_-)$ admits a Gromov compactification as a topological 1-manifold with boundary, and its boundary is in natural bijection with the once-broken configurations $\coprod_{y\in\mathscr X(H_t)}\mathcal Q(y;x_-)\times\mathcal Q(x_+;y)$.
 \end{enumerate}
\end{lemma}
\begin{proof}[Remark on proof]
 Because \eqref{eq:rescaled J} is so stringent, we allow small perturbations of $P_t$. These can be made without changing the set $\mathscr X(H_t)$, and in concert with the freedom to perturb $c_{resc}$ they allow us to achieve transversality even when $\dim M=2$. This was not an issue for chords because, when $\dim M=2$, all chords outside of a compact set in a given end live in different relative homotopy classes.
 
 With regards to compactness, our maximum principle does not apply in the presence of a time-dependent perturbing Hamiltonian, but by choosing a symplectization coordinate $r=\sqrt{\tild H}$ we find ourselves in Ganatra's setup and can apply Theorem A.1 of \cite{Gana_scdwc}.
\end{proof}

Fix a regular $J_t\in\mathcal J^{S^1}_{resc}(M,H_t)$. The differential $\partial$ on $SC^*(M)$ is given by
\[
 \partial x_+=\hspace{-.8cm}\sum_{\deg(x_-)-\deg(x_+)=1}\hspace{-.9cm}
			\#\mathcal Q(x_+;x_-)\cdot x_-
\]
and satisfies $\partial^2=0$ by the usual argument which looks at ends of 1-dimensional moduli spaces. The cohomology $SH^*(M):=H^*(SC^*(M),\partial)$ is known as symplectic cohomology.

The pumpkin structure $\ssigma$ endows $SC^*(M)$ with a filtration similar to that for open strings but slightly more subtle due to the fact that orbits can live on the divisor of a stop. We describe a part of it, which will suffice for our purposes. Let $\mathscr X_\ssigma(H_t)\subset\mathscr X(H_t)$ be the set of orbits which do not intersect $\sigma(\hat F\times\R_+)$ and do not live in $D_\sigma$ for any $\sigma\in\ssigma$. Let $SC_\ssigma^*(M)\subset SC^*(M)$ be the graded linear subspace generated by $\mathscr X_\ssigma(H_t)$.

\begin{lemma}\label{lem:SH intersections filter}
 $SC_\ssigma^*(M)$ is a subcomplex of $SC^*(M)$.
\end{lemma}
\begin{proof}
 We need to show that if $x_-$ intersects $\sigma(\hat F\times\R_+)$ or lives in $D_\sigma$ for some $\sigma\in\ssigma$, then $\tild Q(x_+;x_-)$ is empty for any $x_+\in\mathscr X(H_t)$. In the first case, the conclusion follows from positivity of intersections as in Lemma \ref{lem:intersections filter}. In the second, the asymptotics in \cite{Sief_asymp}, combined with assumption \eqref{cd:SH Ham nicely split}, ensure that $x_-$ behaves as if it had strictly positive winding number around $D_\sigma$. This puts us back in the regime where we can use positivity of intersections.
\end{proof}

\begin{defn}\label{defn:PW SH}
 $SC_\ssigma^*(M)$ is called the \textbf{partially wrapped symplectic cochain complex}.
\end{defn}

Note that unlike the $CW_\ssigma^*(L,L)$, $SC_\ssigma^*(M)$ is in general non-unital. This is because unlike in the open string case, some of the holomorphic caps which contribute to the unit in usual symplectic cohomology live entirely in some $D_\sigma$.

We will be interested in holomorphic curves which interpolate between the open and the closed string worlds. For this, we make the following definitions
\begin{defn}\label{defn:punctured surface w/ cylinders}
 A \textbf{punctured Riemann surface with boundary, ends, and cylinders} is a Riemann surface
 \[
  \Sigma=\ol\Sigma\setminus(Z_{\partial\Sigma}\cup Z_\Sigma),
 \]
 where $\ol\Sigma$ is a compact Riemann surface with boundary, $Z_{\partial\Sigma}$ is a finite subset of the boundary of $\ol\Sigma$, and $Z_\Sigma$ is a finite subset of the interior of $\ol\Sigma$, together with the following additional data.
 \begin{enumerate}
  \item For each $\zeta\in Z_{\partial\Sigma}$, a positive or negative strip-like end at $\zeta$.
  \item For each $\zeta\in Z_\Sigma$, a positive or negative cylindrical end at $\zeta$. These are holomorphic embeddings
  \begin{equation}\label{eq:cylindrical end}
   \epsilon_+\colon\R_{\ge0}\times S^1\to\Sigma\qquad\text{ or }\qquad\epsilon_-\colon\R_{\le0}\times S^1\to\Sigma,
  \end{equation}
  respectively, such that $\lim_{s\to\pm\infty}\epsilon_\pm(s,t)=\zeta$.
  \item A finite number of \textbf{finite cylinders} $\delta_i$. These are holomorphic embeddings
  \[
   \delta_i\colon[a_i,b_i]\times S^1\to\mathrm{int}(\Sigma).
  \]
 \end{enumerate}
 Additionally, we require that all ends and finite cylinders have disjoint images. For cylindrical ends and finite cylinders, we define their \textbf{$m$-shifts} as with strips and define the \textbf{thin part} of $\Sigma$ to be the union of the 3-shifts of all ends, finite cylinders, and, if $\Sigma$ comes with an implicit gluing decomposition, finite strip-like gluing regions.
 
 A \textbf{punctured Riemann surface with labeled boundary, ends, and cylinders} is a punctured Riemann surface $\Sigma$ with boundary, ends, and cylinders, along with an assignment of a Lagrangian $L_i\subset\hat M$ to each boundary component $\partial_i\Sigma$ of $\Sigma$.
\end{defn}

\begin{defn}\label{defn:Floer datum w/ cylinders}
 Let $\Sigma$ be a punctured Riemann surface with labeled boundary, ends, and cylinders. A \textbf{Floer datum} on $\Sigma$ is a 5-tuple $(\beta,H^{strict},P,J,\tau_E)$, where
 \begin{itemize}
  \item $\beta$ is a 1-form on $\Sigma$
  \item $H^{strict}$ is a $\Sigma$-parametrized compatible Hamiltonian
  \item $P$ is a function $P\colon\Sigma\times\hat M\to\R_+$
  \item $J$ is a $\Sigma$-parametrized $\hat\omega_M$-compatible almost complex structure
  \item $\tau_E$ is a function $\tau_E\colon\partial\Sigma\to\R_+$
 \end{itemize}
 with the following properties.
 \begin{enumerate}
  \item Outside the images of the cylindrical ends and finite cylinders, $(\beta,H^{strict},J,\tau_E)$ satisfy the conditions of Definition \ref{defn:Floer datum}.
  \item $d\beta$, $d^\Sigma H^{strict}\wedge\beta$, and $d^\Sigma P\wedge\beta$ are nonpositive everywhere.
  \item For each cylindrical end $\epsilon_i$, $(\epsilon_i^1)^*\beta=w_idt$ for some positive real number $w_i$. Similarly, for each finite cylinder $\delta_i$, $(\delta_i^1)^*\beta=w_idt$ for some positive real number $w_i$.
  \item For each cylindrical end or finite cylinder, there is a scaling constant $\tau_i>0$ such that
  \[
   w_iH^{strict}=\tild H_{\tau_i}
  \]
  on the image of that cylindrical end or finite cylinder.
  \item There is some strictly positive function $g\colon\Sigma\times\R_+\to\R_+$ such that
   \[
    dH^{strict}\circ J=-g(H^{strict})\hat\lambda_M
   \]
   outside a $\Sigma$-independent compact set.
  \item The restriction $J|_{\ker dH^{strict}\cap\ker\hat\lambda_M}$ is asymptotically $\hat Z_M$-invariant.
  \item For each stop $\sigma\in\ssigma$, the projection to $\H_\rho$ is holomorphic along $D_\sigma$.
  \item For each cylindrical end or finite cylinder,
   \[
   J(s,t)=(J_t)_{\tau_i}:=(\phi^{\tau_i})^*J_t
   \]
  in the 2-shift of that cylindrical end or finite cylinder.
  \item $P$ is globally bounded, and $\|X_P\|$ decays exponentially in the symplectization coordinate. Moreover, $P$ is locally constant outside the 2-shifts of the cylindrical ends and finite cylinders, and in the 3-shifts of the cylindrical ends and finite cylinders it satisfies
  \[
   w_iP(s,t)=(P_t)_{\tau_i}+A_i:=\frac1{\tau_i}(\phi^{\tau_i})^*P_t+A_i
  \]
  for some constant $A_i$ depending on the cylindrical region.
 \end{enumerate}
 For simplicity, define $H:=H^{strict}+P$.
 
 A Floer datum for a punctured Riemann surface $\Sigma$ with boundary, ends, cylinders, but no Lagrangian labels, consists of a Floer datum for each Lagrangian labeling of $\Sigma$.
\end{defn}

For such Floer data, there is again a notion of conformal equivalence. Namely, two Floer data $(\beta,H^{strict},P,J,\tau_E)$ and $(\beta',(H^{strict})',P',J',\tau_E')$ are \textbf{conformally equivalent} if there are constants $A,C,W$ with $C,W>0$ such that
\[
 \beta=W\beta',\quad H^{strict}=\frac1W((H^{strict})')_C,\quad P=\frac1W(P')_C+A,\quad J=(J')_C,\quad\tau_E=C\tau'_E.
\]
Compared to the situation without cylinders, we now have the additional freedom to shift $H$ by a constant.

As before, solutions $u\colon\Sigma\to\hat M$ to
\begin{equation}\label{eq:gen Floer w/ cylinders}
 J\circ(du-X_H\otimes\beta)=(du-X_H\otimes\beta)\circ j
\end{equation}
with boundary conditions $u(\partial_iE)\subset(\phi^{\tau_E})^*L_i$ are related to solutions $u'\colon\Sigma\to\hat M$ to
\begin{equation}
 J'\circ(du'-X_{H'}\otimes\beta')=(du'-X_{H'}\otimes\beta')\circ j
\end{equation}
with boundary conditions $u'(\partial_iE)\subset(\phi^{\tau_E'})^*L_i$ via Liouville pullback. Here, of course, $H'$ means $(H^{strict})'+P'$.

As indicated after Lemma \ref{lem:SH transversality}, our maximum principle Lemma \ref{lem:maximum principle} doesn't apply inside the 2-shifts of the cylinders. However, here we still have action control, so Ganatra's $C^0$ estimates apply. Together with the maximum principle outside the 2-shifts of the cylinders, this will be enough to obtain compactness of our moduli spaces.

\subsection{Action inequalities}\label{sec:action}

One generally thinks of a nondegenerate Liouville domain \cite{Gana_scdwc} as one with ``enough Lagrangians''. More precisely though, it is one with enough \emph{isotopy classes} of Lagrangians, and strong nondegeneracy (Definition \ref{defn:strongly nondegenerate}) enhances this to geometric Lagrangians. Morally speaking, it suggests that the ``space-filling holomorphic disk'' counted by the open-closed map can be cut into arbitrarily small pieces by introducing additional Lagrangians. In all known examples, these new Lagrangians are translates of the original split-generators.

In this case, the additional structure comes from an action filtration. To this effect, we review the behavior of the action functional under the $A_\infty$ and similar operations. Let $\Sigma$ be a punctured Riemann surface with labeled boundary, ends, and cylinders, and let $K$ be a Floer datum on $\Sigma$. We consider the space $\mathcal M$ of maps $u\colon\Sigma\to\hat M$ sending every $z\in\partial_i\Sigma$ to a point $L_i(z)$ and satisfying the inhomogeneous Cauchy-Riemann equation \ref{eq:gen Floer w/ cylinders}. Given a K\"ahler metric on $\Sigma$, define the \textbf{geometric energy} of such $u$ as
\[
 E^{geom}(u)=\int_\Sigma\lVert du-X_H\otimes\beta\rVert^2 dvol,
\]
where the metric on $\hat M$ is $\Sigma$-dependent and is obtained from $J$. This is independent of the choice of K\"ahler metric on $\Sigma$ and is in fact given by
\[
 E^{geom}(u)=\int_\Sigma\left(u^*\hat\omega_M-u^*(d^{\hat M}H)\wedge\beta\right).
\]

We also define the \textbf{topological energy} of $u$ by
\[
\begin{aligned}
 E^{top}(u)&=\int_\Sigma\left(u^*\hat\omega_M-d(u^*H\beta)\right)\\
 &=E^{geom}(u)-\int_\Sigma\left(u^*Hd\beta+u^*(d^\Sigma H)\wedge\beta\right).
\end{aligned}
\]
\begin{lemma}\label{lem:E^geom E^top inequality}
 There is some constant $C\in\R$ depending only on $\Sigma$, $H$, and $\beta$ for which
 \[
  E^{top}(u)\ge E^{geom}(u)+C
 \]
 for all $u\in\mathcal M$.
\end{lemma}
\begin{proof}
 $Hd\beta$ is nowhere positive, and $d^\Sigma H\wedge\beta$ is only positive on a compact subset of $\Sigma\times\hat M$.
\end{proof}
\begin{cor}\label{cor:E^geom E^top bound}
 The above constant $C$ is bounded below by $-\int_\Sigma\max_{\hat M}(d^\Sigma H\wedge\beta)$
 \qed
\end{cor}

We now examine the relation between action and topological energy . For any $u\in\mathcal M$ with $E^{geom}(u)<\infty$, $u$ converges at each interior puncture $\zeta_j$ to a 1-periodic orbit $x_j$ of $w_jX_H$ and at each boundary puncture $\zeta^k$ to a chord $\gamma^k$ of $w^kX_H=w^kX_{H^{strict}}$, for appropriate 1-parameter specializations of $H$ and $P$. Before we define the action, choose for each Lagrangian $L$ a primitive $f_L\in C^\infty(L)$ of the exact 1-form $\hat\lambda_M|_L$. For an orbit $x$, define
\[
 A(x)=\int_{S^1}\left(x^*\hat\lambda_M-wx^*Hdt\right).
\]
For a chord $\gamma$, let $L_0$ and $L_1$ be the Lagrangians containing $\gamma(0)$ and $\gamma(1)$. Define
\[
\begin{aligned}
 A_0(\gamma)&=\int_{[0,1]}\left(\gamma^*\hat\lambda_M-w\gamma^*Hdt\right)\\
 A(\gamma)&=\int_{[0,1]}\left(\gamma^*\hat\lambda_M-w\gamma^*Hdt\right)+\frac1\tau f_{L_0}(\phi^\tau\gamma(0))-\frac1\tau f_{L_1}(\phi^\tau\gamma(1)),
\end{aligned}
\]
where $\tau$ is the rescaling factor associated to the boundary puncture mapping to $\gamma$. A short computation gives
\begin{lemma}
 If $K$ is a Floer datum on $\Sigma$, then the action of a chord or orbit with respect to $K$ is $\tau$ times the action of the time $\log\tau$ Liouville pullback of that chord or orbit with respect to the Floer datum $K_\tau$.
 \qed
\end{lemma}

With this set up, we have the standard action-energy identity
\begin{lemma}\label{lem:action-energy identity}
 Every $u\in\mathcal M$ with $E^{geom}(u)<\infty$ satisfies
 \[
  E^{top}(u)=\hspace{-2mm}\sum_{\substack{\text{positive}\\ \text{cylindrical}\\ \text{ends }\epsilon_j}}\hspace{-2mm}A(x_j)+\hspace{-2mm}\sum_{\substack{\text{positive}\\ \text{strip-like}\\ \text{ends }\epsilon^k}}\hspace{-2mm}A(\gamma^k)-\hspace{-2mm}\sum_{\substack{\text{negative}\\ \text{cylindrical}\\ \text{ends }\epsilon_j}}\hspace{-2mm}A(x_j)-\hspace{-2mm}\sum_{\substack{\text{negative}\\ \text{strip-like}\\ \text{ends }\epsilon^k}}\hspace{-2mm}A(\gamma^k)+\int_{\partial\Sigma}\frac{(\phi^{\tau_E}\circ u)^*f_L}{\tau_E^2}d\tau_E.
 \]
\end{lemma}
\begin{proof}
 Stokes's theorem.
\end{proof}
\begin{cor}\label{cor:action bound}
 Suppose $\Sigma$ has exactly one negative end $\epsilon_-$, either strip-like or cylindrical. Let $u\in\mathcal M$ be a finite-energy holomorphic curve converging at $\epsilon_-$ to an orbit or chord $y$. Then
 \[
  A(y)\le\hspace{-2mm}\sum_{\substack{\text{positive}\\ \text{cylindrical}\\ \text{ends }\epsilon_j}}\hspace{-2mm}A(x_j)+\hspace{-2mm}\sum_{\substack{\text{positive}\\ \text{strip-like}\\ \text{ends }\epsilon^k}}\hspace{-2mm}A(\gamma^k)+\int_\Sigma\max_{\hat M}(d^\Sigma H\wedge\beta)+\lVert f_L\rVert_{C^0}\!\int_{\partial\Sigma}\frac{|d\tau_E|}{\tau_E^2}.
\]
 \qed
\end{cor}

\subsection{Action filtrations for Hochschild homology}\label{sec:Hochschild action}

To define a strongly nondegenerate stop, we will need to transfer the action filtration from wrapped Floer homology to Hochschild homology. Unfortunately, the energy leaks of Corollary \ref{cor:action bound} prevent the action from extending directly to a filtration on the Fukaya category, so we will need to find a way to package those leaks.

To begin, recall that for any $A_\infty$-category $\mathcal A$, the Hochschild homology of $\mathcal A$ can be given as the homology of a chain complex
\begin{equation}\label{eq:cyclic complex}
 CC_*(\mathcal A)=\bigoplus_{d=1}^\infty\bigoplus_{\text{words}}\K\gamma_d\otimes\dotsm\otimes\gamma_1
\end{equation}
where $\gamma_i\in\hom(L_i,L_{i+1})$ is a cyclically composable sequence of morphisms in $\mathcal A$. The grading is cohomological and is given by
\[
 \deg(\gamma_d\otimes\dotsm\otimes\gamma_1)=\sum_{i=1}^d\deg(\gamma_i)+1-d.
\]
The differential $\delta\colon CC_*(\mathcal A)\to CC_{*+1}(\mathcal A)$ comes from the $A_\infty$ structure on $\mathcal A$, namely
\[
\begin{gathered}
 \delta(\gamma_d\otimes\dotsm\otimes\gamma_1)=\sum_{\substack{i,j\ge0\\i+j<d}}\mu^{i+j+1}(\gamma_{i},\dotsc,\gamma_{d-j})\otimes\gamma_{d-j-1}\otimes\dotsm\otimes\gamma_{i+1}\\
 \hspace{3cm}+\sum_{\substack{i,j\ge1\\i+j\le d}}\gamma_d\otimes\dotsc\otimes\gamma_{i+j}\otimes\mu^{j}(\gamma_{i+j-1},\dotsc,\gamma_{i})\otimes\gamma_{i-1}\otimes\dotsm\otimes\gamma_1
\end{gathered}
\]
where $\gamma_0:=\gamma_d$.

Now let $(F,\lambda_F)$ be a Liouville domain, which we think of as the fiber of a stop. Choose an auxiliary metric on $\hat F$ which $\hat Z_F$-invariant outside of a compact set. A \textbf{normalizing Hamiltonian} is a function $H^b\colon\hat F\to\R_{\ge0}$ which is smooth and satisfies
\begin{enumerate}
  \item\label{cd:normalizing Ham form} $H^b=Q\kappa(Q)$, where
  \begin{enumerate}
    \item $Q\colon\hat F\to\R_{\ge0}$ is a proper, strictly quadratic continuous function.
    \item $\kappa\colon\R_{\ge0}\to[0,1]$ is a nondecreasing cutoff function which is $0$ in a neighborhood of $0$ and is $1$ outside a compact set.
  \end{enumerate}
  \item Whenever $dH^b(Z)\ne2H^b$, $\lVert H^b\rVert_{C^2}$ is small enough to rule out nonconstant time 1 orbits a priori. Note that this can be achieved by multiplying any function of the above form by a small positive constant.
\end{enumerate}
\begin{rmk}\label{rmk:rescaling normalizing Hams}
 Condition \eqref{cd:normalizing Ham form} implies that $dH^b(\hat Z)=2H^b$ outside a compact set, and that $H^q_\tau\ge\tau H^q$ globally for any $\tau\ge1$. Moreover, given any other $H^b\,'$ of this form, there is some $\tau$ such that $(H^b\,')_\tau\ge H^b$ globally.
\end{rmk}

Let $H^b$ be a normalizing Hamiltonian. As the name suggests, we will use it as a normalization for our action filtration. For $\varepsilon>0$, consider the class of functions $\mathcal H_\varepsilon(F)$ consisting of those compatible Hamiltonians $H$ on $\hat F$ which satisfy
\begin{subequations}\label{cd:close to normalizing Hamiltonian}
 \begin{equation}
  \left\lVert\log\frac H{H^b}\right\rVert_{C^1}<\frac\varepsilon2\quad\mbox{whenever}\quad H^b>1
 \end{equation}
 \begin{equation}
  \lVert H-H^b\rVert_{C^2}<\frac\varepsilon2\quad\mbox{whenever}\quad H^b\le1
 \end{equation}
\end{subequations}
These conditions imply that there is a constant $A_{\min}>0$ such that for sufficiently small $\varepsilon$, any $H\in\mathcal H_\varepsilon(F)$ satisfies
\begin{equation}\label{eq:Amin}
 \parbox{14cm}{Every nonconstant time 1 orbit of $X_H$ has action greater than $A_{\min}$.}
\end{equation}
On the other hand, it is always the case that a constant orbit has negative action.

Consider now a presentation $\W^\varepsilon(F)$ of the wrapped Fukaya category $\W(F)$ which satisfies the following conditions.
\begin{enumerate}
 \item The objects consist of all Lagrangians in the sense of Definition \ref{defn:pumpkin Lag} on which a primitive of $\hat\lambda_F$ is globally constant outside a compact set has total variation bounded by $\varepsilon$. Note that every Lagrangian is isotopic to such a Lagrangian via a combination of Moser flow to fix the ends and inverse Liouville flow to reduce the variation.
 \item For each pair $(L_0,L_1)$ of Lagrangians, the Hamiltonian $H^{0,1}$ belongs to $\mathcal H_\varepsilon(F)$.
 \item All Floer data for $\Sigma^{d+1}$ (cf. Definition \ref{defn:Floer datum}) have the following three properties. First, $e^{1-d}\tau_0<\tau_E<\tau_0$. Second, each boundary component $E$ of $\Sigma$ contains disjoint open intervals $E^+$ and $E^-$ such that $\tau_E$ is nonincreasing outside $E^+$ and nondecreasing outside $E^-$. Third
 \[
  \int_{\Sigma^{d+1}}\max_{\hat F}(d^\Sigma H\wedge\beta)<\frac{2de^{d-1}\varepsilon}{\tau_0}.
 \]
\end{enumerate}
In the last item, the first two conditions say that the Floer data doesn't rescale excessively, while the last says that it doesn't wobble too much on the compact part of $\hat F$. To see that this can be achieved, note first that the thick-thin decomposition of $\Sigma^{d+1}$ has at most $2d$ thin pieces aside from the output strip-like end. We can choose the Floer data so that $d^\Sigma H$ is supported near the thin part, where near each piece, as one moves towards the negative end, $H$ changes within a rescaled $\mathcal H_\varepsilon(F)$ while experiencing a further small positive rescaling. The extent to which it can decrease on the interior is bounded by $a\varepsilon$, for some $a$ slightly greater than 1, times a rescaling factor which depends on the location in $\Sigma^{d+1}$. By the first inequality, this local factor is bounded by $e^{d-1}\tau_0$.

One could reduce the exponential part of the coefficient from $e^{d-1}$ to $(2+\varepsilon')^{d-1}$ for some small $\varepsilon'$ depending on $\varepsilon$, but conformal consistency prevents us from doing any better than that. This is because the bijection from  $X_H$-chords to $X_{H_\tau}$-chords multiplies action by $\frac1\tau$, and $\mathcal R^{d+1}$ has corners which come from sequentially gluing $d-1$ triangles. Each of these triangles rescales the Floer data by a factor of at least 2. That said, in Sections \ref{sec:stabilizations} and \ref{sec:collapsing energy}, it will be convenient to allow rescalings more severe than $2+\varepsilon'$.

The above conditions allow us to write down a \textbf{shifted action}
\[
A^\varepsilon\colon CC_*(\W^\varepsilon(F))\to\R_+
\]
given by
\begin{equation}\label{eq:shifted action}
 A^\varepsilon(\gamma_d\otimes\dotsm\otimes\gamma_1):=e^{d-1}\sum_{i=1}^d(A(\gamma)+8\varepsilon)
\end{equation}
and
\[
 A^\varepsilon(\vec\gamma+\vec\gamma')=\max\{A^\varepsilon(\vec\gamma),A^\varepsilon(\vec\gamma')\}.
\]
Here, $A(\gamma)$ is taken with respect to the Floer data defining the Floer cochain complex (i.e. $\tau=1$). Even though the action of a chord $A(\gamma)$ depends on the choice of a primitive of $\hat\lambda_M$ for each Lagrangian, the shifted action doesn't because it cancels on the cyclic chain.

\begin{lemma}
 $A^\varepsilon$ strictly filters the Hochschild chain complex.
\end{lemma}
\begin{proof}
 The condition on the total variation of primitives of Lagrangians, ensures that the primitive contributes no less than $-\varepsilon$ to the action of a chord, while the definition of $\mathcal H_\varepsilon(F)$ ensures that
 \[
  A_0(\gamma)>-\frac\varepsilon2
 \]
 for any Floer generator $\gamma$. This implies that each term of the sum in \eqref{eq:shifted action} is greater than $\frac{13}2\varepsilon$, and in particular it's positive.
 
 Next, note that the Floer differential strictly decreases the action filtration, so we focus on the nonlinear terms in the Hochschild differential. For this, Corollary \ref{cor:action bound} shows that the unshifted action
 \[
  A(\gamma_d\otimes\dotsm\otimes\gamma_1):=e^{d-1}\sum_{i=1}^dA(\gamma)
 \]
 filters the Hochschild complex up to a leak bounded by
 \[
  \int_\Sigma\max_{\hat F}d^\Sigma H\wedge\beta+\lVert f_L\rVert_{C^0}\!\int_{\partial\Sigma}\frac{|d\tau_E|}{\tau_E^2}.
 \]
 The conditions on the Floer data defining $\W^\varepsilon(F)$ ensure that, after normalizing $\tau_0=1$, this is bounded by $4de^{d-1}\varepsilon$.
\end{proof}

For a homology class $c\in HH_*(\W^\varepsilon(F))$, we define its \textbf{length-$k$ shifted action} to be
\[
 A^{\varepsilon,k}(c)=\inf\left\{A^\varepsilon(\vec\gamma)\;\middle|\;[\vec\gamma]=c\mbox{ and }\vec\gamma\in\bigoplus_{i=1}^k\bigoplus_{\substack{\text{words}\\ \text{of length $i$}}}\K\gamma_i\otimes\dotsm\otimes\gamma_1\right\},
\]
i.e. the minimal shifted action of any representative of $c$ with length at most $k$. Now any functor between $A_\infty$-categories induces a map on Hochschild homology, and hence we get for any continuation functor $\mathscr F\colon\W^\varepsilon(F)\to\W(F)$ a homomorphism of graded vector spaces
\[
 HH(\mathscr F)\colon HH_*(\W^\varepsilon(F))\to HH_*(\W(F)).
\]
From the existence of homotopies between continuation functors, it follows that there is a \emph{canonical} isomorphism $HH_*(\W^\varepsilon(F))\cong HH_*(\W(F))$. Identifying Hochschild homology for the various presentations, we define
\begin{defn}\label{defn:HH action}
 Let $c\in HH_*(\W(F))$. Its \textbf{length-$k$ $H^b$-normalized action} is defined to be
 \begin{equation}\label{eq:HH action}
  A^{b,k}(c)=\liminf_{\varepsilon\to0}\{A^{\varepsilon,k}(c)\mbox{ for some presentation of }\W^\varepsilon(F)\}.
 \end{equation}
\end{defn}

With the normalized action in hand, we can state the nondegeneracy criterion.

\begin{defn}\label{defn:strongly nondegenerate}
 A Liouville domain $F$ is called \textbf{strongly nondegenerate} if following conditions hold.
 \begin{enumerate}
  \item For every connected component $F_i$ of $F$, $SH^*(F_i)\ne0$. Here, $SH^*(F)$ is symplectic cohomology, i.e. the cohomology of $SC^*(F)$, first defined in \cite{Floe-Hofe,Ciel-Floe-Hofe}.
  \item $F$ is nondegenerate, i.e. the unit $\mathbf1\in SH^*(F)$ is in the image of the standard homological open-closed map 
  \[
   \mathcal{OC}\colon HH_*(\W(F))\to SH^*(F).
  \]
  See Section \ref{sec:OC maps} for a definition.. Ganatra \cite{Gana_scdwc} showed that this implies that $\mathcal{OC}$ is an isomorphism.
  \item There is some $k\in\N$ and some normalizing Hamiltonian $H^b$ for which $A^{b,k}(e)=0$, where $e=\mathcal{OC}^{-1}(\mathbf1)$ is the \textbf{Hochschild fundamental class}.
 \end{enumerate}
 A stop is called \textbf{strongly nondegenerate} if its fiber is strongly nondegenerate.
\end{defn}

\begin{eg}
 Abouzaid's description of the wrapped Fukaya category of a cotangent bundle \cite{Abou_wfcbl,Abou_cfgfc} shows that any cotangent bundle $F=T^*X$ is strongly nondegenerate. In this case we may take $k=\dim(X)+1$, and a low-action, length-$k$ Hochschild fundamental cycle (i.e. a representative of $e$) is obtained by picking a fine triangulation of $X$. Likewise, for $F$ a punctured Riemann surface other than $\C$, a low-action, length-3 fundamental cycle can be obtained from a fine decomposition of $F$ into triangles and strips. More generally, it is likely that a careful analysis of the continuation elements in \cite{GPS_gen} would show that any Weinstein domain with nontrivial symplectic cohomology admitting arbitrarily fine covers by nondegenerate Liouville sectors (i.e. pumpkin domains minus the interior of the images of their stops) is strongly nondegenerate.
\end{eg}

As with the $A_\infty$ operations, the energy leakage of the $d$'th term of a continuation functor between presentations of $\W^\varepsilon(F)$ can be arranged to be bounded by a constant times $de^{d-1}\varepsilon$. For a low-action representative of $e\in HH_*(\W^\varepsilon(F))$, this quantity is small because $d$ is a priori bounded, and thus continuation functors map low-action fundamental cycles to low-action fundamental cycles. Hence, we see that the condition $A^{b,k}(e)=0$ can be detected by \emph{any} sequence of presentations of $\W^\varepsilon(F)$ with $\varepsilon$ tending to $0$.
 
In fact, in view of Remark \ref{rmk:rescaling normalizing Hams}, the above argument shows

\begin{lemma}\label{lem:any hamiltonian detects action zero}
 If a Liouville domain is strongly nondegenerate with $A^{b,k}(e)=0$ for a single normalizing Hamiltonian, then $A^{b,k}(e)=0$ for every normalizing Hamiltonian.
 \qed
\end{lemma}
\begin{cor}\label{cor:any sequence detects action zero}
 In the above situation, given any normalizing Hamiltonian and any sequence of presentations $\W^{\varepsilon_i}(F)$ for that Hamiltonian with $\varepsilon_i\to0$, one has $\lim_{i\to\infty}A^{\varepsilon_i,k}(e)=0$.
 \qed
\end{cor}
\begin{defn}
 We will refer to the smallest $k$ such that $A^{b,k}(e)=0$ as the \textbf{Hochschild length} of $F$.
\end{defn}

\subsection{Statement of the theorem}
To make Theorem \ref{thm:quotient-intro} precise, we need one more notion:

\begin{defn}\label{defn:stop Lag}
 Let $(M,\lambda_M,\ssigma)$ be a pumpkin domain, and let $\sigma\in\ssigma$ be a stop with fiber $F$. Recall from Definitions \ref{defn:stab} and \ref{defn:trivial gluing} that the stabilization $\Sigma F$ of $F$ is the product $F\times\C_1$, and that the trivial gluing $M[\sigma]$ is the gluing $M\glu{\sigma}{\sigma_0}\Sigma F$. Note that interior Lagrangians of $M$ and $\Sigma F$ both give rise to interior Lagrangians of $M[\sigma]$ via inclusion. An interior Lagrangian $L$ in $M[\sigma]$ is said to be \textbf{supported in $\sigma$} if it is isomorphic in $\Ws(M[\sigma])$ to an interior Lagrangian of $\Sigma F$. Let $\Bs\subset\W(M[\sigma])$ and $\Bss\subset\Ws(M[\sigma])$ denote the full subcategories of objects supported in $\sigma$. More generally, for a subset $\ssigma'\subset\ssigma$, let $\Bsps$ denote the full subcategory of $\Wsp(M[\sigma])$ composed of objects supported in $\sigma$.
 
 Although we will generally work in $M[\sigma]$ when dealing with a given stop $\sigma$, we will want to state results in $M$ alone. In this case, we abuse notation and denote again by $\Bss\subset\Ws(M)$ the full subcategory of objects whose image under the quasi-equivalence $\Ws(M)\to\Ws(M[\sigma])$ lie in $\Bss$.
\end{defn}

\begin{prop}\label{prop:B is contractible}
 Let $\ssigma'\subset\ssigma$ be any collection of stops not containing $\sigma$. Then $\Bsps$ is contractible in the sense that, for all $L\in\Bsps$, the unit $e_L\in\hom^0_{\Wsp(M)}(L,L)$ is exact.
\end{prop}
\begin{proof}
 We work in $\widehat{\Sigma F}$, so $\sigma=\sigma_1$ is supported near the negative real axis. Fix $L\subset\widehat{\Sigma F}$. We will start by isotoping $L$ through objects of $\Ws(\Sigma F)$ to some $L'$ supported very close to $\sigma(\hat F\times\R_{\ge0})$. By Proposition \ref{prop:pwfc properties}, $L'$ will be isomorphic to $L$.
 
 Let $f_0\colon\hat\C_1\to\R$ have the form $|z|^4\sin\theta$ outside some large disk $D$ and globally satisfy $\frac{\partial f_0}{\partial y}>1+|z|^2$. Then $f_0$ is very nearly linear in the sense that, outside $D$,\[df_0(\hat Z_{\C_1})\approx f_0\] for $Z_{\C_1}$ the Liouville form on $\C_1$
 \[
  \left(|z|^2+\frac12\epsilon\kappa(|z|)-1\right)xdy-\left(|z|^2+\frac12\epsilon\kappa(|z|)+1\right)ydx
 \]
 from Example \ref{eg:C_n}. Additionally, $f_0$ satisfies
 \begin{equation}\label{eq:sign constraint for pushing Ham}
  \mathrm{sign}\left(\frac{\partial f_0}{\partial x}\right)=\mathrm{sign}(xy)
 \end{equation}
 outside $D$.

 Perturb $f_0$ to a function $f$ which is genuinely linear outside $D$, still satisfies $\frac{\partial f_0}{\partial y}>1+|z|^2$ globally, and satisfies \ref{eq:sign constraint for pushing Ham} (with $f_0$ replaced by $f$) for $|xy|>\epsilon$. Consider now the the pullback $f\circ\pi$ under the projection $\pi\colon F\times\hat\C_1\to\hat\C_1$, where $F\subset\hat F$ is chosen large enough that $L$ is strictly conical outside $F\times\hat\C_1$. Let $\tilde f\colon\widehat{\Sigma F}\to\R$ be an extension of $f\circ\pi$ which becomes linear after a very short neck
 \[
  \bigl([1,1+\epsilon']\times\partial F\bigr)\times\hat\C_1.
 \]
 Note that we began with data which was defined on the horizontal completion, and that passing from $f\circ\pi$ to $\tilde f$ only involves a vertical completion. This implies, for sufficiently small $\epsilon'$, $\tilde f$ still satisfies
 \[
  \frac{\partial\tilde f}{\partial\tilde y}>1+|z|^2
 \]
 over the region
 \[
  F_+:=\bigl(F\cup_{\partial F}[1,1+\epsilon']\times\partial F\bigr)\times\hat\C_1,
 \]
 where $\frac{\partial}{\partial\tilde y}$ is the horizontal lift of $\frac{\partial}{\partial y}$.
 
 Indeed, this is tautological over the interior $F\times\hat\C_1$. On the neck, because the hypersurfaces
 \[
  P_r:=\bigl(\{r\}\times\partial F\bigr)\times\hat\C_1
 \]
 are horizontal, it suffices to perform the calculation on $P_r$. But the Liouville vector field is split, so $\tilde f|_{P_r}$ is given by a scalar times a short-time Liouville pushforward of $\tilde f|_{P_1}=f\circ\pi|_{p_1}$. The quantity $\frac{\partial\tilde f}{\partial\tilde y}$ depends continuously on all of this, and hence the inequality follows. The situation with \eqref{eq:sign constraint for pushing Ham} is similar.
 
 Observe now that the Hamiltonian vector field $X_{\tilde f}$ has strictly negative $x$-component on $F_+$. Moreover, once $x$ is sufficiently negative, $X_{\tilde f}$ is non-expanding in the $y$-direction away from a small neighborhood of the $x$-axis. Because $X_{\tilde f}$ is linear, this implies that its large time flow will eventually move $L$ entirely into the image of $\sigma$. In particular, it will move $L$ past $D_\sigma$. To prevent $L$ from moving past $D_\sigma$, instead set $L'$ be the image of $L$ under the large time flow of $X_{\kappa \tilde f}$, where $\kappa\colon\widehat{\Sigma F}\to\R$ is an appropriate cutoff function supported outside a small neighborhood of $\sigma(\hat F\times\R_{\ge0})$.
 
 To complete the proof, construct a second linear Hamiltonian $\tilde g$ whose Hamiltonian vector field has a positive $y$-component. Flowing $L'$ first by $X_{\tilde f}$ past $D_\sigma$ and then by $X_{\tilde g}$, we obtain a new Lagrangian $L''$ which is isomorphic to $L'$ in any partially wrapped Fukaya category not stopped by $\sigma$. However, the isomorphism $e\in CW_{\ssigma'}^*(L',L'')$ comes from a count of holomorphic caps \cite[Section 3.3]{Sylv_pwfc-big} and has action bounded by the data of $L'$, $\tilde f$, and $\tilde g$. On the other hand, there are no chords from $L'$ to $L''$ for the standard Hamiltonian \eqref{eq:sum H} on the image of $\sigma$, and we can arrange that all chords outside the image of $\sigma$ have very large action. This implies $e$ is the zero morphism, and the proposition follows.
\end{proof}

By the universal property of a quotient category \cite{Drin_quot,Lyub-Manz}, since the image of $\Bss\to\mathcal B_{\ssigma\setminus\{\sigma\}}$ is contractible, the inclusion $\Ws(M)\to\W_{\ssigma\setminus\{\sigma\}}(M)$ factors up to homotopy through the quotient $\Ws(M)/\Bss$:
\begin{equation}\label{eq:SR universal property}
 \begin{tikzcd}[column sep=small]
  {}
  &\Ws(M)/\Bss
  \ar[dashrightarrow]{dr}{\mathcal{SR}}
  &\\
  \Ws(M)
  \ar{ur}
  \ar{rr}
  &&\W_{\ssigma\setminus\{\sigma\}}(M)
 \end{tikzcd}
\end{equation}
where $\mathcal{SR}$ stands for ``stop removal''. With this, the precise statement of Theorem \ref{thm:quotient-intro} is
\begin{thm}\label{thm:quotient}
 Let $(M,\lambda_M,\ssigma)$ be a pumpkin domain, and let $\sigma\in\ssigma$ be a strongly nondegenerate stop. Then the map $\mathcal{SR}\colon\Ws(M)/\Bss\to\W_{\ssigma\setminus\{\sigma\}}(M)$ from \eqref{eq:SR universal property} is fully faithful.
\end{thm}

\subsection{Open-closed maps}\label{sec:OC maps}

The main purpose of the action condition in Definition \ref{defn:strongly nondegenerate} of strong nondegeneracy is to pass in a controlled way from holomorphic curves in the fiber to holomorphic curves in the total space. In particular, we will define a weaker version of a nondegenerate stop using an open-closed string map $\mathcal{OC}\colon CC_*(\Bss)\to SC^*_\ssigma(M[\sigma])$ which counts punctured holomorphic disks, as described in \cite{Abou_gcgfc}. Following Abouzaid, let $\mathcal R^1_d$ be the space of disks with one interior puncture and $d\ge1$ boundary punctures, one of which is distinguished. For $\Sigma\in\mathcal R^1_d$, label the interior puncture by $\zeta_-$ and the boundary punctures $\zeta_1$ through $\zeta_d$, ordered counterclockwise, with the distinguished puncture labeled $\zeta_d$. $\mathcal R^1_d$ has a natural compactification to a manifold with corners $\ol{\mathcal R}^1_d$ whose codimension one faces can be canonically identified with
\begin{equation}\label{eq:OC base faces}
 \coprod_{\substack{2\le k\le d\\1\le i\le k}}\ol{\mathcal R}^1_{d+1-k}\times\ol{\mathcal R}^{k+1,i}
 \quad\amalg\;\coprod_{\substack{2\le k\le d-1\\1\le i\le d-k}}\ol{\mathcal R}^1_{d+1-k}\times\ol{\mathcal R}^{k+1}.
\end{equation}
Here, $\ol{\mathcal R}^{k+1,i}$ is diffeomorphic to $\ol{\mathcal R}^{k+1}$, but if $\Sigma^d\in\ol{\mathcal R}^{k+1}$, then the corresponding point of $\ol{\mathcal R}^{k+1,i}$ is $\Sigma^d$ with the additional datum that $\zeta_i\in\Sigma^d$ is distinguished. In other words, it is the space of disks with one negative puncture, $d$ positive punctures, and such that the $i$th puncture is considered special. The first term in \eqref{eq:OC base faces} corresponds then to a collection of punctures which includes $\zeta_d$ colliding, while the second corresponds to some other collection colliding. In this case, the additional index $i$ keeps track of where the collision occurred.

A collection of ends for $\Sigma\in\mathcal R^1_d$, making it into a punctured Riemann surface with boundary, ends, and cylinders, consists of a positive strip-like end $\epsilon_i$ for each boundary puncture $\zeta_i$, along with a negative cylindrical end $\epsilon_-$ at $\zeta_-$. In this case, we ask that $\epsilon_-$ has a very special form. Specifically, in the holomorphic coordinates on $\Sigma$ where $\mathrm{int}(\Sigma)=\{z\in\C\mid0<|z|<1\}$ and $\zeta_d=1$, we require that
\begin{equation}\label{eq:OC neg end alignment}
\epsilon_-(s,t)=ae^{2\pi(s+it)}\qquad\text{with $a\in\R$ positive.}
\end{equation}
for some positive number $a\in\R$. A \textbf{universal family of ends} for $\mathcal R^1_d$ consists of a collection of ends on each $\Sigma\in\mathcal R^1_d$ for every $d$, such that near the boundary of $\ol{\mathcal R}^1_d$ it agrees up to a rotation of $\epsilon_-$ with the collection induced by gluing. This rotation correction is unavoidable, since boundary components $\ol{\mathcal R}^1_{d+1-k}\times\ol{\mathcal R}^{k+1,i}$ have the same ends for all $i$, so that without rotation at most one of them could glue to a configuration which satisfies \eqref{eq:OC neg end alignment}. As a concrete example, consider the boundary
\[
 \partial\ol{\mathcal R}^1_2\;=\;\mathcal R^1_1\times\mathcal R^{3,1}\,\amalg\,\mathcal R^1_1\times\mathcal R^{3,2}.
\]
Because we have not chosen separate Floer data for $\mathcal R^{3,1}$ and $\mathcal R^{3,2}$, the negative strip-like end can be aligned with at most one of the punctures. As a result, after any naive finite gluing on each end with the same gluing parameter, we will have two punctured disks with isomorphic Floer data but different labelings of the boundary marked points. Because the cylindrical end at $\epsilon_-$ can be aligned with at most one boundary puncture, one of these naive gluings will fail to satisfy \eqref{eq:OC neg end alignment}.

However, because we are using an exponential gluing profile the magnitude of the misalignment vanishes to infinite order at the boundary, and hence the family of strip-like ends extends smoothly to $\ol{\mathcal R}^1_d$. One sees as with $\mathcal R^{d+1}$ that universal families of ends for $\mathcal R^1_d$ exist, and we fix one once and for all.

\begin{rmk}
 An alternative to these rotation corrections would be to choose separate Floer data for the $\mathcal R^{d+1,i}$ for different $i$, giving a bimodule quasi-isomorphic but not equal to the diagonal bimodule. This would be equally effective for our purpose but would unnecessarily increase the number of moduli spaces we need to construct.
\end{rmk}

A \textbf{universal and conformally consistent} choice of Floer data for $\mathcal R^1_d$ consists of, for all $d\ge1$, a Floer datum $(\beta,H,P,J,\tau_E)$ for each $\Sigma\in\mathcal R^1_d$ varying smoothly over $\mathcal R^1_d$, and such that near $\partial\ol{\mathcal R}^1_d$ it agrees to infinite order with the conformal class of not-quite Floer datum determined by gluing. We say not-quite due to the rotation corrections for the strip-like ends, which among other things cause the glued datum to not be a Floer datum in the above sense. Denote by $\mathcal K^{\mathcal{OC}}(M[\sigma])$ the space of universal and conformally consistent choices of Floer data for $\mathcal R^1_d$.

Given $\mathbf K\subset\mathcal K^{\mathcal OC}(M[\sigma])$, we can consider the resulting spaces of holomorphic curves. Given a collection of Lagrangian labels $L_i$ and asymptotic ends
\[
 \gamma_i\in\mathscr X(L_i,L_{i+1})\quad\text{ and }\quad x_-\in\mathscr X(H_t),
\]
we are interested in the space
\[
 \mathcal R^1_d(\gamma_d,\dotsc,\gamma_i;x_-).
\]
This consists of all maps $u\colon\Sigma\to\widehat{M[\sigma]}$ for $\Sigma\in\mathcal R^1_d$ satisfying \eqref{eq:gen Floer w/ cylinders} with $u(E_i)\subset (\phi^{\tau_{E}})^*L_i$, $u(\zeta_i)=(\phi^{\tau_i})^*\gamma_i$, and $u(\zeta_-)=(\phi^{\tau_-})^*x_-$.

\begin{lemma}\label{lem:OC transversality}
 For generic $\mathbf K\in\mathcal K^{\mathcal{OC}}(M[\sigma])$, all moduli spaces $\mathcal R^1_d(\gamma_d,\dotsc,\gamma_i;x_-)$ are transversely cut out of dimension $\deg(x_-)-\sum_{i=1}^d\deg(\gamma_i)+d-n-1$, where $n$ is half the dimension of $M$. In this case, Gromov compactness applies, and the codimension $1$ boundary strata are in natural bijection with
  \begin{equation}\label{eq:OC boundary strata}
  \begin{gathered}
   \coprod_{\substack{2\le k\le d\\1\le i\le k\\ \tild\gamma\in\mathscr X(L_{d+1-i},L_{k+1-i})}}\hspace{-6mm}\mathcal R^1_{d+1-k}(\tild\gamma,\gamma_{d-i},\dotsc,\gamma_{1+k-i};x_-)\times\mathcal R^{k+1}(\gamma_{k-i},\dotsc,\gamma_{d+1-i};\tild\gamma)\\
   \hspace{2cm}\amalg\coprod_{\tild\gamma\in\mathscr X(L_d,L_1)}\mathcal R^1_d(\tild\gamma,\gamma_{d-1},\dotsc,\gamma_1;x_-)\times\mathcal R(\gamma_d;\tild\gamma)\\
 \amalg\hspace{-2mm}\coprod_{\substack{2\le k\le d-1\\1\le i\le d-k\\ \tild\gamma\in\mathscr X(L_i,L_{i+k})}}\mathcal R^1_{d+1-k}(\gamma_d,\dotsc,\gamma_{i+k},\tild\gamma,\gamma_{i-1}\dotsc,\gamma_1;x_-)\times\mathcal R^{k+1}(\gamma_{i+k-1},\dotsc,\gamma_i;\tild\gamma)\\
 \amalg\coprod_{\substack{1\le i<d\\ \tild\gamma\in\mathscr X(L_i,L_{i+1})}}\mathcal R^1_d(\gamma_d,\dotsc,\gamma_{i+1},\tild\gamma,\gamma_{i-1}\dotsc,\gamma_1;x_-)\times\mathcal R(\gamma_i;\tild\gamma)\\
 \hspace{2cm}\amalg\coprod_{\tild x\in\mathscr X(H_t)}\mathcal Q(\tild x;x_-)\times\mathcal R^1_d(\gamma_d,\dotsc,\gamma_1;\tild x).
  \end{gathered}
  \end{equation}
\end{lemma}
\begin{proof}[Remark on proof]
 To address the compactness, recall from Definition \ref{defn:Floer datum w/ cylinders} that the perturbing Hamiltonian has no effect outside the 2-shifts of the cylinders. On the other hand, $J$ is of rescaled contact type in the 2-shifts of the cylinders. Thus, we can separate the compactness problem into the 2-shifts of the cylinders and their complement.
 
 In the first case, the $C^0$ estimates in \cite{Gana_scdwc} give a bound on how far elements $u\in\mathcal R^1_d(\gamma_d,\dotsc,\gamma_i;\gamma_-)$ can protrude into the symplectization. Everywhere else, Lemma \ref{lem:maximum principle} applies, even with varying compatibility between $J$ and $H$. Together, these show that the image of any $u$ is constrained to lie in a compact set depending only on $\mathbf K_\Delta$ and the ends $\gamma_i$ and $\gamma_-$.
\end{proof}

Define $\mathcal{OC}\colon CC_*(\Bss)\to SC^{*+n}(M[\sigma])$ by
\[
 \mathcal{OC}(\gamma_d\otimes\dotsm\otimes\gamma_1)=\hspace{-15mm}\sum_{\substack{x\in\mathscr X(H_t)\\ \deg(x)=\sum_{i=1}^d\deg(\gamma_i)+n+1-d}}\hspace{-15mm}\#\mathcal R^1_d(\gamma_d,\dotsc,\gamma_i;x)\cdot x.
\]
The boundary strata in \eqref{eq:OC boundary strata} tell us that $\mathcal{OC}$ is a chain map. Further, arguing as in Lemma \ref{lem:SH intersections filter} gives
\begin{lemma}\label{lem:OC intersections filter}
 The image of $\mathcal{OC}$ lies in $SC_\ssigma^*(M[\sigma])$.
 \qed
\end{lemma}

\begin{defn}\label{defn:nondegenerate stop}
 A stop $\sigma\in\ssigma$ is \textbf{weakly nondegenerate} if, for some choice of Floer data for symplectic cohomology, the Fukaya category and the open-closed map, there is a Hochschild cycle $y\in CC_{1-n}(\Bss)$ such that $\mathcal{OC}(y)=f_\sigma$, where $f_\sigma\in SC_\ssigma^1(M[\sigma])$ is a \textbf{saddle unit} of $\sigma$ as described below.
\end{defn}

Morally speaking, a saddle unit is any cocycle which lives in the central fiber of $\Sigma F\subset M[\sigma]$ and represents the unit of $SH^*(F)$ when thought of as a chain in $SC^0(F)$. The drop in degree from 1 to 0 comes from the fact that the central fiber lives at a saddle point of the Liouville vector field for $\C_1$, which translates to an index 1 Morse critical point for nice choices of compatible Hamiltonian. However, such a cocycle is often exact in $SC_\ssigma^*(M[\sigma])$, so the careful definition of $f_\sigma$ instead involves a count of holomorphic caps.

Concretely, let $\Sigma$ be $\C$ equipped with the negative cylindrical end $\epsilon^f$ asymptotic to $\infty$ given by
\[
 \epsilon^f(s,t)=e^{-2\pi(s+it)}
\]
Let $\mathcal K^\C(M[\sigma])$ denote the space of Floer data on $\Sigma$. Given a Floer datum $K^f\in\mathcal K^\C(M[\sigma])$ and an orbit $x\in\mathscr X(H_t)$, we are interested in the resulting moduli space space $\mathcal C(x)$. This is the space of all maps $u\colon\Sigma\to\widehat{M[\sigma]}$ satisfying \eqref{eq:gen Floer w/ cylinders} and
\[
 \lim_{s\to-\infty}u(\epsilon^f(s,t))=(\phi^\tau)^*x(t),
\]
where $\tau$ is the conformal factor $K^f$ assigns to $\infty\in\ol\Sigma$, and for which
\[
 u(0)\in Y_\sigma,
\]
where $Y_\sigma\subset\hat M[\sigma]$ is the hypersurface which comes from $\hat F\times i\R\subset\Sigma F$. The last condition is the interesting one. Indeed, that is the only place where the stop $\sigma$ comes into the definition of $f_\sigma$. Without it the holomorphic caps would not be forced to avoid any $D_\sigma$, and we would just obtain the unit of symplectic cohomology.
\begin{lemma}\label{lem:f_sigma transversality}
 For generic $K^f\in\mathcal K^\C(M[\sigma])$, all moduli spaces $\mathcal C(x)$ are transversely cut out of dimension $\deg(x)-1$. In this case, Gromov compactness applies, and the codimension $1$  boundary strata of the compactification are in natural bijection with
   \[
    \partial\ol{\mathcal C}(x)=\coprod_{\tild x\in\mathscr X(H_t)}\mathcal Q(\tild x;x)\times\mathcal C(\tild x).
   \]
  \qed
\end{lemma}
Note that if $\deg(x)=2$, then $\tild x$ can only appear when it has degree 1.

\begin{defn}\label{defn:saddle unit}
 A \textbf{saddle unit} of $\sigma$ is any chain
 \[
  f_\sigma=\sum_{\substack{x\in\mathscr X(H_t)\\ \deg(x)=1}}\#\mathcal C(x)\cdot x
 \]
 obtained from a regular Floer datum $K^f\subset\mathcal K^\C(M[\sigma])$. It follows from Lemma \ref{lem:f_sigma transversality} and positivity of intersections that such a chain is in fact a closed element of $SC_\ssigma^1(M[\sigma])$.
\end{defn}

The rest of Section \ref{ch:nondeg stops} is devoted to proving
\begin{prop}\label{prop:nondegeneracy}
 Every strongly nondegenerate stop is weakly nondegenerate.
\end{prop}

\subsection{Stabilizations}\label{sec:stabilizations}
The proof of Proposition \ref{prop:nondegeneracy} amounts to making a very careful choice of Floer data. To do this, we first consider the case of a stabilization, and there it will be helpful to work with a special Liouville form on $\C_1$. For this, let $\lambda_{\C_1}^0$ be the Liouville form on $\C_1$ constructed in Example \ref{eg:C_n}. Near the origin, it is given by
\[
 \lambda_{\C_1}^0=\left(-1+\frac12\epsilon+x^2+y^2\right)xdy-\left(1+\frac12\epsilon+x^2+y^2\right)ydx.
\]
Let $\kappa$ be the cutoff function of Example \ref{eg:C_n}, so that $\kappa$ is radially invariant, equals 1 when $|z|<\frac14$, and equals 0 when $|z|>\frac12$. The Liouville form $\lambda_{\C_1}=\lambda_{\C_1}^0+d(xy\kappa(|z|))$ given near zero by
\[
 \lambda_{\C_1}=\left(\frac12\epsilon+x^2+y^2\right)xdy-\left(\frac12\epsilon+x^2+y^2\right)ydx
\]
is still invariant under the involution $z\mapsto-z$, and $i\R$ is still invariant under its Liouville vector field $\hat\Z_{\C_1}$. Moreover, it has two very nice properties which $\lambda_{\C_1}^0$ lacks. First, its Liouville vector field agrees to second order at the origin with $Z_{std}=\frac12(x\partial_x+y\partial_y)$. Second, it receives a non-proper Liouville embedding
\[
 I\colon(\C,\lambda_{std})\to(\C_1,\lambda_{\C_1})
\]
which is $\Z/2$-equivariant, strictly preserves the Liouville forms, sends $\R$ into $\R$ and $i\R$ onto $i\R$, and whose image avoids the images of the stops. See Figure \ref{fig:stab vector fields}.

\begin{figure}
 \def\svgwidth{16cm}
 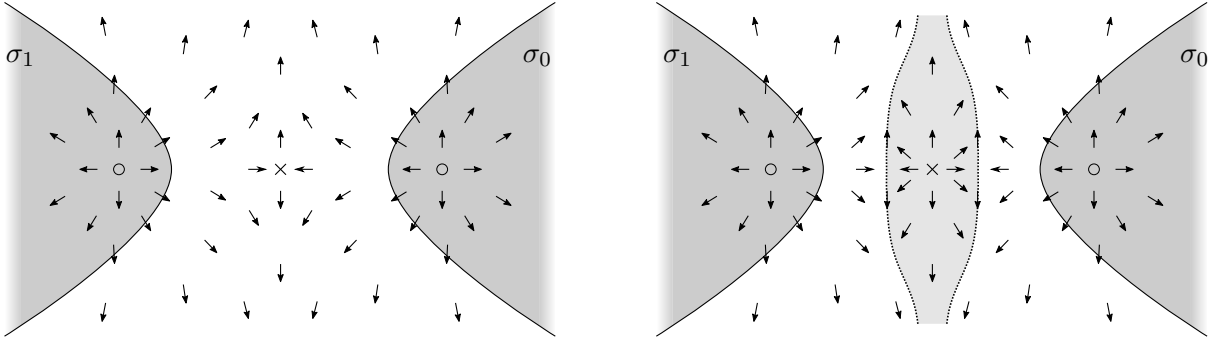
 \caption{The Liouville vector fields for $\lambda_{\C_1}^0$ and $\lambda_{\C_1}$, respectively. The central shaded region on the second picture is the image of the embedding $I$.}\label{fig:stab vector fields}
\end{figure}

With this, let $F$ be a strongly nondegenerate Liouville domain, and consider $N=\Sigma F$ with Liouville form $\lambda_N=\lambda_F+\lambda_{\C_1}$. Denote by $\Phi\colon\hat N\to\hat M$ the involution which is identity in the $\hat F$ component and rotation by $\pi$ in the $\C_1$ component. In particular, $\Phi$ exchanges the two stops in $N$. For $\varepsilon>0$ and an object $L\in\W^\varepsilon(F)$, we can build $\Phi$-invariant Lagrangians in $N$ as follows.

Let $f$ be the unique compactly supported primitive of $\hat\lambda_F|_L$, which exists by the conditions on Lagrangians in $\W^\varepsilon(F)$. Choose $\boldsymbol\tau\colon\R\to\R$ to be a smooth, \emph{even} function which vanishes in a neighborhood of zero and, outside a compact set, is of the form
\[
 \boldsymbol\tau(y)=2\log|y|+\log B
\]
for some constant $B>0$. Write $\phi^\tau$ for the time $\tau$ flow of $\hat Z_N$. Then the map
\[
 (p,y)\mapsto\left(\phi^{\boldsymbol\tau(y)}(p),\left(-f(p)\frac{d}{dy}e^{\boldsymbol\tau(y)},y\right)\right)
\]
is a Lagrangian embedding $\Sigma^{\boldsymbol\tau}_L\colon L\times\R\into(\hat F\times\C,\hat\lambda_F+\hat\lambda_{\C,std})$ whose image is $\Phi$-invariant globally and $\hat Z_{F\times\C}$-invariant outside a compact set. Note that, for $\boldsymbol\tau$ not too wild on the compact part of its domain, we have
\[
 \left|\frac{d}{dy}e^{\boldsymbol\tau(y)}\right|<2B|y|,
\]
which together with the variance condition on $f$ implies that
\begin{equation}
 |x|<2\varepsilon B|y|.
\end{equation}
Let $C>0$ be such that $|I(z)|>1$ for $z\in i\R$ with $|z|>C$, and choose once and for all such a $\boldsymbol\tau$ which additionally vanishes on $[-C,C]$. Then $(Id,I)\circ\Sigma^{\boldsymbol\tau}_L\colon L\times\R\into N$ is a Lagrangian embedding which is split wherever $(Id,I)$ fails to preserve the Liouville form. This implies that its image is a conical Lagrangian in $N$, and we denote it by $\Sigma L$.

In the remainder of this section, we construct Hamiltonians on $N$ which are simultaneously well-adapted to the splitting $F\times\C_1$ and close to being normalizing Hamiltonians. For this, let $H^s$ be a compatible Hamiltonian on $F$ satisfying $dH^s(\hat Z_F)\le2H^s$ globally and $H^s<0.1$ whenever $dH^s(\hat Z_F)\ne2H^s$. Let $f_0\colon\R_{\ge0}\to\R_{\ge0}$ be a smooth, increasing function which satisfies
\begin{enumerate}
 \item There is some $a>0$ such that $f_0(x)=ax$ for $x$ close to zero,
 \item $f_0'(x)\le a$ for $x\le\frac{25}{16}$,
 \item $f_0'(1)=0$,
 \item $f_0(x)=x^2$ for $x$ sufficiently large,
 \item $f_0(2)<0.1$, and
 \item $f_0'(x)\ge2x$ for $x\ge2$.
\end{enumerate}
See Figure \ref{fig:graph for stab ham}.

\begin{figure}
 \def\svgwidth{10cm}
 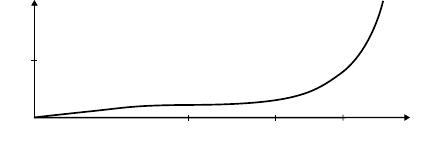
 \caption{}\label{fig:graph for stab ham}
\end{figure}

Given $\alpha>0$, write $H^\alpha\colon N\to\R$ for $\frac1\alpha$ times the compatible Hamiltonian given by Equation \eqref{eq:sum H} with $g=\alpha^2 H^s$ and $f(z)=\alpha f_0(|z^2-1|^2)$. Setting $K=(\alpha H^s)^{-1}((-\infty,2])\subset\hat F$, we have
\begin{equation}\label{eq:H split over disk}
 H^\alpha|_{K\times D_{0.9}}(p,z)=\alpha H^s(p)+f_0(|z^2-1|^2)\quad\mbox{for all $\alpha$ sufficiently small}.
\end{equation}
In this case, $X_{H^\alpha}$ has the crucial property
\begin{subequations}\label{cd:saddle X_H dirs}
 \begin{equation}\label{cd:saddle X_H x-comp}
  dx(X_{H^\alpha})\mbox{ has the opposite sign as }y\mbox{ and}
 \end{equation}
 \begin{equation}\label{cd:saddle X_H y-comp}
  dy(X_{H^\alpha})\mbox{ has the opposite sign as }x,
 \end{equation}
\end{subequations}
where $z=x+iy$ is the coordinate on $D_{0.9}$. 

Let $\kappa\colon\R\to[0,1]$ be a nondecreasing cutoff function with $\kappa(x)=0$ for $x\le1$ and $\kappa(x)=1$ for $x\ge2$. Such a $\kappa$ satisfies $\kappa(f_0(|z^2-1|^2))=0$ for $z\in D_{\frac12}$ and $\kappa(f_0(|z^2-1|^2))=1$ for $z\in D_{0.9}$ with $|\Im(z)|>0.8$. Define $H^N\colon N\to\R$ by
\begin{equation}\label{eq:H^N}
 H^N=cH^\alpha\kappa(H^\alpha).
\end{equation}
where $c>0$ is a small constant which ensures that $H^N$ is $C^2$-small on the interior of $N$. This is not quite a normalizing Hamiltonian on $N$, but it has some very nice properties.

\begin{lemma}\label{lem:H^N superlinearity}
 The following hold.
 \begin{enumerate}
  \item $H^N$ is $\Phi$-invariant on $K\times D_{0.9}$, and there it satisfies \eqref{cd:saddle X_H dirs} whenever $dH^N\ne0$.
  \item $H^N\ge2c$ on $K\times(D_{0.9}\cap\{|\Im(z)|>0.8\})$ and on $\partial K\times D_{0.9}$.
  \item\label{cd:H^N hides the core} $H^N=0$ in a neighborhood of $\{(p,z)\in\hat N\mid\hat Z_N(p,z)=0\}$.
  \item The restriction $H^N|_{\hat F\times\{0\}}$ is a normalizing Hamiltonian on $F$.
  \item Along $K\times\{0\}$, the vector field $X_{H^N}$ vanishes in the $\C_1$-directions to at least third order.
  \item For sufficiently small $a$ in the definition of $f_0$ and sufficiently small $\alpha$ in the definition of $H^\alpha$, the following holds. For all $\tau\ge1$, $H^N_\tau\ge\tau^{3/4}H^N$ everywhere.
 \end{enumerate}
\end{lemma}
\begin{proof}
 The first five are immediate consequences of the construction. The sixth follows from a sequence of straightforward but tedious calculations, see \cite{Sylv_pwfc-big}.
\end{proof}

For the rest of Section \ref{ch:nondeg stops}, fix $H^\alpha$ and $H^N$ satisfying the conclusion of Lemma \ref{lem:H^N superlinearity}.

\subsection{Collapsing the energy}\label{sec:collapsing energy}

Moving to the global situation, let $M$ be a pumpkin domain of the form $M'[\sigma]$, where $\sigma$ is strongly nondegenerate with fiber $F$. Choose the Liouville form $\hat\lambda_M$ on $M$ to strictly agree with $\hat\lambda_N$ on the gluing region. We will extend the Hamiltonian $H^N$ to a Hamiltonian $H^M$ and define a family of spaces $\mathcal H_\varepsilon(M)$ as if $H^M$ were a normalizing Hamiltonian. With this, we will be able to control the open-closed maps by degenerating $\varepsilon$.

To begin, assume $\partial N\subset\hat N$ is chosen to be large enough to easily contain both stops, e.g. $\partial N=(H^N)^{-1}(100)$. Let $U\subset\partial N$ be the intersection
\[
 \partial N\cap\sigma_0(\hat F\times\{\Re(z)>0\}),
\]
and let $P_0\subset U$ be the intersection
\[
 \partial N\cap\sigma_0(\hat F\times\{\Re(z)>\rho\}),
\]
where $\rho$ is the width of the stop $\sigma$ on $M'$. Note that $\partial N\setminus P_0$ embeds canonically into $M$, and that this embedding strictly intertwines the contact 1-forms. Let $P\subset U$ be a compact subset containing $P_0$ such that the positive flow of $\hat Z_M$ is proper on $\partial N\setminus P$.

Pick a smooth function $g\colon\partial N\to[0,1]$ which equals $0$ on a neighborhood $P$ and $1$ outside $U$. This gives rise to a function $H^N_g=H^N\tilde g$, where $\tilde g$ is the $\hat Z_N$-invariant extension of $g$. By \eqref{cd:H^N hides the core} in Lemma \ref{lem:H^N superlinearity}, $H^N_g$ is a smooth function on $\hat N$. From the construction of $P$, we see that it extends by $0$ to a smooth function on $\hat M$, which we continue to call $H^N_g$.

Analogously, pick a normalizing Hamiltonian $H^{M'}$ on $M'$, and cut it off with a function $g'$ so that it extends to a smooth function $H^{M'}_{g'}$ on $M$. By choosing the analog of $P$ appropriately, we may assume that $H^{M'}_{g'}$ vanishes on $K\times D_{0.9}$. With this, define
\begin{equation}
 H^M:=H^N_g+H^{M'}_{g'}.
\end{equation}
It has the following properties.
\begin{enumerate}
 \item $H^M$ is proper, nonnegative, and satisfies $dH^M(\hat Z_M)=2H^M$ outside a compact set.
 \item For all $\tau\ge1$, $H^M_\tau\ge\tau^{3/4}H^M$ everywhere.
 \item $H^M=0$ in a neighborhood of $\{p\in\hat M\mid\hat Z_M(p)=0\}$.
 \item $H^M$ agrees with $H^N$ on $K\times D_{0.9}$ and $\hat F\times\{0\}$. In particular, $H^M|_{\hat F\times\{0\}}$ is a normalizing Hamiltonian on $F$.
 \item After possibly replacing all Hamiltonians with $\epsilon$-multiples of themselves, all nonconstant time 1 orbits of $X_{H^M}$ live in the conical region where $dH^M(\hat Z_M)=2H^M$, and moreover none of these occurs in $K\times\{0\}$.
\end{enumerate}

Although it vanishes sometimes, $H^M$ is close to being a compatible Hamiltonian, and so we obtain spaces $\mathcal J(M,H)$ and $\mathcal J^{S^1}_{resc}(M,H)$ of adapted and rescaled almost complex structures. Choose almost complex structures $J^M\in\mathcal J(M,H)$ and $J^M_{resc}\in\mathcal J^{S^1}_{resc}(M,H)$ which on $K\times D_{0.9}$ are split and take the form $(J^F,J_{std})$. We will treat $(H^M,J^M)$ and $(H^M,J^M_{resc})$ as normalizing Floer data for Lagrangian Floer cohomology and symplectic cohomology, respectively.

In that spirit, choose universal and conformally consistent families of Floer data $\mathbf K^\mu$ for the $A_\infty$ structure and $\mathbf K^{\mathcal{OC}}$ for the open-closed maps, along with a Floer datum $K^\C$ for the saddle unit, such that the following hold.
\begin{enumerate}
 \item All Hamiltonians agree up to rescaling with $H^M$ and satisfy
 \[
 d^\Sigma H\wedge\beta\le0
 \]
 globally.
 \item After the same rescaling as with the Hamiltonians, all almost complex structures are split on $K\times D_{0.9}$, and there they take the form $(J^F,J_{std})$ for some potentially domain-dependent $J^F$.
 \item For every domain $\Sigma$ with nonempty boundary, each boundary component $E$ of $\Sigma$ contains disjoint open intervals $E^+$ and $E^-$ such that $\tau_E$ is nonincreasing outside $E^+$ and nondecreasing outside $E^-$.
 \item The rescaling factors are all at least $1$ and satisfy
  \begin{equation}\label{eq:rescaling bound for A_infty disks}
    e^{1-d}\tau_0<\tau_E<\tau_0
  \end{equation}
 for the $A_\infty$ disks $\Sigma^{d+1}\in\mathcal R^{d+1}$,
  \begin{equation}\label{eq:rescaling bound for OC annuli}
    e^{1-d}\tau_-<\tau_E<\tau_-
  \end{equation}
 for the open-closed annuli $\Sigma^1_d\in\mathcal R^1_d$, and
  \[
   \tau=1
  \]
 for the saddle unit.
 \item The 1-form $\beta^\C$ for the saddle unit satisfies $\beta^\C=f(r)d\theta$, where $f$ is a nonincreasing function which equals $0$ near zero and $\frac{-1}{2\pi}$ outside a compact set.
\end{enumerate}
Conditions \eqref{eq:rescaling bound for A_infty disks} and \eqref{eq:rescaling bound for OC annuli} are more difficult than before, but they can still be achieved. In this case the minimal coefficient is $(2^{4/3})^{d-1}$, but $e$ is bigger than $2^{4/3}$, so we're fine. We will use these Floer data as normalizations rather than to define holomorphic curve counts directly, so the complete lack of nondegeneracy of critical points and regularity of moduli spaces is not a problem.

From here, choose for all ${\mathbf m}\in\N$ a nondegenerate compatible Hamiltonian and regular almost complex structure for each pair of Lagrangians, along with a regular Floer datum $(H_t^{\mathbf m},J_t^{\mathbf m})$ for the symplectic cochain complex. Extend these to universal and conformally consistent, regular families of Floer data $\mathbf K^{\mu,{\mathbf m}}$ for the $A_\infty$ operations, $\mathbf K^{\mathcal{OC},{\mathbf m}}$ for the open-closed maps (compatibly with $\mathbf K^{\mu,{\mathbf m}}$), and $K^{\C,{\mathbf m}}$ for the saddle unit, in such a way that the following hold.
\begin{enumerate}
 \item $\mathbf K^{\bullet,{\mathbf m}}$ is $\frac1{2{\mathbf m}}$-close to $\mathbf K^\bullet$ in $C^{\mathbf m}$. Here, we use a notion of closeness which allows small changes at infinity, as in \eqref{cd:close to normalizing Hamiltonian}.
 \item For fixed inputs, the resulting perturbed holomorphic disks satisfy a uniform-in-${\mathbf m}$ maximum principle. This can be achieved by only allowing perturbations near infinity where $d^\Sigma H$ is supported in the set $\{d\beta\le c\cdot dvol\}$ for some fixed negative number $c$.
 \item For every domain $\Sigma$ with nonempty boundary, each boundary component $E$ of $\Sigma$ contains disjoint open intervals $E^+$ and $E^-$ such that $\tau_E$ is nonincreasing outside $E^+$ and nondecreasing outside $E^-$. Additionally, the $\tau_E$ all satisfy \eqref{eq:rescaling bound for A_infty disks} or \eqref{eq:rescaling bound for OC annuli} up to an additive error of $\frac1{\mathbf m}$.
 \item For all
 \[
  \Sigma\in\bigcup_{d=1}^k\ol{\mathcal R}^{d+1}\,\cup\,\bigcup_{d=1}^k\ol{\mathcal R}^1_d
 \]
 where $k$ is the Hochschild length of $F$, there is a function $\tau^{\mathbf m}\colon\Sigma\to[1,\infty)$ such that
 \begin{enumerate}
  \item $\tau^{\mathbf m}$ agrees on each strip or cylinder with the rescaling factor associated to that strip or cylinder, and elsewhere it extends the function $\tau_E$.
  \item The un-rescaled family of almost complex structures $J^{\bullet,{\mathbf m}}_u$, where $u=\frac1{\tau^{\mathbf m}}$, is split on $K\times D_{0.9}$ and agrees with $J_{std}$ in the second component.
  \item The un-rescaled family of Hamiltonians $H^{\bullet,{\mathbf m}}_u$ is $\Phi$-invariant on $K\times D_{0.9}$, and there it satisfies \eqref{cd:saddle X_H x-comp} whenever $|y|>0.1$ and \eqref{cd:saddle X_H y-comp} whenever $|x|>0.1$.
  \item All critical points of $H^{\bullet,{\mathbf m}}_u|_{Y_\sigma}$ lie in the central fiber $K\times\{0\}$. Moreover, along the central fiber $H^{\bullet,{\mathbf m}}_u$ agrees to second order with
  \[
   H^{\bullet,{\mathbf m}}_u|_{K\times\{0\}}+a(y^2-x^2)
  \]
  for some $a>0$.
 \end{enumerate}
 \item For $\Sigma=\C$, the above holds with $u\equiv1$. Further, the 1-form $\beta^{\C,{\mathbf m}}$ satisfies $\beta^{\C,{\mathbf m}}=f_{\mathbf m}(r)d\theta$, where $f_{\mathbf m}$ is a nonincreasing function which equals $0$ near zero and $\frac{-1}{2\pi}$ outside a compact set.
 \item The interior action leaks satisfy the bounds
 \[
  \int_{\Sigma^{d+1}}\max_{\hat M}(d^\Sigma H^{\mu,{\mathbf m}}\wedge\beta^{\mu,{\mathbf m}})<\frac{2de^{d-1}}{\mathbf m\tau_0}
 \]
 for $\Sigma^{d+1}\in\mathcal R^{d+1}$,
 \[
  \int_{\Sigma^1_d}\max_{\hat M}(d^\Sigma H^{\mathcal{OC},{\mathbf m}}\wedge\beta^{\mathcal{OC},{\mathbf m}})<\frac{2de^{d-1}+1}{\mathbf m\tau_0}
 \]
 for $\Sigma^1_d\in\mathcal R^1_d$, and
 \[
  \int_\Sigma\max_{\hat M}(d^\Sigma H^{\C,{\mathbf m}}\wedge\beta^{\C,{\mathbf m}})<\frac1{\mathbf m}.
 \]
 for $\Sigma=\C$.
\end{enumerate}

\begin{rmk}\label{rmk:Z/2 transversality}
 For perturbed holomorphic curves living entirely in $K\times D_{0.9}$, one needs access to $\Z/2$ transversality statements. This is straightforward in the stable case outside the central fiber, and in the unstable case one can argue as in \cite{Khov-Seid}. For $A_\infty$ disks lying entirely in the central fiber, the situation is essentially identical to that in Section 14 of \cite{Seid_plt}, where Seidel's topological argument guarantees transversality. For domains with negative cylindrical ends, however, the topological argument fails. We address the regularity of the moduli spaces of such curves in Section \ref{sec:transversality for nondegeneracy}, where it turns out that our argument only guarantees transversality when ${\mathbf m}$ is large and the full family of domains is compact. This latter condition is why we only ask $H$ to be $\Phi$-invariant up to $d=k$.
\end{rmk}

We are now prepared to begin proving Proposition \ref{prop:nondegeneracy}. To do so, we will want to show that for large ${\mathbf m}$, essentially all relevant holomorphic curves live in the central fiber. In principle one could argue this as a strict Morse degeneration, but for our purposes the following lemma suffices.

\begin{lemma}\label{lem:weak Morse approx}
 Let $\Sigma$ be a finite strip or cylinder $[a,b]\times C$, where $C$ is either $[0,1]$ or $S^1$. If $\Sigma$ is a strip, equip it with Lagrangian labels $\Sigma L_0$ along $[a,b]\times\{0\}$ and $\Sigma L_1$ along $[a,b]\times\{1\}$, where $L_0$ and $L_1$ are objects of $\W^\varepsilon(F)$ for some $\varepsilon>0$. Equip $\Sigma$ with a Floer datum $(H^{\mathbf m},J^{\mathbf m})$ coming from the above choices for some finite ${\mathbf m}$, and suppose $u\colon\Sigma\to\hat M$ is a perturbed holomorphic map whose image lies in $K\times D_{0.9}$. Then
 \begin{equation}\label{eq:Morse approx for y}
  \pm(y\circ u)>0.1\text{ globally }\implies\pm\left(\int_Cy(u(a,t))dt-\int_Cy(u(b,t))dt\right)>0,
 \end{equation}
 $(x,y)$ is the coordinate function on $D_{0.9}$. If $C=S^1$, then we also have
 \begin{equation}\label{eq:Morse approx for x}
  \pm(x\circ u)>0.1\text{ globally }\implies\mp\left(\int_Cx(u(a,t))dt-\int_Cx(u(b,t))dt\right)>0.
 \end{equation}
\end{lemma}
\begin{proof}
 The proof is identical in all four cases, so we restrict to the case $y\circ u>0.1$. Then
 \begin{align*}
  \left(\int_Cy(u(a,t))dt-\int_Cy(u(b,t))dt\right)&=\int_\Sigma-d(y\circ u)(\partial_s)ds\wedge dt\\
  &=\int_\Sigma-dy\circ du\wedge dt\\
  &=\int_\Sigma-dy\circ(du-X_{H^{\mathbf m}}dt)\wedge dt\\
  &=\int_\Sigma dy\circ J^{\mathbf m}\circ(du-X_{H^{\mathbf m}}dt)\circ j\wedge dt\\
  \mbox{\footnotesize (using that $J^{\mathbf m}$ is split)}\quad&=\int_\Sigma-dx\circ(du-X_{H^{\mathbf m}}dt)\circ j\wedge dt\\
  &=\int_\Sigma-d(x\circ u)(\partial_t)ds\wedge dt+\int_\Sigma dx(X_{H^{\mathbf m}})ds\wedge dt\\
  \mbox{\footnotesize (using that $H^{\mathbf m}$ satisfies \eqref{cd:saddle X_H x-comp})}\quad&<\int_\Sigma-d(x\circ u)(\partial_t)ds\wedge dt\\
  &=-\int_a^b\int_C\frac{\partial(x\circ u(s,\cdot))}{\partial t}dt\,ds
  =-\int_a^b0\,ds=0.
 \end{align*}
 In the last step, when $C=[0,1]$, it is crucial that the Lagrangian boundary conditions lie along $\{x=0\}$. This is where we use the vanishing condition on $\boldsymbol\tau$ in the construction of $\Sigma L$.
\end{proof}

\begin{lemma}\label{lem:no escape with small action}
 Let
 \[
  \mathcal D\subset\left(\{\R\times[0,1]\}\cup\mathcal R^{*+1}\cup\mathcal R^1_*\cup\{\C\}\cup\{\R\times S^1\}\right)
 \]
 be a compact space of domains, where
 \[
  \mathcal R^{*+1}:=\bigcup_{d=2}^{\infty}\ol{\mathcal R}^{d+1}\quad\mbox{and}\quad\mathcal R^1_*:=\bigcup_{d=1}^{\infty}\ol{\mathcal R}^1_d.
 \]
 Then there is some ${\mathbf m}_0\in\N$ and $E_0>0$ such that any holomorphic curve $u$ with domain in $\mathcal D$, boundary conditions on Lagrangians of the form $\Sigma L$, Floer datum $K^{\bullet,{\mathbf m}}$ with ${\mathbf m}>{\mathbf m}_0$, geometric energy and output action below $E_0$, and all inputs in the central fiber $K\times\{0\}$ lies entirely in $K\times D_{0.9}$.
 
 In particular, because it's the only place in $K\times D_{0.9}$ which supports any Floer generators, the output of $u$ must also lie in the central fiber.
\end{lemma}
\begin{rmk}
 Any holomorphic cap with Floer datum $K^{\C,{\mathbf m}}$ is assumed to have incidence condition $u(0)\in Y_\sigma$. For such curves the condition that all inputs lie in the central fiber is vacuously true.
\end{rmk}
\begin{proof}
 Suppose not, so that for every $E_0$ there are infinitely many pairs $({\mathbf m},u)$ which satisfy the assumptions but do not lie entirely in $K\times D_{0.9}$. Then we may form a sequence $({\mathbf m}_j,E_j,u_j)$ with ${\mathbf m}_j\to\infty$ and $E_j\to 0$, $u_j$ satisfying the assumptions with respect to ${\mathbf m}_j$ and $E_j$, and such that no $u_j$ lies entirely in $K\times D_{0.9}$.
 
 Because of the conditions on
 \[
  \int\max_{\hat M}(d^\Sigma H^{\bullet,{\mathbf m}}\wedge\beta^{\bullet,{\mathbf m}}),
 \]
 a bound on geometric energy gives a bound on topological energy. Because $\mathcal D$ is compact, this combines with our bound on the output action to bound the actions of the inputs. Thus, the maps $u_j$ satisfy a uniform maximum principle.
 
 Now even though the Lagrangian boundary conditions become singular in the limit, the primitives of $\lambda_M$ on all Lagrangians tend uniformly to zero, which excludes all bubbling. Thus, again using that $\mathcal D$ is compact, elliptic compactness applies, so that any subsequence $z_j$ in the domain of $u_j$ has a subsequence $z_{j'}$ such that $u_{j'}$ converges locally in $C^\infty$ on an increasing family of neighborhoods $U_{j'}$ of $z_{j'}$. Because both the geometric energy and all the actions go to zero, $u_{j'}|_{U_{j'}}$ converges to a constant map at a critical point of $H^M$. On the other hand, the critical locus of $H^M$ coincides with its zero locus, so in fact $u_{j'}|_{U_{j'}}$ converges to a point where $H^M=0$.
 
 This implies that $\lVert du_j\rVert$ converges uniformly over the entire domain to $0$, since otherwise we could find a sequence $z_j$ near which $u_j$ cannot converge to a constant map.
 
 Now, for any limiting domain $\bigcup U_{j'}$ with nonempty boundary the limit curve must lie on a point $\Sigma L\cap(H^M)^{-1}(0)$ for some $L$, and all such points belong to the interior of $K\times D_{0.9}$. Indeed, without rescaling it is immediate from the properties of $H^N$, and because all rescaling factors are at least one the rescaled Hamiltonian vanishes on a strictly smaller portion of $\Sigma L$. Similarly, if the limiting domain has an incidence condition on $Y_\sigma$, then the limit curve must lie on a point $Y_\sigma\cap(H^M)^{-1}(0)$, and again all such points belong to the interior of $K\times D_{0.9}$.
 
 The only possible limiting domain which does not fall into one of the above classes is the infinite cylinder $\R\times S^1$. Thus, for large $j'$, any $u_j'$ which exits $K\times D_{0.9}$ must do so along a long cylinder $S$ with positive end near $Y_\sigma\cap(H^M)^{-1}(0)$. Because $\lVert du_j\rVert$ tends uniformly to zero, we can increase $j'$ to ensure that $u_{j'}(S)$ is supported arbitrarily close to $(H^M)^{-1}(0)$ and that the image of $u_{j'}(s,\cdot)$ has small diameter for all $s$. Because $H^M$ is bounded below on $K\times D_{0.9}$ when $|y|\ge0.8$, the above implies that $u_{j'}(S)$ exits $K\times D_{0.9}$ along $|x|\ge0.4$. This contradicts \eqref{eq:Morse approx for x} in Lemma \ref{lem:weak Morse approx}.
\end{proof}

\begin{figure}
 \def\svgwidth{10cm}
 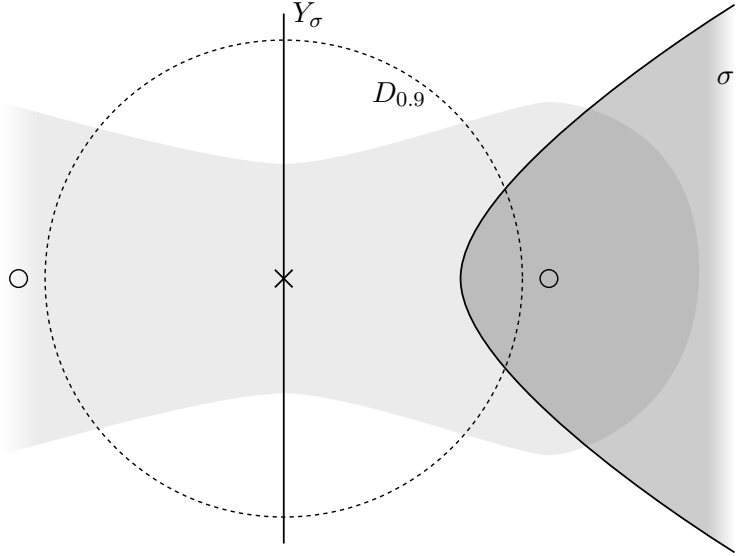
 \caption{The area near the stop $\sigma$ in $M=M'[\sigma]$. The large central shaded region is where $H^M=0$.}\label{fig:trivial gluing}
\end{figure}

\begin{rmk}\label{rmk:confinement for hypersurface to the right}
 Let $Y'$ be some hypersurface which lies to the right of $Y_\sigma$ in the sense of Figure \ref{fig:trivial gluing}. Then holomorphic caps with incidence condition $u(0)\in Y'$ must also have output in the central fiber. This is because the above argument still prevents the output from escaping to the left, while the only generators of symplectic cohomology which are supported entirely to the right of $Y_\sigma$ live in the divisor $D_\sigma$. These are prohibited by the same positivity argument as in the proof of Lemma \ref{lem:SH intersections filter}.
\end{rmk}

\begin{proof}[Proof of Proposition \ref{prop:nondegeneracy}]
 By construction, the restriction of all Floer data to the central fiber are regular Floer data on the central fiber. In particular, $\mathbf K^{\mu,{\mathbf m}}$ gives a choice of $\W^{\frac1{\mathbf m}}(F)$ with respect to the normalizing Hamiltonian $H^M|_{\hat F\times\{0\}}$. Because the Hamiltonians in $\mathbf K^{\mu,{\mathbf m}}$ rotate the imaginary axis counterclockwise, any generator
 \[
  \gamma\in\hom^*_{\W^{\frac1{\mathbf m}}(F)}(L_0,L_1)
 \]
 which lives in $K$ gives rise to a generator
 \[
  \Sigma\gamma:=\gamma\times\{0\}\in\hom^*_{\Ws^{\mathbf m}(M)}(\Sigma L_0,\Sigma L_1)
 \]
 of the same degree living in the central fiber. Here, $\Ws^{\mathbf m}(M)$ is the presentation of $\Ws(M)$ coming from $\mathbf K^{\mu,{\mathbf m}}$. Similarly, any generator $x$ of $SC^*(F)$ for Floer datum $(H^{\mathbf m}_t,J^{\mathbf m}_t)_{\hat F\times\{0\}}$ which lives in $K$ gives rise to a generator
 \[
  \Sigma x:=x\times\{0\}\in SC^{*+1}_\ssigma(M)
 \]
 for Floer datum $(H^{\mathbf m}_t,J^{\mathbf m}_t)$ living in the central fiber. Here, the degree shift comes from the fact that in the disk-direction, the central fiber lives at a saddle point of the Hamiltonian vector field.
 
 Now, because $F$ is strongly nondegenerate, we may choose for all ${\mathbf m}$ a Hochschild fundamental cycle $e_{\mathbf m}$ of length at most the Hochschild length $k$ of $F$ such that $\lim_{{\mathbf m}\to\infty}A^{\frac1{\mathbf m}}(e_{\mathbf m})=0$. For large ${\mathbf m}$, all entries in $e_{\mathbf m}$ live in $K$, and we write $\Sigma e_{\mathbf m}\in CC_*(\mathcal B_\ssigma^{\mathbf m}(\sigma))$ for the corresponding chain in the central fiber, where $\mathcal B_\ssigma^{\mathbf m}(\sigma)$ is as usual the presentation of $\Bss$ coming from $\mathbf K^{\mu,{\mathbf m}}$.  For sufficiently large ${\mathbf m}$, Lemma \ref{lem:no escape with small action} then says that all holomorphic curves contributing to $\delta(\Sigma e_{\mathbf m})$ and $\mathcal{OC}(\Sigma e_{\mathbf m})$ are supported entirely in $K\times D_{0.9}$, and that their outputs live in the central fiber. Because all of the associated Floer data are $\Phi$-invariant on this region and we are working in characteristic 2, the only holomorphic curves which contribute are those which live in the central fiber. This implies that
 \[
  \delta(\Sigma e_{\mathbf m})=\Sigma(\delta(e_{\mathbf m}))=0\quad\mbox{and}\quad\mathcal{OC}(\Sigma e_{\mathbf m})=\Sigma(\mathcal{OC}(e_{\mathbf m})).
 \]
 
 It remains to see that $\Sigma(\mathcal{OC}(e_{\mathbf m}))$ is a saddle unit. For this, note that $\mathcal{OC}(e_{\mathbf m})$ represents $\mathbf 1\in SH^0(F)$, and that it has small action when ${\mathbf m}$ is large. Thus, it is the sum of the Morse minima, as that is the unique such cochain. Playing the same game as for $\mu$ and $\mathcal{OC}$ with holomorphic caps, we see that the saddle unit is also supported in the central fiber, and there it comes from counting holomorphic caps in $F$ with no incidence condition. The output of such caps must have small action, which implies that the saddle unit is also the sum of the Morse minima in $F$, as desired.
\end{proof}
We have in fact proved
\begin{cor}\label{cor:nondegeneracy for all large n}
 For all sufficiently large ${\mathbf m}$, the Floer data $\mathbf K^{\mu,{\mathbf m}}$, $\mathbf K^{\mathcal{OC},{\mathbf m}}$, and $K^{\C,{\mathbf m}}$ witness a Hochschild cycle $y_{\mathbf m}\in CC_{1-n}(\Bss)$ such that $\mathcal{OC}(y_{\mathbf m})$ is a saddle unit.
 
 Moreover, the action of $y_{\mathbf m}$ tends to zero as $\mathbf m$ tends to infinity, the saddle unit $\mathcal{OC}(y_{\mathbf m})$ lives in the central fiber, and $y_{\mathbf m}$ is the stabilization of a Hochschild fundamental chain for the fiber of $\sigma$. This last part means that the objects making up $y_{\mathbf m}$ are all of the form $\Sigma L$ and that the morphisms making up $y_{\mathbf m}$ all live in the central fiber.
\end{cor}

\subsection{Transversality}\label{sec:transversality for nondegeneracy}
We now verify that the holomorphic curves of the previous section which occur in the compact part of the central fiber are automatically regular. For this, we will deduce transversality on $M$ from transversality on $F$. This means that we will assume the Floer data restricts to a regular choice on the central fiber and not allow ourselves further modifications of it.

Note now that the almost complex structures are strictly split, and the Hamiltonians are split to second order. Hence, the linearized Cauchy-Riemann operator splits. In the disk component, the map is constant, which implies that the disk component of the linearized operator has no parameter associated to changing the conformal structure. This means that the disk component of the linearized operator at a map $u\colon\Sigma\to\hat M$ landing in $K\times\{0\}$ takes the form
\begin{equation}\label{eq:linearized dbar along fiber}
 D^{disk}\ol\partial(u)(\xi)=\frac12\Bigl(\bigl(d\xi-dX_H(u(z))(\xi)\otimes\beta\bigr)+J(z)\circ\bigl(d\xi-dX_H(u(z))(\xi)\otimes\beta\bigr)\circ j\Bigr).
\end{equation}
Now the disk component of $X_H$ along the central fiber depends only on the local rescaling factor on $\Sigma$, and in particular it does not depend on $u$. Moreover, the Liouville vector field agrees to second order along the central fiber with the standard radially expanding one, which means that its flow is holomorphic to second order. This means that \eqref{eq:linearized dbar along fiber} simplifies to
\begin{equation}\label{eq:linearized dbar along fiber, simplified}
 D^{disk}\ol\partial(u)(\xi)=\frac12\Bigl(\bigr(d\xi-X_{H^{disk}}(z)(\xi)\otimes\beta\bigr)+J_{std}\circ\bigl(d\xi-X_{H^{disk}}(z)(\xi)\otimes\beta\bigr)\circ j\Bigr),
\end{equation}
where $H^{disk}$ is a Liouville rescaling of $a\cdot(y^2-x^2)$. This is manifestly independent of $u$, so it allows us to treat transversality by simply studying perturbed holomorphic maps to $\C$ for a saddle Hamiltonian. In particular, we immediately obtain transversality for Floer strips and cylinders, and as noted in Remark \ref{rmk:Z/2 transversality} Seidel's topological argument provides transversality for $A_\infty$ disks.

\begin{lemma}\label{lem:equivariant OC transversality in fiber}
 For all sufficiently large ${\mathbf m}$, the following holds. If $u\colon\Sigma\to K\times\{0\}$ is a holomorphic curve for Floer data $\mathbf K^{\mathcal{OC},{\mathbf m}}|_{K\times\{0\}}$ with at most $k$ inputs, then $u$ is regular as a map to $\hat M$ if and only if it is regular as a map to $\hat F$. Here, as before, $k$ is the Hochschild length of $F$.
\end{lemma}
\begin{proof}
 It suffices to show that with the conformal structure of $\Sigma$ fixed,
 \[
  D^{disk}\ol\partial\colon W^{1,p}(\Sigma,\C)\to L^p(\Sigma,\Omega^1\Sigma\otimes\C)
 \]
 is surjective. For this, note that every input has degree 0, while the output has degree 1, so that
 \[
  \mathrm{ind}(D^{disk}\ol\partial)=\deg(x_-)-\sum_{i=1}^d\deg(\gamma_i)-\dim_\C(\C)=0.
 \]
 Hence, proving that $D^{disk}\ol\partial$ is surjective is equivalent to proving that it is injective.
 
 To that end, suppose not, so that for all ${\mathbf m}$ there is some element $\xi_{\mathbf m}$ in the kernel of $D\ol\partial$ for $H^{disk}$ and $\beta$ coming from a Floer datum $K^{\bullet,{\mathbf m}}$. Without loss of generality, assume that $\lVert\xi_{\mathbf m}\rVert_{C^0}=1$ for all ${\mathbf m}$. Now because $H^{\mathcal{OC},{\mathbf m}}$ converges in $C^\infty$ to $H^M$ and $X_{H^M}$ vanishes to at least third order in the disk directions along the central fiber, the linearized Hamiltonians $H^{disk}$ tend uniformly to zero on the image of $\xi_{\mathbf m}$. This, together with the trivial identity $E^{top}(\xi_{\mathbf m})=0$, implies that
 \[
  \lim_{{\mathbf m}\to\infty}E^{geom}(\xi_{\mathbf m})=0.
 \]
 
 To conclude the proof, note that Lemma \ref{lem:no escape with small action} applies to the maps $\xi_{\mathbf m}$. Indeed, the proof carries through to this situation verbatim, except now we need a new way to prevent the maps from escaping in the $y$-direction. For this, we observe that what goes up must come down: if some portion of $\xi_{\mathbf m}$ reaches $|y|\ge0.8$ for all ${\mathbf m}$, then because the output is at the origin some strip or cylinder must make the return journey to $|y|<0.1$. This contradicts \eqref{eq:Morse approx for y} in Lemma \ref{lem:weak Morse approx}. On the other hand, the conclusion $\mathrm{image}(\xi_{\mathbf m})\in D_{0.9}$ contradicts the assumption that $\lVert\xi_{\mathbf m}\rVert_{C^0}=1$.
\end{proof}

\begin{lemma}\label{lem:equivariant saddle unit transversality in fiber}
 If $u\colon\C\to K\times\{0\}$ is a holomorphic curve for Floer data $K^{\C,{\mathbf m}}|_{K\times\{0\}}$, then $u$ is regular as a map to $\hat M$ if and only if it is regular as a map to $\hat F$. In this case, the evaluation map at zero is transverse to $Y_\sigma$.
\end{lemma}
\begin{proof}
 This time we wish to show both that $D^{disk}\ol\partial$ is surjective and that $ev_0(\ker D^{disk}\ol\partial)$ is transverse to the $y$-axis.
 
 Because the output has degree 1, showing that $D^{disk}\ol\partial$ is surjective amounts to showing that $\ker D^{disk}\ol\partial$ is 1-dimensional. To that end, note first that for any nonconstant solution $\xi$, the asymptotic winding number of $\xi$ around the origin is nonpositive, which implies that the total number of zeros of $\xi$ is nonpositive. By positivity of intersections, this implies that $\xi$ has no zeros. This means that $\xi=0$ whenever $ev_0(\xi)=0$. By linearity, this implies that
 \[
  ev_0\colon\ker D^{disk}\ol\partial\to\C
 \]
 is injective.
 
 Now there is a 1-dimensional space $X$ of rotationally invariant solutions coming from partial increasing Morse trajectories from some point $p\in\R$ to the origin. If $X$ were not all of $\ker D^{disk}\ol\partial$, then because $ev_0$ is injective $\ker D^{disk}\ol\partial$ would be 2-dimensional. Since $\ker D^{disk}\ol\partial$ comes with a continuous $S^1$ action which fixes $X$ pointwise, the entire $S^1$ action must be trivial, which means all solutions are rotationally invariant. On the other hand, any rotationally invariant solution must be a partial Morse trajectory, and the asymptotic condition $\xi(\infty)=0$ implies that all such solutions belong to $X$. This shows that $\ker D^{disk}\ol\partial=X$ is 1-dimensional.
 
 For the second part, note simply that $ev_0(X)=\R$ is indeed transverse to $i\R$.
\end{proof}

\section{Stop removal}\label{ch:stop removal}

\subsection{A filtration on the quotient category}\label{sec:filtration on quotient}
To prove Theorem \ref{thm:quotient}, we will work in Lyubashenko--Ovsienko's model for the quotient of an $A_\infty$-category by a full subcategory \cite{Lyub-Ovsi}, which is the $A_\infty$ version of Drinfeld's construction for dg-categories \cite{Drin_quot}. For an $A_\infty$-category $\mathcal A$ and a full subcategory $\mathcal B\subset\mathcal A$, the quotient $\mathcal A/\mathcal B$ is the $A_\infty$-category with the same objects as $\mathcal A$ and whose morphism spaces are given by
\[
 \hom_{\mathcal A/\mathcal B}(L_0,L_1)=\bigoplus_{k=0}^\infty\bigoplus_{B_i\in\mathcal B}\hom_\mathcal A(B_k,L_1)\otimes\hom_\mathcal A(B_{k-1},B_k)\otimes\dotsm\otimes\hom_\mathcal A(L_0,B_1),
\]
where for $k=0$ the right-hand side is just $\hom_\mathcal A(L_0,L_1)$. The grading is given by
\[
 \deg(\gamma^k\otimes\dotsm\otimes\gamma^0)=\sum_{i=0}^k\deg(\gamma^i)-k.
\]
The differential $\mu^1_{\mathcal A/\mathcal B}$ is the bar differential, i.e.
\[
 \mu^1_{\mathcal A/\mathcal B}(\gamma^k\otimes\dots\otimes\gamma^0)=\sum_{0\le i\le j\le k}
 \gamma^k\otimes\dotsm\otimes\gamma^{j+1}\otimes\mu^{1+j-i}(\gamma^j,\dotsc,\gamma^i)\otimes\gamma^{i-1}\otimes\dotsm\otimes\gamma^0.
\]
The higher operations are similar. Specifically, we have
\[
\begin{gathered}
 \mu^d_{\mathcal A/\mathcal B}\left((\gamma_d^{k_d}\otimes\dotsm\otimes\gamma_d^0),\dotsc,(\gamma_1^{k_1}\otimes\dotsm\otimes\gamma_1^0)\right)\hspace{5cm}\\
 \hspace{2cm}=\sum_{\substack{0\le i\le k_1\\0\le j\le k_d}}\gamma_d^{k_d}\otimes\dotsm\otimes \mu^{i+j+d+\sum_{s=2}^{d-1}k_s}(\gamma_d^j,\dotsc,\gamma_1^{k_1-i})\otimes\dotsm\otimes\gamma_1^0.
\end{gathered}
\]
In this model, $\Ws(M)/\Bss$ is naturally a subcategory of $\W_{\ssigma\setminus\{\sigma\}}(M)/\mathcal B_{\ssigma\setminus\{\sigma\}}(\sigma)$, and $\mathcal{SR}$ is just the inclusion.

Further, we can work in the trivial gluing $M[\sigma]$. Writing $\ssigma'$ for $\ssigma\setminus\{\sigma\}$, this means we will study the inclusion
\begin{equation}\label{eq:SR as inclusion}
 \mathcal{SR}_{inc}\colon\Ws(M[\sigma])/\Bss\into\Wsp(M[\sigma])/\Bsps.
\end{equation}
For the sake of readability, we will write $\hom_\ssigma$ in place of $\hom_{\Ws(M[\sigma])/\Bss}$ and $\hom_{\ssigma'}$ in place of $\hom_{\Wsp(M[\sigma])/\Bsps}$. Similarly, we'll write $\mu^k_{\ssigma'}$ and $\mu^k_\ssigma$ for the $A_\infty$ operations on the quotient categories.

Our main tool in studying the quotient category is an increasing filtration on the morphism spaces $\hom^*_{\ssigma'}(L_0,L_1)$.

\begin{defn}\label{defn:filtration}
 Consider the lexicographic order on $\left(\Z_{\ge0}\right)^2$, namely $(\mathfrak n,\mathfrak m)<(\mathfrak n',\mathfrak m')$ if $\mathfrak n<\mathfrak n'$ or both $\mathfrak n=\mathfrak n'$ and $\mathfrak m<\mathfrak m'$. This has order type $\omega^2$, so in particular it is a well-ordering.
 
 Define $A^*_{\mathfrak n,\mathfrak m}\subset\hom^*_{\ssigma'}(L_0,L_1)$ to be the graded vector subspace generated by
 \[
 \left\{\gamma^k\otimes\dotsm\otimes\gamma^0\;\Big|\,\left(\sum_{i=0}^kn_\sigma(\gamma^i),k\right)<(\mathfrak n,\mathfrak m)\right\}.
 \]
 Then $A^*_{\mathfrak n,\mathfrak m}$ is a subcomplex, and $A^*_{\mathfrak n,\mathfrak m}\subset A^*_{\mathfrak n',\mathfrak m'}$ whenever $(\mathfrak n,\mathfrak m)<(\mathfrak n',\mathfrak m')$. This means that the $A^*_{\mathfrak n,\mathfrak m}$ form an exhausting filtration of $\hom^*_{\ssigma'}(L_0,L_1)$, which we call the \textbf{main filtration}. Note that $A^*_{1,0}=\hom^*_\ssigma(L_0,L_1)$.
\end{defn}

Theorem \ref{thm:quotient} is equivalent to the assertion that $\mathcal{SR}_{inc}$ is fully faithful whenever $\sigma$ is strongly nondegenerate, which follows from the following statement:

\begin{prop}\label{prop:retracting homotopy}
 Assume $\sigma$ is strongly nondegenerate, and let $L_0$ and $L_1$ be interior Lagrangians in $M[\sigma]$. Then there is a homotopy
 \[
  \Delta_y\colon\hom^*_{\ssigma'}(L_0,L_1)\to\hom^{*-1}_{\ssigma'}(L_0,L_1)
 \]
 such that
 \begin{equation}\label{eq:basic retraction def}
 R_y:=\mathrm{id}+\mu_{\ssigma'}^1\Delta_y+\Delta_y\mu_{\ssigma'}^1
 \end{equation}
 is the identity on $A^*_{1,0}=\hom^*_\ssigma(L_0,L_1)$ and strictly decreases the filtration on $A^*_{\mathfrak n,\mathfrak m}$ for $(\mathfrak n,\mathfrak m)>(1,0)$.
 
 The homotopy $\Delta_y$ and $R_y$ depend on a Hochschild cycle $y\in CC_{1-n}(\Bss)$ which is the stabilization of a low-action Hochschild fundamental cycle, i.e. of the type furnished by Corollary \ref{cor:nondegeneracy for all large n}.
\end{prop}

Indeed, we want to show that $\mathcal{SR}_{inc}$ is both injective and surjective on cohomology. For injectivity, let $x\in A^*_{1,0}$ have $\mu_\ssigma^1x=0$, and suppose that $x$ becomes exact in $\hom^*_{\ssigma'}(L_0,L_1)$. By well-ordering, there is some smallest $(\mathfrak n,\mathfrak m)$ such that $x$ becomes exact in $A^*_{\mathfrak n,\mathfrak m}$. Take $z\in A^*_{\mathfrak n,\mathfrak m}$ satisfying $\mu_{\ssigma'}^1(z)=x$. Then
\begin{align*}
 \mu_{\ssigma'}^1(R_yz) &= \mu_{\ssigma'}^1(z)+\mu_{\ssigma'}^1\Delta_y\mu_{\ssigma'}^1(z)\\
			&= x + \mu_{\ssigma'}^1\Delta_y(x) \\
			&= x + \Delta_y\mu_{\ssigma'}^1(x) \qquad\mbox{ because }x\in A^*_{1,0}\\
			&= x
\end{align*}
Now, if $(\mathfrak n,\mathfrak m)>(1,0)$, then $R_yz$ is a primitive for $x$ which lives lower on the main filtration, contradicting the assumption that $A^*_{\mathfrak n,\mathfrak m}$ is the smallest place where $x$ becomes exact. Hence $z\in A^*_{1,0}$, which proves injectivity. For surjectivity, the proof is similar.

The rest of the paper is devoted to proving Proposition \ref{prop:retracting homotopy}.

\subsection{Coproduct disks}
The map $\Delta_y$ giving rise to $R_y$ will be given by a certain coproduct operation, which will come as always from counts of holomorphic disks. We describe these now.

\begin{defn}\label{defn:coproduct disks}
 For nonnegative integers $d,k,l$, let $\mathcal R^{d;k,l}$ be the abstract \textbf{moduli space of coproduct disks}. These are disks $\Sigma$ with $d+k+l+3$ boundary punctures labeled in counterclockwise order as follows.
\begin{enumerate}
 \item $\zeta_i$, for $i$ increasing from $-k$ to $+l$. These will eventually be equipped with positive strip-like ends.
 \item $\zeta_a$, which will eventually be equipped with a negative strip-like end.
 \item $\zeta^j$ for $j$ increasing from $1$ to $d$. These will eventually be equipped with positive strip-like ends.
 \item $\zeta_b$, which will eventually be equipped with a negative strip-like end.
\end{enumerate}
 The punctures $\zeta_0$, $\zeta_a$, and $\zeta_b$ are considered \textbf{distinguished points}. Observe that any disk with $n+3$ punctures can made into an element of some $\mathcal R^{d;k,l}$ with $d+k+l=n$ by specifying and labeling three distinguished points. The compactified moduli space $\ol{\mathcal R}^{d;k,l}$ is diffeomorphic to the associahedron $\ol{\mathcal R}^{(d+k+l+2)+1}$, where the identification can be taken to match $\zeta_0$ with $\zeta_0$. The codimension $r$ boundary faces of $\ol{\mathcal R}^{d;k,l}$ are identified with products of some lower-dimensional $\mathcal R^{d';k',l'}$ with $r$ lower-dimensional associahedra inductively as follows.
 
 For a codimension 1 face, a point corresponds to a pair of disks identified at new boundary punctures $\tild\zeta$, and we may look at the induced labels of boundary punctures on these two disks. One of these disks, which we call $\Sigma_0$, contains two or three distinguished points, while the other disk $\Sigma_1$, contains one or zero. In each case $\Sigma_0$ will be taken to lie in $\mathcal R^{d';k',l'}$, but there are several ways that this can happen.
 \begin{enumerate}
  \item The first possibility is that $\Sigma_0$ contains all three distinguished points. In this case it is identified with an element of $\mathcal R^{d';k',l'}$ by matching up the distinguished points. $\Sigma_1$ is identified with a point of $\mathcal R^{m+1}$ by taking $\tild\zeta$ to be the root.
  \item The second possibility, similar to the first, is that $\Sigma_0$ contains $\zeta_a$ and $\zeta_b$, while $\Sigma_1$ contains $\zeta_0$. In this case $\Sigma_0$ is identified with the element of $\mathcal R^{d';k',l'}$ which has $\zeta_a$ and $\zeta_b$ in the same place and $\zeta_0$ in the position of $\tild\zeta$. $\Sigma_1$ is again identified with a point of $\mathcal R^{m+1}$ by taking $\tild\zeta$ to be the root. By remembering the distinguished point $\zeta_0\in\Sigma_1$, we may upgrade it to an element of $\mathcal R^{m+1,i}$ for some $i$.
  \item The third and fourth possibilities are that $\Sigma_1$ contains $\zeta_a$ or $\zeta_b$. We assume that $\zeta_a\in\Sigma_1$, as the other situation is strictly similar. In this case, for $\Sigma_0$, $\tild\zeta$ takes the place of $\zeta_a$ as the third distinguished point, while $\Sigma_1$ is identified with an element of $\mathcal R^{m+1}$ by setting $\zeta_a$ to be the root.
 \end{enumerate}
 For a higher codimension face, we obtain a decomposition by following a sequence of faces, each of which has codimension 1 in the previous. To see that the decomposition is unique, note that an element $\Sigma$ of the boundary of $\ol{\mathcal R}^{d;k,l}$ is a disk with boundary nodes described by a tree $T$. If at least two distinguished points of $\Sigma$ live on the same component, then that component is the one identified with an element of $\mathcal R^{d';k',l'}$. Otherwise, there is a unique vertex $v$ of $T$ such that every path from $v$ to a vertex containing a distinguished point leaves $v$ along a different edge. The component of $\Sigma$ identified with a point of $\mathcal R^{d';k',l'}$ is the one corresponding to $v$.
\end{defn}

As indicated, a collection of strip-like ends for a disk $\Sigma\in\mathcal R^{d;k,l}$ consists of a positive strip-like end at each puncture $\zeta_i$ and $\zeta^j$, along with a negative strip-like end at each of $\zeta_a$ and $\zeta_b$, such that the images of the ends are pairwise disjoint. A \textbf{universal choice} of strip-like ends for $\mathcal R^{d;k,l}$ consists of, for all $d,k,l\ge0$, a collection of strip-like ends for each $\Sigma\in\mathcal R^{d;k,l}$ which varies smoothly over $\mathcal R^{d;k,l}$ and agrees near the boundary with the collection of ends induced by gluing. As with associahedra, a universal choice of strip-like ends for $\mathcal R^{d;k,l}$ can be constructed inductively, and we fix one once and for all.

Similarly, a \textbf{universal and conformally consistent} choice of Floer data for $\mathcal R^{d;k,l}$ consists of, for all $d,k,l\ge0$, a Floer datum for each $\Sigma\in\mathcal R^{d;k,l}$ which varies smoothly over $\mathcal R^{d;k,l}$, and which additionally satisfies the asymptotic consistency condition of Definition \ref{defn:consistent Floer data} with $\mathcal R^{d+1}$ replaced by $\mathcal R^{d;k,l}$. Let $\mathcal K_\Delta(M[\sigma])$ denote the space of all universal and conformally consistent choices Floer data for $\mathcal R^{d;k,l}$.

For any $\mathbf K_\Delta\in\mathcal K_\Delta(M[\sigma])$, we obtain a perturbed Cauchy-Riemann operator. Given Lagrangians $L_0,\dotsc,L_d$ and $B_{-k-1},\dotsc,B_l$ and chords
\begin{equation}\label{eq:coprod ends location}
 \begin{aligned}
  \gamma_i&\in\mathscr X(B_{i-1},B_i)\quad&\gamma_a&\in\mathscr X(L_0,B_l)\\
  \gamma^j&\in\mathscr X(L_{j-1},L_j)\quad&\gamma_b&\in\mathscr X(B_{-k-1},L_d),
 \end{aligned}
\end{equation}
we can consider the space $\mathcal R^{d;k,l}(\vec\gamma^\star,\vec\gamma_\star;\gamma_b,\gamma_a)$, where $\vec\gamma^\star$ and $\vec\gamma_\star$ are the tuples $(\gamma^d,\dotsc,\gamma^1)$ and $(\gamma_l,\dotsc,\gamma_{-k})$, respectively. This consists of all maps $u\colon\Sigma\to\widehat{M[\sigma]}$, with $\Sigma$ ranging over $\mathcal R^{d;k,l}$, satisfying \eqref{eq:gen Floer} with
\[
 \begin{aligned}
  u(\zeta_i)&=(\phi^{\tau_i})^*\gamma_i\quad&u(\zeta_a)&=(\phi^{\tau_a})^*\gamma_a\\
  u(\zeta^j)&=(\phi^{\tau^j})^*\gamma^j\quad&u(\zeta_b)&=(\phi^{\tau_b})^*\gamma_b
 \end{aligned}
\]
and with the appropriate boundary conditions, where $\tau_i$ is the rescaling factor assigned to $\zeta_i$, and similarly with $\tau^j$, $\tau_a$, and $\tau_b$. As usual, Lemma \ref{lem:maximum principle} tells us that the images of such $u$ are all contained in a fixed compact subset of $\hat M$, so Gromov compactness gives $\mathcal R^{d;k,l}(\vec\gamma^\star,\vec\gamma_\star;\gamma_b,\gamma_a)$ a natural compactification $\ol{\mathcal R}^{d;k,l}(\vec\gamma^\star,\vec\gamma_\star;\gamma_b,\gamma_a)$ whose new points are broken configurations consisting of one element of $\mathcal R^{d';k',l'}((\vec\gamma^\star)',(\vec\gamma_\star)';\gamma_b',\gamma_a')$ for some $d'\le d$, $k'\le k$, and $l'\le l$, along with disks contributing to the $A_\infty$ structure. We list those configurations with exactly two nonconstant components.
\begin{subequations}\label{eq:coproduct breaks}
\begin{equation}\label{eq:coprod left Hochschild breaks}
\begin{gathered}
  \mathcal R^{d;k+1-m,l}(\vec\gamma^\star,(\gamma_l,\dotsc,\gamma_{i+m},\tild\gamma,\gamma_{i-1},\dotsc,\gamma_{-k});\gamma_b,\gamma_a)\\
  \times\mathcal R^{m+1}(\gamma_{i+m-1},\dotsc,\gamma_i;\tild\gamma)
\end{gathered}
\hspace{.6cm}
\begin{gathered}
  \mbox{\footnotesize$2\le m\le-i\le k$}\\
  \mbox{\footnotesize$\tild\gamma\in\mathscr X(B_{i-1},B_{i+m-1})$}
\end{gathered}
\end{equation}
\begin{equation}\label{eq:coprod right Hochschild breaks}
\begin{gathered}
  \mathcal R^{d;k,l+1-m}(\vec\gamma^\star,(\gamma_l,\dotsc,\gamma_{i+m},\tild\gamma,\gamma_{i-1},\dotsc,\gamma_{-k});\gamma_b,\gamma_a)\\
  \times\mathcal R^{m+1}(\gamma_{i+m-1},\dotsc,\gamma_i;\tild\gamma)
\end{gathered}
\hspace{.6cm}
\begin{gathered}
  \mbox{\footnotesize$2\le m\le 1+l-i\le l$}\\
  \mbox{\footnotesize$\tild\gamma\in\mathscr X(B_{i-1},B_{i+m-1})$}
\end{gathered}
\end{equation}
\begin{equation}\label{eq:coprod middle Hochschild breaks}
\begin{gathered}
  \mathcal R^{d;k+i,l-i+1-m}(\vec\gamma^\star,(\gamma_l,\dotsc,\gamma_{i+m},\tild\gamma,\gamma_{i-1},\dotsc,\gamma_{-k});\gamma_b,\gamma_a)\\
  \times\mathcal R^{m+1}(\gamma_{i+m-1},\dotsc,\gamma_i;\tild\gamma)
\end{gathered}
\hspace{.3cm}
\begin{gathered}
  \mbox{\footnotesize$-k\le i\le 0$}\\
  \mbox{\footnotesize$\min\{2,1-i\}\le m\le l-i+1$}\\
  \mbox{\footnotesize$\tild\gamma\in\mathscr X(B_{i-1},B_{i+m-1})$}
\end{gathered}
\end{equation}
\begin{equation}\label{eq:coprod Floer Hochschild breaks}
\begin{gathered}
  \mathcal R^{d;k,l}(\vec\gamma^\star,(\gamma_l,\dotsc,\gamma_{i+1},\tild\gamma,\gamma_{i-1},\dotsc,\gamma_{-k});\gamma_b,\gamma_a)\\
  \times\mathcal R(\gamma_i;\tild\gamma)
\end{gathered}
\hspace{1cm}
\begin{gathered}
  \mbox{\footnotesize$-k\le i\le l$}\\
  \mbox{\footnotesize$\tild\gamma\in\mathscr X(B_{i-1},B_i)$}
\end{gathered}
\end{equation}
\begin{equation}\label{eq:coprod right domain break}
\begin{gathered}
  \mathcal R^{i+j+1+1}(\gamma_l,\dotsc,\gamma_{l-j+1},\tild\gamma,\gamma^i,\dotsc,\gamma^1;\gamma_a)\\
  \times\mathcal R^{d-i;k,l-j}((\gamma^d,\dotsc,\gamma^{i+1}),(\gamma_{l-j},\dotsc,\gamma_{-k});\gamma_b,\tild\gamma)
\end{gathered}
\hspace{.6cm}
\begin{gathered}
  \mbox{\footnotesize$0\le i\le d$}\\
  \mbox{\footnotesize$\max\{0,1-i\}\le j\le l$}\\
  \mbox{\footnotesize$\tild\gamma\in\mathscr X(L_i,B_{l-j})$}
\end{gathered}
\end{equation}
\begin{equation}\label{eq:coprod right Floer break}
\hspace{2cm}
\mathcal R(\tild\gamma;\gamma_a)
  \times\mathcal R^{d;k,l}(\vec\gamma^\star,\vec\gamma_\star;\gamma_b,\tild\gamma)
\hspace{2.2cm}
  \mbox{\footnotesize$\tild\gamma\in\mathscr X(L_0,B_l)$}
\end{equation}
\begin{equation}\label{eq:coprod left domain break}
\begin{gathered}
  \mathcal R^{i+j+1+1}(\gamma^d,\dotsc,\gamma^{d+1-i},\tild\gamma,\gamma_{j-k-1},\dotsc,\gamma_{-k};\gamma_b)\\
  \times\mathcal R^{d-i;k-j,l}((\gamma^{d-i},\dotsc,\gamma^1),(\gamma_l,\dotsc,\gamma_{j-k});\tild\gamma,\gamma_a)
\end{gathered}
\hspace{1cm}
\begin{gathered}
  \mbox{\footnotesize$0\le i\le d$}\\
  \mbox{\footnotesize$\max\{0,1-i\}\le j\le k$}\\
  \mbox{\footnotesize$\tild\gamma\in\mathscr X(B_{j-k-1},L_{d-i})$}
\end{gathered}
\end{equation}
\begin{equation}\label{eq:coprod left Floer break}
\hspace{2.5cm}
\mathcal R(\tild\gamma;\gamma_b)
  \times\mathcal R^{d;k,l}(\vec\gamma^\star,\vec\gamma_\star;\tild\gamma,\gamma_a)
\hspace{2.2cm}
  \mbox{\footnotesize$\tild\gamma\in\mathscr X(B_{-k-1},L_d)$}
\end{equation}
\begin{equation}\label{eq:coprod upper domain break}
\begin{gathered}
  \mathcal R^{d+1-m;k,l}((\gamma^d,\dotsc,\gamma^{i+m},\tild\gamma,\gamma^{i-1},\dotsc,\gamma^1),\vec\gamma_\star;\gamma_b,\gamma_a)\\
  \times\mathcal R^{m+1}(\gamma^{i+m-1},\dotsc,\gamma^i;\tild\gamma)
\end{gathered}
\hspace{.6cm}
\begin{gathered}
  \mbox{\footnotesize$2\le m\le d+1-i\le d$}\\
  \mbox{\footnotesize$\tild\gamma\in\mathscr X(L_{i-1},L_{i+m-1})$}
\end{gathered}
\end{equation}
\begin{equation}\label{eq:coprod upper Floer break}
\begin{gathered}
  \mathcal R^{d;k,l}((\gamma^d,\dotsc,\gamma^{i+1},\tild\gamma,\gamma^{i-1},\dotsc,\gamma^1),\vec\gamma_\star;\gamma_b,\gamma_a)\\
  \times\mathcal R(\gamma^i;\tild\gamma)
\end{gathered}
\hspace{2cm}
\begin{gathered}
  \mbox{\footnotesize$1\le i\le d$}\\
  \mbox{\footnotesize$\tild\gamma\in\mathscr X(L_{i-1},L_i)$}
\end{gathered}
\end{equation}
\end{subequations}
This looks like a lot, but the first four are just different ways that an $A_\infty$ disk can break off on the ``subscript'' input side, of which the first three differ only in the placement of the marked input. The others, in pairs, describe the possible breakings of an $A_\infty$ disk at $\zeta_a$, at $\zeta_b$, and on the ``superscript'' input side.

\begin{lemma}\label{lem:coprod transversality}
 For generic $\mathbf K_\Delta\in\mathcal K_\Delta(M[\sigma])$, all moduli spaces $\mathcal R^{d;k,l}(\vec\gamma^\star,\vec\gamma_\star;\gamma_b,\gamma_a)$ are transversely cut out of dimension
 \[
  \deg(\gamma_a)+\deg(\gamma_b)-\sum_{r=d}^{d}\deg(\gamma^r)-\sum_{s=-k}^l\deg(\gamma_s)+d+k+l-n.
 \]
 In this case, the codimension $1$ boundary strata of the compactifications $\ol{\mathcal R}^{d;k,l}(\vec\gamma^\star,\vec\gamma_\star;\gamma_b,\gamma_a)$ are in natural bijection with the configurations \eqref{eq:coproduct breaks}.
 
 Moreover, $\mathcal R^{d;k,l}(\vec\gamma^\star,\vec\gamma_\star;\gamma_b,\gamma_a)$ is empty whenever
  \[
   n_\sigma(\gamma_a)+n_\sigma(\gamma_b)>\sum_{r=1}^{d}n_\sigma(\gamma^r)+\sum_{s=-k}^ln_\sigma(\gamma_s).
  \]
\qed
\end{lemma}

\subsection{The main homotopy}

We are now prepared to begin constructing the operation $\Delta_y$ which is used in the definition of the basic retraction $R_y$ in \eqref{eq:basic retraction def}. For now, we define a coproduct operation $\Delta_y^0$. This is the \textbf{main homotopy}. Later on, we will define a second operation $h_y$ and set $\Delta_y=\Delta_y^0+h_y$.

For all ${\mathbf m}\in\N$, choose as in Section \ref{sec:collapsing energy} a regular Floer datum $\mathbf K_\Delta^{\mathbf m}$ for the coproduct which satisfies the analogs of Lemmas \ref{lem:weak Morse approx} and \ref{lem:no escape with small action}, but for which the Hamiltonian is not necessarily $\Phi$-invariant on $K\times D_{0.9}$. Denote the resulting moduli spaces by $\mathcal R^{d;k,l}_{\mathbf m}(\vec\gamma^\star,\vec\gamma_\star;\gamma_b,\gamma_a)$.

\begin{figure}
  \def\svgwidth{15.5cm}
  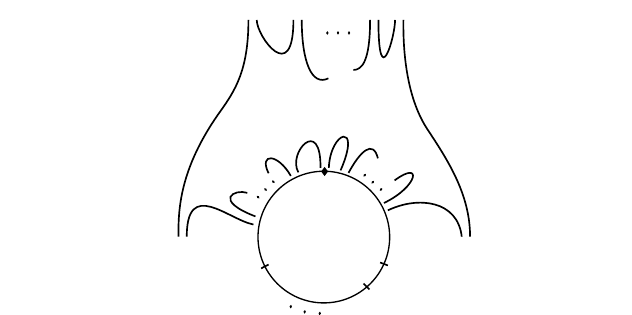
  \caption{}\label{fig:main htopy}
\end{figure}

We now define $\Delta_y$ on generators and extend by linearity. Concretely, for generators
\[
 \gamma=\gamma^m\otimes\dotsm\otimes\gamma^0\in\hom^*_{\ssigma'}(L_0,L_1)
\]
and
\[
 y=\gamma_q\otimes\dotsm\otimes\gamma_1\in CC_*(\Bss),
\]
where $y$ a low-action stabilization (cf. Corollary \ref{cor:nondegeneracy for all large n}), define
\begin{equation}\label{eq:main homotopy}
 \Delta_{y,{\mathbf m}}^0(\gamma)=\sum_{\substack{0\le i\le m+1\\ 0\le d\le m+1-i\\ k+l<q\\ n_\sigma(\gamma^r)=0\,\forall r<i}}
 \hspace{-7cm}\sum_{\substack{\gamma_a,\gamma_b\\ n_\sigma(\gamma_a)=0\\ \hspace{72mm}\deg(\gamma_a)+\deg(\gamma_b)=\sum_{r=i}^{i+d-1}\deg(\gamma^r)+\sum_{s=q-k}^l\deg(\gamma_s)+n-d-k-l}}
 \hspace{-7cm}\#\mathcal R^{d;k,l}_{\mathbf m}((\gamma^{i+d-1},\dotsc,\gamma^i),(\gamma_l,\dotsc,\gamma_{q-k});\gamma_b,\gamma_a)\cdot\widehat\gamma,
\end{equation}
where tuples with increasing or nonexistent indices are the empty tuple $()$, and
\begin{equation}
 \widehat\gamma:=\gamma^m\otimes\dotsm\otimes\gamma^{i+d}\otimes\gamma_b\otimes\gamma_{q-k-1}\otimes\dotsm\otimes\gamma_{l+1}\otimes\gamma_a\otimes\gamma^{i-1}\otimes\dotsm\otimes\gamma^0
\end{equation}
is required to be composable. See Figure \ref{fig:main htopy}. Again the indices for the Hochschild chain $v$ are cyclically ordered. A straightforward calculation gives that $\Delta_{y,{\mathbf m}}^0$ is homogeneous of degree $\deg(y)+n-2$. In particular, $\Delta_{y,{\mathbf m}}^0$ has degree $-1$ when $\deg(y)=1-n$, which is precisely the grading which appears in Definition \ref{defn:nondegenerate stop}.

\begin{rmk}\label{rmk:large m gives good coprod}
 Note that $\Delta_{y,{\mathbf m}}^0$ is in general poorly defined, as \eqref{eq:main homotopy} in principle allows contributions with $d=0$ and $i=0$ or $m+1$. Such terms would be elements of
 \[
  \hom(B_m,L_1)\otimes\dotsm\otimes\hom(L_0,B_1)\otimes\hom(B^v_{q-k},L_0)\otimes\dotsm\otimes\hom(L_0,B^v_{l+1})
 \]
 or
 \[
  \hom(B^v_{q-k},L_1)\otimes\dotsm\otimes\hom(L_1,B^v_{l+1})\otimes\hom(B_m,L_1)\otimes\dotsm\otimes\hom(L_0,B_1),
 \]
 which are not components of $\hom_{\ssigma'}(L_0,L_1)$. However, for large ${\mathbf m}$ and $y$ a low-action stabilization, these terms can be excluded either for energy reasons or by appealing to the analog of Lemma \ref{lem:no escape with small action}. In particular, $\Delta_{y,{\mathbf m}}^0$ is well defined when ${\mathbf m}$ is large and $y$ is the stabilization of a low-action Hochschild fundamental cycle. We will suppress the $\mathbf m$ and write $\Delta_y^0$ to denote $\Delta_{y,{\mathbf m}}^0$ for large but indeterminate ${\mathbf m}$. Later on we will need to increase ${\mathbf m}$ whenever we introduce a new moduli space. This only happens finitely many times, so it is not a problem, and we will do it implicitly without change to the notation.
\end{rmk}

\begin{rmk}
The key condition here is that, in the output, all chords starting with $\gamma_a$ must have crossing number zero with $\sigma$. This is what allows the resulting chain map to interact with the intersection filtration. In particular, we will see that $\mu_{\ssigma'}^1\Delta^0_y+\Delta^0_y\mu_{\ssigma'}^1$ is nontrivial precisely at the smallest $r$ such that $n_\sigma(\gamma^r)\ne0$, where it is homotopic to the identity up to terms lower in the main filtration.
\end{rmk}

We begin by examining the configurations of holomorphic disks which appear in $\Delta_y^0\mu_{\ssigma'}^1(\gamma)$. These come in two types.
\begin{equation}\label{descr:delta mu noninteracting}
 \parbox{14cm}{The first type occurs when the superscript inputs for $\Delta_y^0$ do not include the output for $\mu_{\ssigma'}^1$. In this case, there are two components which are disjoint and do not want to glue together.}
\end{equation}
\begin{equation}\label{descr:delta mu interacting}
 \parbox{14cm}{The second type occurs when the superscript inputs for $\Delta_y^0$ do include the output for $\mu_{\ssigma'}^1$. In this case, the configuration is a broken disk of the form \eqref{eq:coprod upper domain break} or \eqref{eq:coprod upper Floer break}.}
\end{equation}

\begin{figure}
  \def\svgwidth{9cm}
  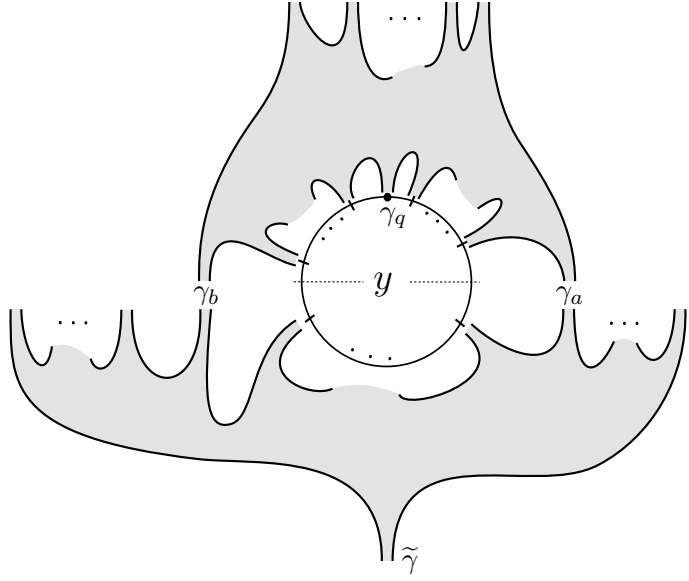
  \caption{A formal annulus contributing to $A_y$.}\label{fig:formal annulus}
\end{figure}

Next, we examine the configurations of holomorphic disks which appear in $\mu_{\ssigma'}^1\Delta_y^0(\gamma)$. These come in five types.
\begin{equation}\label{descr:mu delta noninteracting}
 \parbox{14cm}{The first type occurs when the inputs for $\mu_{\ssigma'}^1$ do not include any of $\gamma_a$, $\gamma_b$, or the $\gamma_i$ coming from the unused components of $y$. In this case, there are two components which are disjoint and do not want to glue together. These configurations are exactly the same as those in \eqref{descr:delta mu noninteracting}, so their contributions to $\Delta_y^0\mu_{\ssigma'}^1+\mu_{\ssigma'}^1\Delta_y^0$ cancel.}
\end{equation}
\begin{equation}\label{descr:mu delta Hochschild differential}
 \parbox{14cm}{The second type occurs when the inputs for $\mu_{\ssigma'}^1$ consist of one or more of the $\gamma_i$ coming from the unused components of $y$. In this case, the configuration consists of two disjoint disks, each of which uses different portions of the Hochschild chain $y$.}
\end{equation}
\begin{equation}\label{descr:mu delta right composition}
 \parbox{14cm}{The third type occurs when the inputs for $\mu_{\ssigma'}^1$ include $\gamma_a$ but not $\gamma_b$. In this case, the configuration is a broken disk of the form \eqref{eq:coprod right domain break} or \eqref{eq:coprod right Floer break}.}
\end{equation}
\begin{equation}\label{descr:mu delta left composition}
 \parbox{14cm}{The fourth type occurs when the inputs for $\mu_{\ssigma'}^1$ include $\gamma_b$ but not $\gamma_a$. In this case, the configuration is a broken disk of the form \eqref{eq:coprod left domain break} or \eqref{eq:coprod left Floer break}.}
\end{equation}
\begin{equation}\label{descr:mu delta annulus}
 \parbox{14cm}{The fifth type occurs when the inputs for $\mu_{\ssigma'}^1$ include both $\gamma_a$ and $\gamma_b$. In this case, the configuration is formally an annulus with two nodes, see Figure \ref{fig:formal annulus}. The outside of this annulus is labeled with some substring of $\gamma$ and an output chord $\tild\gamma$, while the inside is labeled with the entire Hochschild chain $y$. Let
 \[
  A_y\colon\hom^*_{\ssigma'}(L_0,L_1)\to\hom^{*+\deg(y)+n-1}_{\ssigma'}(L_0,L_1)
 \]
 be the linear map obtained by counting only such annuli.}
\end{equation}

We claim that, modulo terms which decrease the main filtration,
\begin{equation}\label{eq:main homotopy simplified}
 \Delta_y^0\mu_{\ssigma'}^1+\mu_{\ssigma'}^1\Delta_y^0=\Delta^0_{\delta y}+A_y.
\end{equation}
For this, it suffices to consider those contributions to $\Delta_y^0\mu_{\ssigma'}^1+\mu_{\ssigma'}^1\Delta_y^0$ which come from configurations which avoid $D_\sigma$. Considering only such configurations, we want to see that the terms coming from  \eqref{descr:delta mu noninteracting}-\eqref{descr:mu delta left composition} add up to the corresponding portion of $\Delta^0_{\delta y}$. To do so, note that there are two types of configurations which contribute to $\Delta^0_{\delta y}$. The first type occurs when the output of the $A_\infty$ disk contributing to $\delta y$ is not an input of $\mathcal R^{d;k,l}(\vec\gamma^\star,\vec\gamma_\star;\gamma_b,\gamma_a)$, so instead it appears in $\widehat\gamma$. In this case, the configuration is precisely what is counted in \eqref{descr:mu delta Hochschild differential}. The remaining type of configuration occurs when the output of the $A_\infty$ disk contributing to $\delta y$ is a component of $\vec\gamma_\star$. In this case, the broken configuration is one of \eqref{eq:coprod left Hochschild breaks}-\eqref{eq:coprod Floer Hochschild breaks}. Because the spaces in \eqref{eq:coproduct breaks} form the boundary of a compact 1-manifold, we are left with terms coming from configurations of the form \eqref{eq:coprod upper domain break}-\eqref{eq:coprod left Floer break}. On the other hand, for spaces of coproduct disks which do not intersect $D_\sigma$, the condition that $n_\sigma(\gamma_a)=0$ is preserved under breaking off an $A_\infty$ disk. Thus, the operation coming from these configurations coincides precisely with the sum of the remaining terms \eqref{descr:delta mu interacting}, \eqref{descr:mu delta right composition}, and \eqref{descr:mu delta left composition}.

Since we will eventually be interested in the case where $y$ is closed, we can ignore the $\delta y$ term in \eqref{eq:main homotopy simplified}. Moreover, since our goal is to show that $R_y$ satisfies the conditions of Proposition \ref{prop:retracting homotopy}, we may ignore all terms of $\Delta_y^0\mu_{\ssigma'}^1+\mu_{\ssigma'}^1\Delta_y^0$ which strictly decrease the main filtration. Because all operations involved satisfy positivity of intersections and $y$ is made up of chords with $n_\sigma=0$, the only way in which they can fail to decrease the filtration is by failing to decrease the length of a generator $\gamma^m\otimes\dotsm\otimes\gamma^0$. For the annulus term $A_y$, this only happens when the broken annuli are labeled with zero or one $\gamma^j$ input. Write $A_y=\sum_{\nu,\mu}A_y^{\nu,\mu}$, where $A_y^{\nu,\mu}$ is the operation coming from those broken holomorphic annuli with $\mu$ superscript inputs and intersection number $\nu$ with $D_\sigma$. This means, for $A_y^{\nu,\mu}(\gamma)\ne0$, we have
\[
 (\mathfrak n,\mathfrak m)(A_y^{\nu,\mu}(\gamma))=(\mathfrak n,\mathfrak m)(\gamma)+(-\nu,1-\mu).
\]
We conclude

\begin{lemma}\label{lem:main htopy effect}
 Let $y\in CC_*(\Bss)$ be a closed, low-action stabilization. Then, up to terms which decrease the main filtration, $\Delta_y^0\mu_{\ssigma'}^1+\mu_{\ssigma'}^1\Delta_y^0=A_y^{0,0}+A_y^{0,1}$.
 \qed
\end{lemma}

\subsection{Closed-open maps}\label{sec:CO maps}

The eventual objective will be to show that $A_y^{0,0}+A_y^{0,1}$ is homotopic to a closed-open operation $\mathcal{CO}^{filt}_x$ depending on a cochain $x\in SC_\ssigma^*(M[\sigma])$, where it will turn out that $x=\mathcal{OC}(y)$. We now construct this operation.

Let $\mathcal R^{0+1}_1$ be the singleton set containing a disk $D^{0,1}_1$ with one interior puncture $\zeta_+$ and one boundary puncture $\zeta_0$. Up to biholomorphism, there is a unique such disk. Equip $\zeta_0$ with a negative strip-like end $\epsilon_0$ and $\zeta_+$ with a positive cylindrical end $\epsilon_+$. As with the punctured disks giving rise to $\mathcal{OC}$, we ask that $\epsilon_+$ has a very special form. Specifically, in the holomorphic coordinates on $D^{0,1}_1$ where $\mathrm{int}(\Sigma)=\{z\in\C\mid0<|z|<1\}$ and $\zeta_0=-1$, we require that
\begin{equation}\label{eq:CO pos end alignment}
\epsilon_+(s,t)=ae^{-2\pi(s+it)}\qquad\text{with $a\in\R$ positive.}
\end{equation}

Going up in dimension, let $\mathcal R^{1+1}_1$ be the space of disks with one interior puncture $\zeta_+$ and two boundary punctures $\zeta_0$ and $\zeta_1$. The corresponding compactified moduli space $\ol{\mathcal R}^{1+1}_1$ is
\begin{equation}\label{OC domain compactification}
 \ol{\mathcal R}^{1+1}_1=\left(\mathcal R^{2+1}\times\mathcal R^{0+1}_1\right)\amalg\mathcal R^{1+1}_1\amalg\left(\mathcal R^{2+1}\times\mathcal R^{0+1}_1\right),
\end{equation}
where the two broken configurations correspond to the two ways of attaching the negative end of $D^{0,1}_1$ to one of the two positive ends of the unique disk $\Sigma^{2+1}\in\mathcal R^{2+1}$. Choose smooth, disjoint $\mathcal R^{1+1}_1$-parametrized families of positive strip-like ends $\epsilon_1$ for $\zeta_1$, negative strip-like ends $\epsilon_0$ for $\zeta_0$, and positive cylindrical ends $\epsilon_+$ for $\zeta_+$, which satisfy the compatibility conditions
\begin{enumerate}
 \item In the gluing charts of the form $[0,a)\times\mathcal R^{2+1}\times\mathcal R^{0+1}_1$ with gluing length $\ell=e^{\frac1\rho}$, where $\rho\in[0,a)$, the families of ends agree to infinite order at $\rho=0$ with those induced by gluing.
 \item For all $\Sigma\in\mathcal R^{1+1}_1$, in the holomorphic coordinates on $\Sigma$ where the interior of $\Sigma$ is the punctured disk $\{z\in\C\mid0<|z|<1\}$ and $\zeta_0=-1$, $\epsilon_+$ satisfies \eqref{eq:CO pos end alignment}.
\end{enumerate}
Unlike with open-closed maps, for careful choices of ends elsewhere the agreement to infinite order could be strengthened to agreement in a neighborhood of the boundary, but there is no benefit to doing so. One could also extend the above and construct a map to Hochschild cohomology as in \cite{Gana_scdwc}, but in our application the higher terms would reduce the main filtration, so we ignore them.

A \textbf{conformally consistent} choice of Floer data for the closed-open maps consists of a Floer datum on $D^{0,1}_1$, along with a Floer datum on $\Sigma$ for each $\Sigma\in\mathcal R^{1+1}_1$ varying smoothly over $\mathcal R^{1+1}_1$, and such that near $\partial\ol{\mathcal R}^{1+1}_1$ it agrees to infinite order with the conformal class of not-quite Floer data determined by gluing. Denote by $\mathcal K^{\mathcal{CO}}(M[\sigma])$ the space of conformally consistent choices of Floer data for the closed-open maps.

Given $\mathbf K\subset\mathcal K^{\mathcal{CO}}(M[\sigma])$, we can consider the resulting holomorphic curves. Given Lagrangian labels $L_i$ and asymptotic ends $\gamma_i$ as in \eqref{eq:gamma_i location} and $x_+\in\mathscr X(H_t)$, we are interested in the spaces
\begin{align*}
 &\mathcal R^{0+1}_1(x_+;\gamma_0)\\
 &\mathcal R^{1+1}_1(x_+,\gamma_1;\gamma_0).
\end{align*}
These consists of all maps $u\colon\Sigma\to\widehat{M[\sigma]}$ for $\Sigma\in\mathcal R^{0+1}_1$ or $\Sigma\in\mathcal R^{1+1}_1$, respectively, satisfying \eqref{eq:gen Floer w/ cylinders} with $u(E_i)\subset(\phi^{\tau_{E}})^*L_i$, $u(\zeta_i)=(\phi^{\tau_i})^*\gamma_i$, and $u(\zeta_+)=(\phi^{\tau_+})^*x_+$.

\begin{lemma}\label{lem:CO transversality}
 For generic $\mathbf K\in\mathcal K^{\mathcal{CO}}(M[\sigma])$, all moduli spaces $\mathcal R^{0+1}_1(x_+;\gamma_0)$ and $\mathcal R^{1+1}_1(x_+,\gamma_1;\gamma_0)$ are transversely cut out and have the following features.
 \begin{enumerate}
  \item $\mathcal R^{0+1}_1(x_+;\gamma_0)$ has dimension $\deg(\gamma_0)-\deg(x_+)$. It is empty unless $n_\sigma(\gamma_0)\le n_\sigma(x_+)$. It has a Gromov compactification whose codimension one boundary strata are in natural bijection with
  \begin{equation}\label{CO^0 boundary strata}
   \coprod_{\tild x\in\mathscr X(H_t)}\left(\mathcal R^{0+1}_1(\tild x;\gamma_0)\times\mathcal Q(x_+;\tild x)\right)\;\amalg\coprod_{\tild\gamma\in\mathscr X(L,L)}\left(\mathcal R(\tild\gamma;\gamma_0)\times\mathcal R^{0+1}_1(x_+;\tild\gamma)\right).
  \end{equation}
  \item $\mathcal R^{1+1}_1(x_+,\gamma_1;\gamma_0)$ has dimension
  \[
   \deg(\gamma_0)-\deg(\gamma_1)-\deg(x_+)+1.
  \]
  It is empty unless $n_\sigma(\gamma_0)\le n_\sigma(\gamma_1)+n_\sigma(x_+)$. It has a Gromov compactification whose codimension one boundary strata are in natural bijection with
  \begin{equation}\label{eq:CO^1 boundary strata}
  \begin{gathered}
   \coprod_{\tild x\in\mathscr X(H_t)}\left(\mathcal R^{1+1}_1(\tild x,\gamma_1;\gamma_0)\times\mathcal Q(x_+;\tild x)\right)\;\amalg\coprod_{\tild\gamma\in\mathscr X(L_0,L_1)}\hspace{-3mm}\left(\mathcal R^{1+1}_1(x_+,\tild\gamma;\gamma_0)\times\mathcal R(\gamma_1;\tild\gamma)\right)\\
   \amalg\coprod_{\tild\gamma\in\mathscr X(L_0,L_1)}\hspace{-3mm}\left(\mathcal R(\tild\gamma;\gamma_0)\times\mathcal R^{1+1}_1(x_+,\gamma_1;\tild\gamma)\right)\\   
   \amalg\coprod_{\tild\gamma\in\mathscr X(L_1,L_1)}\hspace{-3mm}\left(\mathcal R^{2+1}(\tild\gamma,\gamma_1;\gamma_0)\times\mathcal R^{0+1}_1(x_+;\tild\gamma)\right)\;\amalg\coprod_{\tild\gamma\in\mathscr X(L_0,L_0)}\hspace{-3mm}\left(\mathcal R^{2+1}(\gamma_1,\tild\gamma;\gamma_0)\times\mathcal R^{0+1}_1(x_+;\tild\gamma)\right).
  \end{gathered}
  \end{equation}
 \end{enumerate}
\qed
\end{lemma}

\begin{defn}\label{defn:filtered CO component}
Suppose $u\in\mathcal R^{1+1}_1(x_+,\gamma_1;\gamma_0)$ with $x_+\in SC_\ssigma^*(M[\sigma])$ and $n_\sigma(\gamma_0)=n_\sigma(\gamma_1)$. Then, by positivity of intersections, $u$ doesn't pass through $D_\sigma$. Let $\Sigma$ be the domain of $u$, and let $e\colon[0,1]\to\Sigma$ be a path with $e(0)\in E_0$ and $e(1)=\zeta_+$. Since $u$ avoids $D_\sigma$, so does $u\circ e$, and hence the topological intersection number of $u\circ e$ with $\sigma(\hat F\times\R_+)$ is well defined and independent of the choice of $e$. Let $n^{\mathcal{CO}}_\sigma(u)$ be this number. The \textbf{filtered closed-open moduli space} $\mathcal R^{1+1,filt}_1(x_+,\gamma_1;\gamma_0)$ is the connected component of $\mathcal R^{1+1}_1(x_+,\gamma_1;\gamma_0)$ consisting of $u$ with $n^{\mathcal{CO}}_\sigma(u)=0$.

For $n_\sigma(\gamma_0)<n_\sigma(\gamma_1)$, we take $\mathcal R^{1+1,filt}_1(x_+,\gamma_1;\gamma_0)$ to be empty, though one could just as well take it to be all of $\mathcal R^{1+1}_1(x_+,\gamma_1;\gamma_0)$.
\end{defn}

For $x\in SC_\ssigma^*(M[\sigma])$, define $\mathcal{CO}^{filt}_x\colon\hom^*_{\ssigma'}(L_0,L_1)\to\hom^{*+\deg(x)-1}_{\ssigma'}(L_0,L_1)$ to depend linearly on $x$ and, for $x\in\mathscr X(H_t)$ a generator, to satisfy
\begin{equation}\label{eq:CO formula}
\begin{gathered}
 \mathcal{CO}^{filt}_x(\gamma^m\otimes\dotsm\otimes\gamma^0)=\hspace{-15mm}\sum_{\substack{0\le i\le m+1\\ n_\sigma(\gamma^r)=0\,\forall r<i\\ \tild\gamma\text{ making the result composable}\\ \deg(\tild\gamma)=\deg(x)}}\hspace{-15mm}\#\mathcal R^{0+1}_1(x;\tild\gamma)\cdot\gamma^m\otimes\dotsm\otimes\gamma^i\otimes\tild\gamma\otimes\gamma^{i-1}\otimes\dotsm\otimes\gamma^0\\
 +\hspace{-15mm}\sum_{\substack{0\le i\le m\\ n_\sigma(\gamma^r)=0\,\forall r<i\\ \tild\gamma\text{ making the result composable}\\ \deg(\tild\gamma)=\deg(x)+\deg(\gamma^i)-1}}\hspace{-15mm}\#\mathcal R^{1+1,filt}_1(x,\gamma^i;\tild\gamma)\cdot\gamma^m\otimes\dotsm\otimes\gamma^{i+1}\otimes\tild\gamma\otimes\gamma^{i-1}\otimes\dotsm\otimes\gamma^0.
\end{gathered}
\end{equation}

\begin{rmk}\label{rmk:large m gives good CO}
 As with the main homotopy, this is in fact poorly defined in general. However, by choosing a sequence of Floer data $\mathbf K_{\mathbf m}$ with collapsing energies, we can again ensure that the problematic terms don't appear when $\mathbf m$ is large and $x$ is a small-action chain supported on the central fiber. In this case, energy alone doesn't exclude the bad components of $\mathbf R^{0+1}_1(x,\tild\gamma)$, but the analog of Lemma \ref{lem:no escape with small action} implies that $\tild\gamma$ lives in the central fiber. This is impossible for $L_0$ and $L_1$, since they are assumed to be interior Lagrangians of $M$ and not just $M[\sigma]$.
 
 \emph{For subsequent moduli spaces, we will implicitly choose Floer data for each $\mathbf m$ and retroactively increase $\mathbf m$ to exclude the new bad terms.}
\end{rmk}

Note, as with the coproduct, that we have only allowed configurations where all long chords occur after the new one. The filtered moduli space is the object which captures the contributions of $\Delta_y^0$ for which $n_\sigma(\gamma_a)=0$ while $n_\sigma(\gamma_b)>0$.

\subsection{Annuli, part 1}

To relate $A_y^{0,0}+A_y^{0,1}$ with $\mathcal{CO}^{filt}_{\mathcal{OC}(y)}$, we follow Abouzaid's construction in \cite{Abou_gcgfc}. Specifically, we will coherently extend his first and second homotopies to allow for one outer input and verify that the result is a homotopy on $\hom^*_{\ssigma'}(L_0,L_1)$. This section constructs the first homotopy. The result is a homotopy $h^1_y$ such that, for $\delta y=0$, the operation $A_y^{0,0}+A_y^{0,1}+h^1_y\mu_{\ssigma'}^1+\mu_{\ssigma'}^1h^1_y$ counts analytically gluable broken annuli.

For that, let
\[
\mathcal P^0_d=\coprod_{\substack{k,l\ge0\\ m\ge2\\ k+l+m=d+1}}[0,1]\times\ol{\mathcal R}^{m+1}\times\ol{\mathcal R}^{0;k,l}.
\]
Note that this differs slightly from Abouzaid's terminology. First, the superscript of zero means that there are no outer inputs, meaning that we should think of gluing the first input $\zeta_1$ of the $A_\infty$ disk to the first output $\zeta_a$ of the coproduct, and likewise we should glue the last input $\zeta_m$ of the $A_\infty$ disk to the last output $\zeta_b$ of the coproduct. Second, we do not bother identifying paired boundary components. When we consider holomorphic curves with domains in $\mathcal P^0_d$, this means that there will be extra boundary terms which cancel in pairs.

Similarly, let
\[
\begin{gathered}
\hspace{-9cm}\mathcal P^1_d=\hspace{-3mm}\coprod_{\substack{k,l\ge0\\ m\ge2\\ k+l+m=d+1}}[0,1]\times\ol{\mathcal R}^{m+1}\times\ol{\mathcal R}^{1;k,l}\vspace{-7mm}\\
\hspace{35mm}\amalg\coprod_{\substack{k,l\ge0\\ m\ge3\\ k+l+m=d+2}}[0,1]\times\ol{\mathcal R}^{m+1,1}\times\ol{\mathcal R}^{0;k,l}\;\amalg\coprod_{\substack{k,l\ge0\\ m\ge3\\ k+l+m=d+2}}[0,1]\times\ol{\mathcal R}^{m+1,m}\times\ol{\mathcal R}^{0;k,l}.
\end{gathered}
\]
Here, the first term is as before except with an outer input in the coproduct disk. For the other two, we have an extra distinguished input on the $A_\infty$ disk. In this case, we attach $\zeta_a$ to the first nondistinguished input and $\zeta_b$ to the last nondistinguished input.

\begin{defn}\label{defn:first htopy disk components}
Before we can start to choose Floer data, we need some auxiliary definitions. For any disk
\[
 \Sigma\in\ol{\mathcal R}^{m+1}\amalg\ol{\mathcal R}^{m+1,1}\amalg\ol{\mathcal R}^{m+1,m},
\]
let
\[
\begin{aligned}
 \zeta_{first}&=\begin{cases}
                \zeta_1&\text{ for }\Sigma\in\ol{\mathcal R}^{m+1}\amalg\ol{\mathcal R}^{m+1,m}\\
                \zeta_2&\text{ for }\Sigma\in\ol{\mathcal R}^{m+1,1}
               \end{cases}\\
 \zeta_{last}&=\begin{cases}
                \zeta_m&\text{ for }\Sigma\in\ol{\mathcal R}^{m+1}\amalg\ol{\mathcal R}^{m+1,1}\\
                \zeta_{m-1}&\text{ for }\Sigma\in\ol{\mathcal R}^{m+1,m}.
               \end{cases}
\end{aligned}
\]
For a two-component stable disk
\[
 \Sigma\in\partial\left(\ol{\mathcal R}^{m+1}\amalg\ol{\mathcal R}^{m+1,1}\amalg\ol{\mathcal R}^{m+1,m}\right),
\]
the \textbf{main component} of $\Sigma$ is the component which contains at least two of $\{\zeta_0,\zeta_{first},\zeta_{last}\}$. For a stable disk with more than two components, the main component is the one for which $\zeta_0$, $\zeta_{first}$, and $\zeta_{last}$ all lie in different directions.
\end{defn}

\begin{defn}\label{defn:consistent FD first htopy}
 For $p=(t,\Sigma^{m+1},\Sigma^{j;k,l})\in\mathcal P^{\ol j}_d$, a Lagrangian labeling of $p$ consists of a Lagrangian labeling for each of $\Sigma^{m+1}$ and $\Sigma^{j;k,l}$ such that the labels at $\zeta_{first}\in\Sigma^{m+1}$ agree with the labels at $\zeta_a\in\Sigma^{j;k,l}$, and similarly with $\zeta_{last}$ and $\zeta_b$. A \textbf{universal and conformally consistent} choice of Floer data $\mathbf K^{\mathcal P}$ for the first homotopy consists, for all $d\ge1$ and $i\in\{0,1\}$ and each $p=(t,\Sigma^{m+1},\Sigma^{j;k,l})\in\mathcal P^{\ol j}_d$ with Lagrangian labels, of a Floer datum $\mathbf K^{\mathcal P}(p)$ on $\Sigma^{m+1}$ with the corresponding labels, such that $\mathbf K^{\mathcal P}$ varies smoothly on $\mathcal P^{\ol j}_d$ and has the following properties.
 \begin{enumerate}
  \item For $t=0$, $\mathbf K^{\mathcal P}(p)$ agrees up to conformal rescaling with the Floer datum on $\Sigma^{m+1}$ chosen for the $A_\infty$ structure.
  \item For $t=1$, the configuration is gluable to an annulus after a conformal rescaling. Concretely, let $r_\Delta\colon\ol{\mathcal R}^{j;k,l}\to(0,\infty)$ be the unique smooth function with $r_\Delta(\Sigma)=\frac{\tau_b}{\tau_a}$ for $\Sigma\in\mathcal R^{j;k,l}$. Similarly, let
  \[
   r_\mu\colon\coprod\mathcal P^{\ol j}_d\to(0,\infty)
  \]
  be the unique smooth function with
  \[
   r_\mu(p)=\frac{\tau_{last}(p)}{\tau_{first}(p)},
  \]
  where $\tau_{first}$ and $\tau_{last}$ are the rescaling factors that $\mathbf K^{\mathcal P}$ assigns to the ends $\zeta_{first}\in\Sigma^{m+1}$ and $\zeta_{last}\in\Sigma^{m+1}$, respectively. We require that
  \[
   r_\mu(1,\Sigma^{m+1},\Sigma^{j;k,l})=r_\Delta(\Sigma^{j;k,l}).
  \]
  \item If $\Sigma^{m+1}$ is a nontrivial stable disk, then on every component of $\Sigma^{m+1}$ aside from the main component, $\mathbf K^{\mathcal P}(p)$ is conformally equivalent to the Floer datum chosen for that disk as an element of the associahedron.
  \item If $\Sigma^{m+1}$ is a nontrivial stable disk, let $\Sigma_{main}$ be its main component. If $\Sigma_{main}$ doesn't contain $\zeta_{first}$, let $\Sigma_{first}$ be the possibly-nodal connected piece of $\Sigma^{m+1}\setminus\Sigma_{main}$ containing $\zeta_{first}$. Likewise, if $\Sigma_{main}$ doesn't contain $\zeta_{last}$, let $\Sigma_{last}$ be the possibly-nodal connected piece of $\Sigma^{m+1}\setminus\Sigma_{main}$ containing $\zeta_{last}$. Define a probably-nodal disk
  \[
   \Sigma_{big}:=\left(\Sigma_{first}\amalg\Sigma^{j;k,l}\amalg\Sigma_{last}\right)/(\zeta_a=\zeta_{first},\zeta_b=\zeta_{last})\in\ol{\mathcal R}^{j';k',l'}.
  \]
  Then the restriction of $\mathbf K^{\mathcal P}(t,\Sigma^{m+1},\Sigma^{j;k,l})$ to $\Sigma_{main}$ is conformally equivalent to the Floer datum $\mathbf K^{\mathcal P}(t,\Sigma_{main},\Sigma_{big})$.
  \item Suppose $\Sigma^{j;k,l}$ is a nontrivial stable disk, and that $\Sigma_{leaf}\subset\Sigma^{j;k,l}$ is an irreducible $A_\infty$-type component which is only attached to the rest of $\Sigma^{j;k,l}$ at the negative puncture. In other words, $\Sigma_{leaf}$ has one nodal negative puncture, zero other negative punctures, and its positive punctures are all honest positive punctures of $\Sigma^{j;k,l}$ instead of nodes. Define
  \[
   \Sigma_{small}=\Sigma^{j;k,l}\setminus\Sigma_{leaf}.
  \]
  Then $\mathbf K^{\mathcal P}(t,\Sigma^{m+1},\Sigma^{j;k,l})$ is conformally equivalent to $\mathbf K^{\mathcal P}(t,\Sigma^{m+1},\Sigma_{small})$.
 \end{enumerate}
 Denote by $\mathcal K^{\mathcal P}(M[\sigma])$ the space of universal and conformally consistent choices of Floer data for the first homotopy.
\end{defn}

Suppose we have picked some universal choice $\mathbf K^{\mathcal P}\subset\mathcal K^{\mathcal P}(M[\sigma])$. For a generator
\[
 y=\gamma_d\otimes\dotsm\otimes\gamma_1\in CC_*(\Bss),
\]
let $B_i\in\Bss$ be such that $\gamma_i\in\mathscr X(B_{i-1},B_i)$. For $L_0,\dotsc,L_{\ol j}$ interior Lagrangians in $M[\Sigma]$, $\gamma_0\in\mathscr X(L_0,L_{\ol j})$, and
\[
\gamma^\star=\begin{cases}
              \text{nothing}&\text{ for }\ol{j}=0\\
              \gamma^1\in\mathscr X(L_0,L_1)&\text{ for }\ol{j}=1,
             \end{cases}
\]
define
\[
 \mathcal P^{\ol j}_d(y,\gamma^\star;\gamma_0)=\coprod_{\substack{q_a,j,q_b,k,l\ge0\\q_a+j+q_b=\ol{j}\\k+l<d}}\mathcal P^{q_a,j,q_b;k,l}_d(y,\gamma^\star;\gamma_0)
\]
for certain spaces $\mathcal P^{q_a,j,q_b;k,l}_d(y,\gamma^\star;\gamma_0)$. These describe broken annuli with $q_a$ outer inputs between the output and the $a$-node, $j$ outer inputs on the coproduct, and $q_b$ outer inputs between the $b$-node and the output. Concretely, this is the union over all
\[
p=\left(t,\Sigma_\mu,\Sigma_\Delta\right)\in[0,1]\times\mathcal R^{(d-k-l+q_a+q_b+1)+1}\times\mathcal R^{j;k,l}
\]
of the space of all maps
\[
 u\colon\Sigma_\Delta\amalg\Sigma_\mu\to\widehat{M[\sigma]}
\]
satisfying the following conditions.
\begin{enumerate}
 \item Write $u_\Delta:=u|_{\Sigma_\Delta}$ and $u_\mu:=u|_{\Sigma_\mu}$. Then
 \[
  u_\Delta\in\mathcal R^{j;k,l}\left((\gamma^\star),(\gamma_l,\dotsc,\gamma_{d-k});\gamma_b,\gamma_a\right)
 \]
 for some $\gamma_a\in\mathscr X(L_{q_a},B_l)$ and $\gamma_b\in\mathscr X(B_{d-k-1},L_{q_a+j})$.
 \item $u_\mu$ satisfies \eqref{eq:gen Floer} for the Floer datum $\mathbf K^{\mathcal P}(p)$.
 \item Let $\tau_E(p)$ be the boundary rescaling function assigned to $\Sigma_\mu$ by $\mathbf K^{\mathcal P}(p)$. Then
 \[
  u\left(\partial_i\Sigma_\mu\right)\in\begin{cases}
                                                  \left(\phi^{\tau_E(p)}\right)^*L_i&\text{for }i\le q_a\\
                                                  \left(\phi^{\tau_E(p)}\right)^*L_{\ol{j}+(d-k-l+q_a+q_b+1)-i}&\text{for }i\ge d-k-l+q_a+1\\
                                                  \left(\phi^{\tau_E(p)}\right)^*B_{l+i-q_a-1}&\text{otherwise}
                                                 \end{cases}
 \]
 where $\partial_i\Sigma_\mu$ is the portion of the boundary between $\zeta_i$ and $\zeta_{i+1}$, ordered cyclically.
 \item Let $\tau_i$ be the rescaling factor assigned by $\mathbf K^{\mathcal P}(p)$ to $\zeta_i\in\Sigma_\mu$. Then
 \[
  \begin{aligned}
   u(\zeta_0)&=(\phi^{\tau_0})^*\gamma_0&\\
   u(\zeta_{1+q_a})&=(\phi^{\tau_{1+q_a}})^*\gamma_a&\\
   u(\zeta_i)&=(\phi^{\tau_i})^*\gamma_{l+i-q_a-1}&\text{ for $2+q_a\le i\le d-k-l+q_a$}\\
   u(\zeta_{d-k-l+q_a+1})&=(\phi^{\tau_{d-k-l+q_a+1}})^*\gamma_b.&
  \end{aligned}
 \]
 This is exhaustive for $\ol{j}=0$ or $\ol{j}=j=1$. For $\ol{j}=1$ but $j=0$, there is one remaining end. In this case, if $q_a=1$, then we require $u(\zeta_1)=(\phi^{\tau_1})^*\gamma^\star$. Similarly, if $q_b=1$, we require $u(\zeta_{d-k-l+2})=(\phi^{\tau_{d-k-l+2}})^*\gamma^\star$.
\end{enumerate}

For any fixed $y\in CC_*(\Bss)$, because of action, there are only finitely many choices of intermediate chords $(\gamma_a,\gamma_b)$ which $u_\Delta$ can approach. The maximum principle (Lemma \ref{lem:maximum principle}) and the usual Gromov compactness argument then imply that the spaces $\mathcal P^{q_a,j,q_b;k,l}_d(y,\gamma^\star;\gamma_0)$ have Gromov compactifications $\ol{\mathcal P}^{q_a,j,q_b;k,l}_d(y,\gamma^\star;\gamma_0)$ obtained by allowing either or both of $u_\Delta$ and $u_\mu$ to break.

\begin{lemma}\label{lem:first htopy transversality}
For convenience of notation, set $\deg(\gamma^\star)=n_\sigma(\gamma^\star)=0$ whenever $\gamma^\star$ is nothing (i.e. when $\ol{j}=0$). Then for generic $\mathbf K\in\mathcal K^{\mathcal P}(M[\sigma])$, all moduli spaces $\mathcal P^{q_a,j,q_b;k,l}_d(y,\gamma^\star;\gamma_0)$ transversely cut out of dimension $\deg(\gamma_0)-\deg(\gamma^\star)-\deg(y)+\ol{j}+1-n$. They are empty unless $n_\sigma(\gamma_0)\le n_\sigma(\gamma^\star)$. The codimension one boundary strata of their Gromov compactifications consist of all broken configurations $u$ of the following types.
\begin{equation}\label{eq:first htopy coprod breaks}
\parbox{12cm}{In the first type of configuration, $t\in(0,1)$,
\[
u_\Delta\in\partial\ol{\mathcal R^{j;k,l}}((\gamma^?),(\gamma_l,\dotsc,\gamma_{d-k});\gamma_b,\gamma_a),
\]
where
\[
    \gamma^?=\begin{cases}
	      \text{nothing}&\text{ for }j=0\\
	      \gamma^\star=\gamma^1&\text{ for }j=1,
	    \end{cases}
\]
while $u_\mu$ is a map $u_\mu\colon\Sigma_\mu\to\widehat{M[\sigma]}$ satisfying the Cauchy-Riemann equation perturbed by the Floer datum chosen for $(t,\Sigma_\mu,\Sigma_\Delta)$ with the corresponding boundary and asymptotic conditions.}
\end{equation}
\begin{equation}\label{eq:first htopy reducing breaks}
\parbox{12cm}{In the second type of configuration, $t\in(0,1)$,
\[
u_\Delta\in\mathcal R^{j;k,l}(\gamma^?,(\gamma_l,\dotsc,\gamma_{d-k});\gamma_b,\gamma_a),
\]
while $u_\mu$ has broken, such that one component is an honest $A_\infty$ disk $u_\mu^0$ which has neither $\gamma_a$ nor $\gamma_b$ as an input. Such a disk is either a portion of the Hochschild differential on $y$, or it is a 1- or 2-input disk involving $\gamma^\star$ and/or $\gamma_0$. The other disk, $u_\mu^{main}$, is a map $u_\mu^{main}\colon\Sigma_\mu^{main}\to\widehat{M[\sigma]}$ satisfying the Cauchy-Riemann equation perturbed by the Floer datum chosen for $(t,\Sigma_\mu^{main},\Sigma_\Delta)$ with the induced boundary and asymptotic conditions. This configuration is part of some $\mathcal P^{\ol{j}'}_{d'}(y',(\gamma^\star)';\gamma_0)$, where $\ol{j}'=\ol{j}$ unless $\ol{j}=1$ and $u_\mu^0$ involves both $\gamma^1$ and $\gamma_0$.}
\end{equation}
\begin{equation}\label{eq:first htopy transferring breaks}
\parbox{12cm}{In the third type of configuration, $t\in(0,1)$,
\[
u_\Delta\in\mathcal R^{j;k,l}((\gamma^?),(\gamma_l,\dotsc,\gamma_{d-k});\gamma_b,\gamma_a),
\]
while $u_\mu$ has broken, such that one component is an honest $A_\infty$ disk $u_\mu^{in}$ which has $\gamma_a$ or $\gamma_b$ as an input but does not have $\gamma_0$ as its output. The other, $u_\mu^{out}$, is a map $u_\mu^{out}\colon\Sigma_\mu^{out}\to\widehat{M[\sigma]}$ satisfying the Cauchy-Riemann equation perturbed by the Floer datum chosen for $(t,\Sigma_\mu^{out},\Sigma_{big})$ with the induced boundary and asymptotic conditions, where $\Sigma_{big}$ is the broken disk formed by joining the domains of $u_\Delta$ and $u_\mu^{in}$.}
\end{equation}
\begin{equation}\label{eq:first htopy formal annuli}
\parbox{12cm}{In the fourth type of configuration, $t=0$, in which case $u$ is a two-component broken annulus of the type contributing to $A_y^{0,{\ol j}}$. Indeed, $A_y^{0,{\ol j}}$ is a count of precisely such annuli satisfying either (1) $n_\sigma(\gamma^\star)=0$, or (2) $q_a=0$, $n_\sigma(\gamma_a)=0$, and $n_\sigma(\gamma_0)=n_\sigma(\gamma^1)$.}
\end{equation}
\begin{equation}\label{eq:first htopy gluable annuli}
\parbox{12cm}{In the fifth type of configuration, $t=1$, in which case $u$ is a two-component broken annulus which can be glued into an honest perturbed holomorphic annulus.}
\end{equation}
\qed
\end{lemma}

While $\mathcal P^0_d(y;\gamma_0)$ already extends the moduli space giving rise to $A_y^{0,0}$, for $A_y^{0,1}$ we need to look at a connected component of $P^1_d(y,\gamma^1;\gamma_0)$ as hinted by \eqref{eq:first htopy formal annuli}. This is the space $\mathcal P^1_{d,filt}(y,\gamma^1;\gamma_0)$ consisting of those $u\in P^1_d(y,\gamma^1;\gamma_0)$ for which either (1) $n_\sigma(\gamma^1)=0$, or (2) $q_a=0$, $n_\sigma(\gamma_a)=0$, and $n_\sigma(\gamma_0)=n_\sigma(\gamma^1)$. For an equivalent description closer in spirit to the filtered closed-open moduli space, choose for all $\Sigma_\mu$ a path $e\colon[0,1]\to\Sigma_\mu$ starting on the edge $\partial_0\Sigma_\mu$ and ending on $\partial_{1+q_a}\Sigma_\mu$. $\mathcal P^1_{d,filt}(y,\gamma^1;\gamma_0)$ is the space of all $u$ which avoid $D_\sigma$ and for which the topological intersection number $u\circ e$ with $\sigma(\hat F\times\R_+)$ vanishes.

Define a linear map $h^1_y\colon\hom^*_{\ssigma'}(L_0,L_1)\to\hom^{*+\deg(y)+n-2}_{\ssigma'}(L_0,L_1)$ to depend linearly on a low-action stabilization $y$ and, when $y=\gamma_d\otimes\dotsm\otimes\gamma_1\in CC_*(\Bss)$ is a generator, to be given by
\begin{equation}\label{eq:first htopy formula}
\begin{gathered}
 h^1_y(\gamma^m\otimes\dotsm\otimes\gamma^0)\;=\hspace{-15mm}\sum_{\substack{0\le i\le m+1\\ n_\sigma(\gamma^r)=0\,\forall r<i\\ \tild\gamma\text{ making the result composable}\\ \deg(\tild\gamma)=\deg(y)+n-1}}\hspace{-15mm}\#\mathcal P^0_d(y;\tild\gamma)\cdot\gamma^m\otimes\dotsm\otimes\gamma^i\otimes\tild\gamma\otimes\gamma^{i-1}\otimes\dotsm\otimes\gamma^0\\
 +\hspace{-15mm}\sum_{\substack{0\le i\le m\\ n_\sigma(\gamma^r)=0\,\forall r<i\\ \tild\gamma\text{ making the result composable}\\ \deg(\tild\gamma)=\deg(y)+\deg(\gamma^i)+n-2}}\hspace{-15mm}\#\mathcal P^1_{d,filt}(y,\gamma^i;\tild\gamma)\cdot\gamma^m\otimes\dotsm\otimes\gamma^{i+1}\otimes\tild\gamma\otimes\gamma^{i-1}\otimes\dotsm\otimes\gamma^0.
\end{gathered}
\end{equation}
In the same way, we define the \textbf{gluable annulus maps}
\[
 \tild A_y^{0,0}\text{ and }\tild A_y^{0,1}\colon\hom^*_{\ssigma'}(L_0,L_1)\to\hom^{*+\deg(y)+n-1}_{\ssigma'}(L_0,L_1)
\]
via
\begin{subequations}
 \begin{equation}
  \tild A_y^{0,0}(\gamma^m\otimes\dotsm\otimes\gamma^0)\;=\hspace{-15mm}\sum_{\substack{0\le i\le m+1\\ n_\sigma(\gamma^r)=0\,\forall r<i\\ \tild\gamma\text{ making the result composable}\\ \deg(\tild\gamma)=\deg(y)+n}}\hspace{-15mm}\#\left[\mathcal P^0_d(y;\tild\gamma)\right]_{t=1}\cdot\gamma^m\otimes\dotsm\otimes\gamma^i\otimes\tild\gamma\otimes\gamma^{i-1}\otimes\dotsm\otimes\gamma^0
 \end{equation}
 \begin{equation}
  \tild A_y^{0,1}(\gamma^m\otimes\dotsm\otimes\gamma^0)\;=\hspace{-15mm}\sum_{\substack{0\le i\le m\\ n_\sigma(\gamma^r)=0\,\forall r<i\\ \tild\gamma\text{ making the result composable}\\ \deg(\tild\gamma)=\deg(y)+\deg(\gamma^i)+n-1}}\hspace{-15mm}\#\left[\mathcal P^1_{d,filt}(y,\gamma^i;\tild\gamma)\right]_{t=1}\cdot\gamma^m\otimes\dotsm\otimes\gamma^{i+1}\otimes\tild\gamma\otimes\gamma^{i-1}\otimes\dotsm\otimes\gamma^0,
 \end{equation}
\end{subequations}
where the notation $[\cdot]_{t=1}$ refers to the portion of the corresponding moduli space which occurs at $t=1$. This is the portion of the boundary of the 1-dimensional part of the filtered moduli space described in \eqref{eq:first htopy gluable annuli}.

\begin{lemma}\label{lem:first htopy effect}
 Up to terms which decrease the main filtration,
 \begin{equation}\label{eq:first htopy effect}
  h^1_y\mu_{\ssigma'}^1+\mu_{\ssigma'}^1h^1_y=h^1_{\delta y}+A_y^{0,0}+A_y^{0,1}+\tild A_y^{0,0}+\tild A_y^{0,1}.
 \end{equation}
\end{lemma}
\begin{proof}
 Write $h_y^1=(h_y^1)^{0,0}+(h_y^1)^{0,1}$, where $(h_y^1)^{0,0}$ is the part of $h_y^1$ coming from the first sum in \eqref{eq:first htopy formula}, while $(h_y^1)^{0,1}$ is the part coming from the second sum. We begin by analyzing the part of $h^1_y\mu_{\ssigma'}^1+\mu_{\ssigma'}^1h^1_y$ which increases word length. This consists of all ways of applying $(h^1_y)^{0,0}$ and $\mu^1_{\Wsp(M[\sigma])}$ in some order. All such terms cancel except those in which $\mu^1_{\Wsp(M[\sigma])}$ is applied to the output of $h_y^1$, which we see constitute the part of \eqref{eq:first htopy reducing breaks} for $\ol{j}=0$ which do not contribute to the Hochschild differential. We therefore examine the rest of the boundary of the corresponding 1-dimensional moduli space.
 
 The portion of the Hochschild differential not in \eqref{eq:first htopy reducing breaks} appears in \eqref{eq:first htopy coprod breaks}, giving rise to a $(h_{\delta y}^1)^{0,0}$. The rest of \eqref{eq:first htopy coprod breaks} for $\ol{j}=0$ comes from the breaking of an $A_\infty$ disk outputting $\zeta_a$ or $\zeta_b$. Such disks precisely form the contribution of \eqref{eq:first htopy transferring breaks} for a different connected component, and so they cancel in pairs. The remaining terms \eqref{eq:first htopy formal annuli} and \eqref{eq:first htopy gluable annuli} correspond precisely to $A_y^{0,0}$ and $\tild A_y^{0,0}$, respectively, which confirms the portion of \eqref{eq:first htopy effect} which increases word length.
 
 For the portion which preserves word length, in order to avoid leaving the filtered moduli space, we consider only the part of $\mu_{\ssigma'}^1$ which does not decrease intersection number. Among such terms, we are interested in all ways of performing both $\mu^2_{\Wsp(M[\sigma])}$ and $(h^1_y)^{0,0}$ or both $\mu^1_{\Wsp(M[\sigma])}$ and $(h^1_y)^{0,1}$ in some order. These again cancel when the operations take place at different places in $\gamma^m\otimes\dotsm\otimes\gamma^0\in\hom_{\ssigma'}(L_0,L_1)$. The remaining terms include not just the non-Hochschild part of \eqref{eq:first htopy reducing breaks} for $\ol{j}=1$, but also the component of \eqref{eq:first htopy coprod breaks} given by a Floer strip escaping at the outer input of the coproduct, as in \eqref{eq:coprod upper Floer break}. The rest of the argument proceeds as above.
\end{proof}

\subsection{Annuli, part 2}

We can now construct the homotopy between $\tild A_y^{0,0}+\tild A_y^{0,1}$ and $\mathcal{CO}^{filt}_{\mathcal{OC}(y)}$ for a Hochschild cycle $y\in CC_*(\Bss)$. This is the induced effect of the Cardy relation for wrapped Floer theory \cite{Abou_gcgfc,Gana_scdwc} on the quotient category, modified to interact with the intersection filtration.

For $d\ge1$, let $\mathcal A^{0+1}_d$ be the space of conformal annuli $\Sigma$ with the following data
\begin{enumerate}
 \item $d$ punctures on the inner boundary, labeled $\zeta_1$ through $\zeta_d$ in clockwise order. Note that this becomes a standard counterclockwise ordering after exchanging the inner and outer boundary components.
 \item One puncture $\zeta_0$ on the outer boundary component, such that in coordinates
 \begin{equation}\label{eq:annulus modulus}
  \mathrm{int}(\Sigma)=\{z\in\C\mid1<|z|<R\}
 \end{equation}
 for some $R>1$ with $\zeta_d=1$, we have $\zeta_0=-R$.
\end{enumerate}
$\mathcal A^{0+1}_d$ admits a Deligne-Mumford compactification $\ol{\mathcal A}^{0+1}_d$ which is a manifold with stratified boundary. In addition to ordinary smooth corners, it has boundary strata of codimension at least two for which neighborhoods are subvarieties of the standard corner $[0,a)^i\times(-a,a)^j$. This is similar to the situation with multiplihedra. Its codimension 1 boundary components come in three types.
\begin{enumerate}
 \item The first type occurs as some of the inner boundary punctures come together while $R$ remains finite and strictly greater than $1$. Such configurations are described by 
 \begin{equation}\label{eq:broken annuli 0 input Hochschild type}
  \coprod_{\substack{2\le k\le d\\1\le i\le k}}\mathcal A^{0+1}_{d+1-k}\times\mathcal R^{k+1,i}
  \quad\amalg\;\coprod_{\substack{2\le k\le d-1\\1\le i\le d-k}}\mathcal A^{0+1}_{d+1-k}\times\mathcal R^{k+1}
 \end{equation}
 as in \eqref{eq:OC base faces}. As before, the index $i$ in the second term keeps track of where the punctures collided.
 \item The second type occurs as $R$ tends to $1$. Because of the anti-alignment condition on $\zeta_0$ and $\zeta_d$, $\Sigma$ has to break into a nodal configuration in which $\zeta_0$ and $\zeta_d$ are on different irreducible components. Thus, $\Sigma$ must have at least two components, and the codimension 1 condition is that it breaks into exactly two components. Such configurations are described by
 \[
  \coprod_{\substack{k,l\ge0\\ m\ge2\\ k+l+m=d+1}}\mathcal R^{m+1}\times\mathcal R^{0;k,l},
 \]
 which we identify with $\mathrm{int}[\mathcal P^0_d]_{t=1}$, the interior of the portion of $\mathcal P^0_d$ lying over $t=1$.
 \item The third type occurs as $R$ tends to $\infty$. In this case, we obtain two disks attached nodally at their centers, and the anti-alignment condition gives rise to a preferred angular gluing parameter. The configuration is thus parametrized by
 \[
  \mathcal R^{0+1}_1\times\mathcal R^1_d,
 \]
 where the alignment conditions on the cylindrical ends for $\mathcal{CO}$ and $\mathcal{OC}$ implements the restriction on gluing angles. We then obtain boundary charts by introducing a gluing parameter $\rho$ satisfying $\rho=\frac1{\log\ell}$, where $\ell$ is the gluing length for the cylindrical ends.
\end{enumerate}
The higher codimension strata are either combinations of the above or paired boundary strata of $[\mathcal P^0_d]_{t=1}$, or in other words configurations which arise as the boundary of two different components of $[\mathcal P^0_d]_{t=1}$.

Next, let $\mathcal A^{1+1}_d$ for $d\ge1$ be the space of conformal annuli $\Sigma$ with the following data
\begin{enumerate}
 \item $d$ punctures on the inner boundary, labeled $\zeta_1$ through $\zeta_d$ in clockwise order. Note that this becomes a standard counterclockwise ordering after exchanging the inner and outer boundary components.
 \item Two punctures $\zeta_0$ and $\zeta^1$ on the outer boundary component, such that in coordinates \eqref{eq:annulus modulus} with $\zeta_d=1$, we have $\zeta_0=-R$.
\end{enumerate}
$\mathcal A^{1+1}_d$ admits a Deligne-Mumford compactification $\ol{\mathcal A}^{1+1}_d$ which is again a manifold with stratified boundary. Its codimension 1 boundary components come in four types, the first three of which are essentially the same as for $\ol{\mathcal A}^{0+1}_d$.
\begin{enumerate}
 \item The first type occurs as some of the inner boundary punctures come together while $R$ remains finite and strictly greater than $1$. Such configurations are described by 
 \begin{equation}\label{eq:broken annuli 1 input Hochschild type}
  \coprod_{\substack{2\le k\le d\\1\le i\le k}}\mathcal A^{1+1}_{d+1-k}\times\mathcal R^{k+1,i}
  \quad\amalg\;\coprod_{\substack{2\le k\le d-1\\1\le i\le d-k}}\mathcal A^{1+1}_{d+1-k}\times\mathcal R^{k+1}.
 \end{equation}
 \item The second type occurs as $R$ tends to $1$. Such configurations can be identified with $\mathrm{int}[\mathcal P^1_d]_{t=1}$, the interior of the portion of $\mathcal P^1_d$ lying over $t=1$.
 \item The third type occurs as $R$ tends to $\infty$ and is parametrized by
 \[
  \mathcal R^{1+1}_1\times\mathcal R^1_d.
 \]
 \item The fourth type occurs when $\zeta^1$ collides with $\zeta_0$ while $R\in(1,\infty)$. This case is formally similar to the first, but here we reduce the number of outer punctures. The configurations are described by
 \[
  \mathcal R^{2+1}\times\mathcal A^{0+1}_d.
 \]
\end{enumerate}
As before, the higher codimension strata are either combinations of the above or paired boundary strata of $[\mathcal P^1_d]_{t=1}$.

A collection of strip-like ends for an annulus $\Sigma\in\mathcal A^{j+1}_d$ consists of positive strip-like ends $\epsilon_i$ at $\zeta_i$ for $i\in\{1,\dotsc,d\}$ and, if applicable, $\epsilon^1$ at $\zeta^1$, along with a negative strip-like end $\epsilon_0$ at $\zeta_0$, such that the images of the ends are pairwise disjoint. A cylinder for $\Sigma$ is a finite cylinder $\delta\colon[a,b]\times S^1\to\Sigma$ which is disjoint from the strip-like ends and, in the coordinates \eqref{eq:annulus modulus} with $\zeta_d=1$ and $\zeta_0=-R$, takes the form
\[
 \delta(s,t)=ce^{-2\pi(s+it)}\qquad\text{with $c\in\R$ positive.}
\]

A \textbf{universal choice of ends and cylinders} for $\mathcal A^{j+1}_d$ consists, for all $d\ge1$ and $j\in\{0,1\}$, of a collection of strip-like ends for each $\Sigma\in\mathcal A^{j+1}_d$ which varies smoothly over $\mathcal A^{j+1}_d$, along with a cylinder for $\Sigma$ whenever $R\ge2$ which also varies smoothly over $\mathcal A^{j+1}_d$, which satisfy
\begin{enumerate}
 \item The strip-like ends agree to infinite order at the boundary with the collection of strip-like ends induced by gluing.
 \item Near $R=\infty$, the cylinder agrees with the finite cylinder induced by gluing.
 \item When $R=2$, the width $b-a$ of the cylinder is zero.
\end{enumerate}
Fix once and for all a universal choice of strip-like ends and cylinders for $\mathcal A^{j+1}_d$.

Similarly, a \textbf{universal and conformally consistent} choice of Floer data for $\mathcal A^{j+1}_d$ consists, for all $d\ge0$ and $j\in\{0,1\}$, of a Floer datum for each $\Sigma\in\mathcal A^{j+1}_d$ varying smoothly over $\mathcal A^{j+1}_d$, and such that at the boundary it agrees to infinite order with the conformal class of Floer data induced by gluing. It is easy to see that conformal consistency can be achieved, at least away from the $R=1$ boundary of $\ol{\mathcal A}^{j+1}_d$. At the $R=1$ boundary, one needs the observation that, for fixed $d$, the Floer data on paired boundary strata of $\mathcal P^j_d$ agree up to a global conformal factor, so consistency can be extended across the corresponding strata of $\partial\ol{\mathcal A}^{j+1}_d$. Let $\mathcal K^{\mathcal A}(M[\sigma])$ denote the space of all universal and conformally consistent choices of Floer data for $\mathcal A^{j+1}_d$.

Given $\mathbf K^{\mathcal A}\in\mathcal K^{\mathcal A}(M[\sigma])$, we obtain spaces of holomorphic annuli. For $j=0$, these are specified by a generator $y=\gamma_d\otimes\dotsm\otimes\gamma_1\in CC_*(\Bss)$ and a chord $\gamma_0\in\mathscr X(L,L)$ with $L$ an interior Lagrangian in $M[\sigma]$. The resulting moduli space
\[
 \mathcal A^{0+1}_d(y;\gamma_0)
\]
is the space of all maps $u\colon\Sigma\to\widehat{M[\sigma]}$ for $\Sigma\in\mathcal A^{0+1}_d$ satisfying \eqref{eq:gen Floer w/ cylinders} such that $u(\zeta_i)=(\phi^{\tau_i})^*\gamma_i$ and with the corresponding boundary conditions. Similarly, if $L_0$ and $L_1$ are interior Lagrangians of $M[\sigma]$ and $\gamma_0,\gamma^1\in\mathscr X(L_0,L_1)$, then we obtain
\[
 \mathcal A^{1+1}_d(y,\gamma^1;\gamma_0),
\]
the space of perturbed holomorphic curves $u$ with domain in $\mathcal A^{1+1}_d$ such that $u(\zeta_i)=(\phi^{\tau_i})^*\gamma_i$ and $u(\zeta^1)=(\phi^{\tau^1})^*\gamma^1$, and which satisfy the appropriate boundary conditions.

As usual, $\mathcal A^{i+1}_d(y;\gamma^?,\gamma_0)$ have Gromov compactifications $\ol{\mathcal A}^{i+1}_d(y;\gamma^?,\gamma_0)$ obtained by including broken configurations.

\begin{lemma}\label{lem:second htopy transversality}
 For generic $\mathbf K^{\mathcal A}\in\mathcal K^{\mathcal A}(M[\sigma])$, all moduli spaces $\mathcal A^{0+1}_d(y;\gamma_0)$ and $\mathcal A^{1+1}_1(y,\gamma^1;\gamma_0)$ are transversely cut out and have the following features.
 \begin{enumerate}
  \item $\mathcal A^{0+1}_d(y;\gamma_0)$ has dimension $\deg(\gamma_0)-\deg(y)+1-n$. It is empty unless $n_\sigma(\gamma_0)=0$. The codimension one boundary strata of its Gromov compactification consists of all broken configurations of the following types.
  \begin{equation}\label{eq:second htopy 0 input Hochschild breaks}
   \parbox{12cm}{The first type corresponds to the domain hitting a boundary stratum of $\ol{\mathcal A}^{0+1}_d$ of the form \eqref{eq:broken annuli 0 input Hochschild type} or a Floer strip breaking off at a puncture on the inner boundary. In symbols, these are essentially the same as the first four terms of \eqref{eq:OC boundary strata}, though there we separated the chords making up $y$.}
  \end{equation}
  \begin{equation}\label{eq:second htopy 0 input output Floer break}
   \parbox{12cm}{The second type comes from a Floer strip breaking off at $\gamma_0$ and is parametrized by
   \[
    \coprod_{\tild\gamma\in\mathscr X(L,L)}\mathcal R(\tild\gamma;\gamma_0)\times\mathcal A^{0+1}_d(y;\tild\gamma).
   \]
   }
  \end{equation}
  \begin{equation}\label{eq:second htopy 0 input coprod breaks}
   \parbox{12cm}{The third type comes from the domain hitting the $R=1$ boundary and is precisely $[\mathcal P^0_d(y;\gamma_0)]_{t=1}$.}
  \end{equation}
  \begin{equation}\label{eq:second htopy 0 input OC breaks}
   \parbox{12cm}{The fourth type comes from the domain hitting $R=\infty$ and is parametrized by
   \[
    \coprod_{\tild x\in\mathscr X(H_t)}\mathcal R^{0+1}_1(\tild x;\gamma_0)\times\mathcal R^1_d(\gamma_d,\dotsc,\gamma_1;\tild x)
   \]
   }
  \end{equation}
  \item $\mathcal A^{1+1}_1(y,\gamma^1;\gamma_0)$ has dimension
  \[
   \deg(\gamma_0)-\deg(\gamma^1)-\deg(y)+2-n.
  \]
  It is empty unless $n_\sigma(\gamma_0)\le n_\sigma(\gamma^1)$. The codimension one boundary strata of its Gromov compactification consists of all broken configurations of the following types.
  \begin{equation}\label{eq:second htopy 1 input Hochschild breaks}
   \parbox{12cm}{The first type corresponds to the domain hitting a boundary stratum of $\ol{\mathcal A}^{1+1}_d$ of the form \eqref{eq:broken annuli 1 input Hochschild type} or a Floer strip breaking off at a puncture on the inner boundary. In symbols, these are also essentially the same as the first four terms of \eqref{eq:OC boundary strata}.}
  \end{equation}
  \begin{equation}\label{eq:second htopy 1 input outside breaks}
   \begin{gathered}
   \hspace{-1cm}\parbox{12cm}{The second type comes from a Floer strip breaking off at $\gamma^1$ or $\gamma_0$ or from a collision of $\zeta^1$ with $\zeta_0$. Such configurations are parametrized by}\\
    \coprod_{\tild\gamma\in\mathscr X(L_0,L_1)}\mathcal A^{1+1}_d(y,\tild\gamma;\gamma_0)\times\mathcal R(\gamma^1;\tild\gamma)
    \;\amalg\coprod_{\tild\gamma\in\mathscr X(L_0,L_1)}\mathcal R(\tild\gamma;\gamma_0)\times\mathcal A^{1+1}_d(y,\gamma^1;\tild\gamma)\\
    \amalg\coprod_{\tild\gamma\in\mathscr X(L_0,L_0)}\mathcal R^{2+1}(\gamma^1,\tild\gamma;\gamma_0)\times\mathcal A^{0+1}_d(y;\tild\gamma)
    \;\amalg\coprod_{\tild\gamma\in\mathscr X(L_1,L_1)}\mathcal R^{2+1}(\tild\gamma,\gamma^1;\gamma_0)\times\mathcal A^{0+1}_d(y;\tild\gamma).
   \end{gathered}
  \end{equation}
  \begin{equation}\label{eq:second htopy 1 input coprod breaks}
   \parbox{12cm}{The third type comes from the domain hitting the $R=1$ boundary and is precisely $[\mathcal P^1_d(y,\gamma^1;\gamma_0)]_{t=1}$.}
  \end{equation}
  \begin{equation}\label{eq:second htopy 1 input OC breaks}
   \parbox{12cm}{The fourth type comes from the domain hitting $R=\infty$ and is parametrized by
   \[
    \coprod_{\tild x\in\mathscr X(H_t)}\mathcal R^{1+1}_1(\tild x,\gamma^1;\gamma_0)\times\mathcal R^1_d(\gamma_d,\dotsc,\gamma_1;\tild x)
   \]
   }
  \end{equation}
 \end{enumerate}
\qed
\end{lemma}

To extend the filtered versions of the moduli spaces for the closed-open maps and the first homotopy, choose for all $\Sigma\in\mathcal A^{1+1}_d$ a path $e\colon[0,1]\to\Sigma$ such that $e(0)$ is on the outer boundary component to the right of $\zeta_0$ and $e(1)$ is on the inner boundary. Then for any $\gamma_0$ and $\gamma^1$ with $n_\sigma(\gamma_0)=n_\sigma(\gamma_1)$ and any $u\in\mathcal A^{1+1}_d(y;\gamma^1;\gamma_0)$, $u\circ e$ is a path between interior Lagrangians which avoids $D_\sigma$, so it has a well defined intersection number with $\sigma(\hat F\times\R_+)$. Since the chords $\gamma_i$ for $i>0$ have $n_\sigma(\gamma_i)=0$, we can homotope the end of $e$ through $\zeta_i$ without changing the intersection number. This implies that the intersection number is independent of the choice of $e$, and so we call it $n_\sigma^{\mathcal A}(u)$. The space
\[
 \mathcal A^{1+1}_{d,filt}(y,\gamma^+;\gamma_0)
\]
consists of all $u\in\mathcal A^{1+1}_d(y;\gamma^1;\gamma_0)$ which avoid $D_\sigma$ and satisfy $n_\sigma^{\mathcal A}(u)=0$.

The space $\mathcal A^{1+1}_{d,filt}(y,\gamma^+;\gamma_0)$ is a union of connected components of $\mathcal A^{1+1}_d(y,\gamma^+;\gamma_0)$, and its boundary inherits the filtered condition. In other words, they are the same except in the following two ways. First, all annuli, broken annuli, and closed-open maps are replaced by their filtered versions. Second, for $n_\sigma(\gamma_0)=n_\sigma(\gamma^1)>0$, the terms 
\[
 \coprod_{\tild\gamma\in\mathscr X(L_1,L_1)}\mathcal R^{2+1}(\tild\gamma,\gamma^1;\gamma_0)\times\mathcal A^{0+1}_d(y;\tild\gamma)
\]
in \eqref{eq:second htopy 1 input outside breaks} no longer contribute.

Define a linear map $h^2_y\colon\hom^*_{\ssigma'}(L_0,L_1)\to\hom^{*+\deg(y)+n-2}_{\ssigma'}(L_0,L_1)$ to depend linearly on $y$ and, for $y=\gamma_d\otimes\dotsm\otimes\gamma_1\in CC_*(\Bss)$, to be given by
\begin{equation}\label{eq:second htopy formula}
\begin{gathered}
 h^2_y(\gamma^m\otimes\dotsm\otimes\gamma^0)\;=\hspace{-15mm}\sum_{\substack{0\le i\le m+1\\ n_\sigma(\gamma^r)=0\,\forall r<i\\ \tild\gamma\text{ making the result composable}\\ \deg(\tild\gamma)=\deg(y)+n-1}}\hspace{-15mm}\#\mathcal A^{0+1}_d(y;\tild\gamma)\cdot\gamma^m\otimes\dotsm\otimes\gamma^i\otimes\tild\gamma\otimes\gamma^{i-1}\otimes\dotsm\otimes\gamma^0\\
 +\hspace{-15mm}\sum_{\substack{0\le i\le m\\ n_\sigma(\gamma^r)=0\,\forall r<i\\ \tild\gamma\text{ making the result composable}\\ \deg(\tild\gamma)=\deg(y)+\deg(\gamma^i)+n-2}}\hspace{-15mm}\#\mathcal A^{1+1}_{d,filt}(y,\gamma^i;\tild\gamma)\cdot\gamma^m\otimes\dotsm\otimes\gamma^{i+1}\otimes\tild\gamma\otimes\gamma^{i-1}\otimes\dotsm\otimes\gamma^0.
\end{gathered}
\end{equation}

By essentially the same argument as for Lemma \ref{lem:first htopy effect}, we conclude
\begin{lemma}\label{lem:second htopy effect}
 Up to terms which decrease the main filtration,
 \begin{equation}\label{eq:second htopy effect}
  h^2_y\mu_{\ssigma'}^1+\mu_{\ssigma'}^1h^2_y=h^2_{\delta y}+\tild A_y^{0,0}+\tild A_y^{0,1}+\mathcal{CO}^{filt}_{\mathcal{OC}(y)}.
 \end{equation}
 \qed
\end{lemma}

\subsection{The last homotopy}\label{sec:last htopy}

Our goal now is to construct, for a saddle unit $f_\sigma\in SC_\ssigma^1(M[\sigma])$, a homotopy $h^3_{f_\sigma}$ between $\mathcal{CO}^{filt}_{f_\sigma}$ and an operation $\mathrm{id}_\sigma$ which, while not the identity, induces the identity on the portion of the associated graded of $\hom^*_{\ssigma'}(L_0,L_1)$ which does not lie in $\hom^*_\ssigma(L_0,L_1)$.

Thus, let $K^f$ be a Floer datum on $\C$ giving rise to $f_\sigma$ as in \ref{defn:saddle unit}. Denote by $D^Y$ the closed unit disk with a puncture $\zeta_0^Y$ at $-1$, which we equip with a negative strip-like end $\epsilon_0^Y$ and a family of finite cylinders $\delta^Y$ as follows. Let $D^{0,1}_1$ be as in Section \ref{sec:CO maps} with negative strip-like end $\epsilon_0$. Then $\epsilon_0$ induces $\epsilon_0^Y$ via the unique biholomorphism $D^{0,1}_1\to D^Y\setminus\{0\}$. For the cylinders, we are interested in a $(0,\frac12]$-parametrized family
\[
 \delta^{D,Y}(\rho)\colon[a_\rho,b_\rho]\times S^1\to\mathrm{int}(D^Y)
\]
which satisfies
\begin{enumerate}
 \item For $\rho$ close to $0$, $\delta^{D,Y}(\rho)$ agrees with the finite cylinder obtained by gluing $\epsilon^f$ on $\C$ to $\epsilon_+$ on $D^{0,1}_1$ with length $e^{\frac1\rho}$. Here, we are implicitly using the biholomorphism from the glued surface to $D^Y$ which sends $0\in\C$ to $0\in D^Y$ and $\zeta_0$ to $\zeta_0^Y$.
 \item $b_{\frac12}=a_{\frac12}$. In other words, at $\rho=\frac12$, the cylinder has width zero.
\end{enumerate}
We will think of this data as a $(0,1]$-parametrized space of Riemann surfaces with boundary, ends, and cylinders $D^Y(\rho)$, which for $\rho\in(0,\frac12]$ is equipped with the strip-like end $\epsilon_0^Y$ and finite cylinder $\delta^{D,Y}(\rho)$ and for $\rho>\frac12$ is equipped only with $\epsilon_0^Y$.

We consider smooth families of Floer data $\mathbf K^{D,Y}(\rho)$ on $D^Y(\rho)$ such that, in a gluing chart near $\rho=0$, $\mathbf K^{D,Y}$ extends smoothly to $0$, where it is conformally equivalent to the Floer data chosen for $\C$ and $D^{0,1}_1$. Let $\mathcal K^{D,Y}(M[\sigma])$ denote the space of such families.

In addition to the above data, choose $p_D\colon(0,1]\to[0,1]$ to be a nondecreasing smooth function which is $0$ on $(0,\frac13]$ and $1$ on $[\frac12,1]$. Similarly, choose a smooth isotopy $Y_\sigma(\rho)$ of properly embedded hypersurfaces in $\widehat{M[\sigma]}$ which avoid $\sigma(\hat F\times\R_{\ge0})$ for all $\rho$ and satisfy the following conditions. First, $Y_\sigma(\rho)=Y_\sigma$ for $\rho\le\frac14$. Second, $Y_\sigma(\rho)\in\mathrm{image}(\sigma)$ for $\rho\ge\frac13$, and moreover $Y_\sigma(1)$ is transverse to all chords between interior Lagrangians which appear in $\Wsp(M[\sigma])$.

Let $L$ be an interior Lagrangian of $M[\sigma]$, and let $\gamma\in\mathscr X(L,L)$. Given a family of Floer data $\mathbf K^{D,Y}\in\mathcal K^{D,Y}(M[\sigma])$, let $\mathcal D^Y(\gamma)$ be the union over all $\rho\in(0,1]$ of the spaces $\mathcal D^Y_\rho(\gamma)$ of maps
\[
 u\colon D^Y\to\widehat{M[\sigma]}
\]
satisfying the following conditions.
\begin{enumerate}
 \item $u$ satisfies \eqref{eq:gen Floer w/ cylinders} for $\mathbf K^{D,Y}(\rho)$.
 \item $u(z)\in(\phi^{\tau_E(\rho)(z)})^*L$ for $z\in\partial D^Y$, and $u(\zeta_0^Y)=(\phi^{\tau(\rho)})^*\gamma$, where $\tau$ is the conformal factor that $\mathbf K^{D,Y}$ assigns to $\zeta_0^Y$.
 \item $u(p_D(\rho))\in Y_\sigma(\rho)$.
\end{enumerate}

Gromov compactness applies to $\mathcal D^Y(\gamma)$, and in fact the situation is better than expected. Namely, suppose $\rho_i$ is a sequence with $p_D(\rho_i)<1$ but $\lim p_D(\rho_i)=1$, and that $u_i\in\mathcal D^Y_{\rho_i}(\gamma)$ is a Gromov convergent sequence. Then we expect the Gromov limit to contain a bubble component with the incidence condition. However, in this case everything is exact, so all bubbles are constant. Thus, the incidence condition on the bubble is equivalent to an incidence condition on $\partial D^Y$, which means the incidence condition is a point of $L\cap Y_\sigma(\lim \rho_i)$. However, $Y_\sigma(\rho)$ lies in the image of $\sigma$ whenever $p_D(\rho)=1$, while $L$ is an interior Lagrangian, which means that no such point exists. This shows that, in fact, $\mathcal D^Y_\rho(\gamma)$ is empty for $p_D(\rho)$ sufficiently close to 1.

Applying the usual transversality argument, we now obtain
\begin{lemma}\label{lem:D^Y transversality}
 For generic $\mathbf K^{D,Y}\in\mathcal K^{D,Y}(M[\sigma])$, the moduli spaces $\mathcal D^Y(\gamma)$ are transversely cut out of dimension $\deg(\gamma)$ for all interior Lagrangians $L$ and all $\gamma\in\mathscr X(L,L)$. They are empty if $n_\sigma(\gamma)>0$. The codimension one boundary strata of their Gromov compactifications are in canonical identification with
   \[
    \coprod_{\tild x\in\mathscr X(H_t)}\mathcal R^{0+1}_1(\tild x;\gamma)\times\mathcal C(\tild x)\;\amalg\coprod_{\tild\gamma\in\mathscr X(L,L)}\mathcal R(\tild\gamma;\gamma)\times\mathcal D^Y(\tild\gamma)
   \]
  \qed
\end{lemma}

Fix a regular $\mathbf K^{D,Y}\in\mathcal K^{D,Y}(M[\sigma])$. We now repeat the above with one input. Choose a diffeomorphism $\Psi\colon[0,1]\to\ol{\mathcal R}^{1+1}_1$ such that $\Psi(0)$ is nodal at the first positive puncture of $\Sigma^{2+1}\in\mathcal R^{2+1}$ and $\Psi(1)$ is nodal at the second positive puncture of $\Sigma^{2+1}$. Let $Z^Y$ be the strip $\R\times[0,1]$, where $\R\times\{i\}$ is $\partial_i Z^Y$, and the ends $-\infty$ and $+\infty$ are labeled $\zeta_0^Y$ and $\zeta_1^Y$, respectively. Then $\Psi$ induces a $(0,1)$-parametrized family of strip-like ends $\epsilon^{Z,Y}_q$ on $Z^Y$ by specifying, for $q\in(0,1)$, the embedding $\Psi(q)\into Z^Y$ which sends $\partial\Psi(q)$ to $\partial Z^Y$, $\zeta_i$ to $\zeta_i^Y$, and $\zeta_+$ to a point on $\{0\}\times(0,1)$. Extend this to a $(0,1]\times(0,1)$-parametrized family $\boldsymbol\epsilon^{Z,Y}$ with the following properties
\begin{enumerate}
 \item For $(\rho,q)\in(0,1]\times(0,1)$ with $\rho$ small or $q\le\frac14$ or $q\ge\frac34$, $\boldsymbol\epsilon^{Z,Y}(\rho,q)$ agrees with $\epsilon^{Z,Y}_q$.
 \item For $\rho$ close to $1$ and $q\in[\frac14,\frac34]$, $\boldsymbol\epsilon^{Z,Y}(\rho,q)$ agrees up to shift with the canonical strip-like ends on $Z$.
\end{enumerate}

Similarly, choose a smooth $(0,\frac12]\times(0,1)$-parametrized family of finite cylinders $\delta^{Z,Y}$ on $Z^Y$ as follows.
\begin{enumerate}
 \item For $(\rho,q)\in(0,1]\times(0,1)$ with $\rho$ small, $\delta^{Z,Y}(\rho,q)$ agrees with the finite cylinder obtained by gluing $\epsilon^f$ on $\C$ to $\epsilon_+$ on $\Psi(q)$ with length $e^{\frac1\rho}$.
 \item For $q$ close to $0$ or $1$, $\delta^{Z,Y}(\rho,q)$ agrees with the finite cylinder induced by gluing $D^Y(\rho)$ to the appropriate input of $\Sigma^{2+1}$ with length dictated by consistency with $\Psi$.
 \item For $\rho=\frac12$ and any $q$, the cylinder $\delta^{Z,Y}(\rho,q)$ has width zero.
\end{enumerate}
We then think of this data as a $(0,1]\times(0,1)$-parametrized space of Riemann surfaces with boundary, ends, and cylinders $Z^Y(\rho,q)$, which for $\rho\in(0,\frac12]$ is equipped with the strip-like ends $\boldsymbol\epsilon^{Z,Y}$ and finite cylinder $\delta^{D,Y}(\rho)$ and for $\rho>\frac12$ is equipped only with $\boldsymbol\epsilon^{Z,Y}$.

Consider the space $\mathcal K^{Z,Y}(M[\sigma])$ of smooth, $(0,1]\times(0,1)$-parametrized families of Floer data $\mathbf K^{Z,Y}$ on $Z^Y$ with the following properties.
\begin{enumerate}
 \item For $(\rho,q)\in(0,1]\times(0,1)$ with $\rho$ small, $\mathbf K^{Z,Y}(\rho,q)$ extends smoothly to $\rho=0$, where it agrees up to conformal equivalence with the Floer data chosen for $\C$ and $\Psi(q)$.
 \item For $q$ close to $0$ or $1$, $\mathbf K^{Z,Y}(\rho,q)$ is conformally close in the sense of Definition \ref{defn:consistent Floer data} to the Floer datum induced by gluing $D^Y(\rho)$ to the appropriate input of $\Sigma^{2+1}$ with length dictated by consistency with $\Psi$.
 \item For $\rho=1$ and $q\in[\frac14,\frac34]$, $\mathbf K^{Z,Y}(\rho,q)$ is conformally equivalent to the Floer perturbation $(H^{0,1},dt,1)$ which gives rise to the Floer differential.
\end{enumerate}

Finally, we choose a function $p_Z\colon(0,1]\times(0,1)\to Z^Y$ extending $p_D$ in the following sense. For $q$ near $0$ or $1$, $p_D(\rho)$ can be thought of as a point on the thick part of $D^Y$, and we require $p_Z(\rho,q)$ to be the point of $Z^Y$ with corresponds to $p_D(\rho)$ under gluing. For $\rho$ close to $0$, $p_Z(\rho,q)$ agrees with the point of $Z^Y(\rho,q)$ coming from the origin in $\C$ under the gluing of $\C$ to $\Psi(q)$. For $\rho$ close to $1$, we require that $p_Z(\rho,q)$ is $\rho$-independent and depends on $q$ in the following way. For $q\le\frac13$ or $q\ge\frac23$, $p_Z(\rho,q)$ is on the boundary of $Z^Y$. For $q\in[\frac13,\frac23]$, the $[0,1]$-component of $p_Z$ increases monotonically from $0$ to $1$.

Let $L_0$ and $L_1$ be interior Lagrangians of $M[\sigma]$, and let $\gamma_0,\gamma_1\in\mathscr X(L_0,L_1)$. For a universal choice $\mathbf K^{Z,Y}\in\mathcal K^{Z,Y}(M[\sigma])$, the corresponding space of holomorphic strips is called $\mathcal Z^Y(\gamma_1;\gamma_0)$ and is the union over all $(\rho,q)\in(0,1]\times(0,1)$ of the spaces $\mathcal Z^Y_{\rho,q}(\gamma)$ of maps
\[
 u\colon Z^Y\to\widehat{M[\sigma]}
\]
satisfying the conditions
\begin{enumerate}
 \item $u$ satisfies \eqref{eq:gen Floer w/ cylinders} for $\mathbf K^{Z,Y}(\rho,q)$.
 \item $u(z)\in(\phi^{\tau_E(\rho)(z)})^*L_i$ for $z\in\partial_i D^Y$, and $u(\zeta_i^Y)=(\phi^{\tau_i(\rho,q)})^*\gamma_i$, where $\tau_i$ is the conformal factor that $\mathbf K^{Z,Y}$ assigns to $\zeta_i^Y$.
 \item $u(p_Z(\rho,q))\in Y_\sigma(\rho)$.
\end{enumerate}

The compactness situation is the same as before, and we have
\begin{lemma}\label{lem:Z^Y transversality}
 For generic $\mathbf K^{Z,Y}\in\mathcal K^{Z,Y}(M[\sigma])$, the moduli spaces $\mathcal Z^Y(\gamma_1;\gamma_0)$ are transversely cut out of dimension $\deg(\gamma_0)-\deg(\gamma_1)+1$ for all interior Lagrangians $L_0,L_1$ and all $\gamma_0,\gamma_1\in\mathscr X(L_0,L_1)$. They are empty if $n_\sigma(\gamma_0)>n_\sigma(\gamma_1)$. The codimension one boundary strata of their Gromov compactifications are in canonical identification with
 \[
   \begin{gathered}
    \coprod_{\tild x\in\mathscr X(H_t)}\mathcal R^{1+1}_1(\tild x,\gamma_1;\gamma_0)\times\mathcal C(\tild x)\;\amalg\coprod_{\tild\gamma\in\mathscr X(L_0,L_1)}\mathcal Z^Y(\tild\gamma;\gamma_0)\times\mathcal R(\gamma_1;\tild\gamma)\\
    \amalg\coprod_{\tild\gamma\in\mathscr X(L_1,L_1)}\mathcal R^{2+1}(\tild\gamma,\gamma_1;\gamma_0)\times\mathcal D^Y(\tild\gamma)\;\amalg\coprod_{\tild\gamma\in\mathscr X(L_0,L_0)}\mathcal R^{2+1}(\gamma_1,\tild\gamma;\gamma_0)\times\mathcal D^Y(\tild\gamma)\\
    \amalg\coprod_{\tild\gamma\in\mathscr X(L_0,L_1)}\mathcal R(\tild\gamma;\gamma_0)\times\mathcal Z^Y(\gamma_1;\tild\gamma)\;\amalg\coprod_{\substack{t\in(0,1)\\ \gamma_0(t)\in Y_\sigma(1)}}\!\tild{\mathcal R}(\gamma_1;\gamma_0).
   \end{gathered}
 \]
 
 Of course, in the case where $\deg(\gamma_0)=\deg(\gamma_1)$, the last term only occurs when $\gamma_0=\gamma_1$ by regularity of Floer strips.
  \qed
\end{lemma}

As with the other homotopies, there is a filtered version of $\mathcal Z^Y(\gamma_1;\gamma_0)$. This is obtained by choosing, for all $(\rho,q)$, a path $e\colon[0,1]\to Z^Y$ such that $e(0)\in\partial_0 Z^Y$ and $e(1)=p_Z(\rho,q)$. The filtered component $\mathcal Z^Y_{filt}(\gamma_1;\gamma_0)$ consists of all $u\in\mathcal Z^Y(\gamma_1;\gamma_0)$ such that $u$ avoids $D_\sigma$ and $u\circ e$ has topological intersection number zero with $\sigma(\hat F\times\R_+)$. The condition that $Y_\sigma(\rho)$ avoids $\sigma(\hat F\times\R_{\ge0})$ ensures that this is indeed a connected component of $\mathcal Z^Y(\gamma_1;\gamma_0)$.

Fixing a regular $\mathbf K^{Z,Y}\in\mathcal K^{Z,Y}(M[\sigma])$, define a linear map $h_{f_\sigma}^3\colon\hom^*_{\ssigma'}(L_0,L_1)\to\hom^{*-1}_{\ssigma'}(L_0,L_1)$ by
\begin{equation}\label{eq:third htopy formula}
\begin{gathered}
 h_{f_\sigma}^3(\gamma^m\otimes\dotsm\otimes\gamma^0)=\hspace{-15mm}\sum_{\substack{0\le i\le m+1\\ n_\sigma(\gamma^r)=0\,\forall r<i\\ \tild\gamma\text{ making the result composable}\\ \deg(\tild\gamma)=0}}\hspace{-15mm}\#\mathcal D^Y(\tild\gamma)\cdot\gamma^m\otimes\dotsm\otimes\gamma^i\otimes\tild\gamma\otimes\gamma^{i-1}\otimes\dotsm\otimes\gamma^0\\
 +\hspace{-15mm}\sum_{\substack{0\le i\le m\\ n_\sigma(\gamma^r)=0\,\forall r<i\\ \tild\gamma\text{ making the result composable}\\ \deg(\tild\gamma)=\deg(\gamma^i)-1}}\hspace{-15mm}\#\mathcal Z^Y_{filt}(\gamma^i;\tild\gamma)\cdot\gamma^m\otimes\dotsm\otimes\gamma^{i+1}\otimes\tild\gamma\otimes\gamma^{i-1}\otimes\dotsm\otimes\gamma^0.
\end{gathered}
\end{equation}

As always (cf. Remark \ref{rmk:large m gives good coprod}), one needs to exclude certain bad terms to make $h_{f_\sigma}^3$ well defined. The argument for this is essentially the same as before, with the moving incidence condition handled as in Remark \ref{rmk:confinement for hypersurface to the right}.

We are finally rewarded for the bizarre filtered moduli spaces:
\begin{lemma}\label{lem:third htopy effect}
 Let $f_\sigma$ be a saddle unit, and let $\gamma=\gamma^m\otimes\dotsm\otimes\gamma^0\in\hom^*_{\ssigma'}(L_0,L_1)$ have $\mathfrak n(\gamma)=\sum_{i=0}^mn_\sigma(\gamma_i)>0$. Then, up to terms which decrease the main filtration,
 \begin{equation}\label{eq:third htopy effect}
  \left(h_{f_\sigma}^3\mu^1_{\ssigma'}+\mu^1_{\ssigma'}h_{f_\sigma}^3\right)(\gamma)=\mathcal{CO}^{filt}_{f_\sigma}(\gamma)+\gamma
 \end{equation}
\end{lemma}
\begin{proof}
 Following the usual argument, we obtain
 \[
  \left(h_{f_\sigma}^3\mu^1_{\ssigma'}+\mu^1_{\ssigma'}h_{f_\sigma}^3\right)(\gamma)=\mathcal{CO}^{filt}_{f_\sigma}(\gamma)+\hspace{-1cm}\sum_{\substack{0\le i\le m\\ n_\sigma(\gamma^r)=0\,\forall r<i\hspace{1cm}}}\hspace{-32mm}\sum_{\substack{t\in(0,1)\\ \hspace{6mm}\gamma^i(t)\in Y_\sigma(1)\\ \hspace{2cm}\gamma^i(t')\not\in\sigma(\hat F\times\R_+)\text{ for any }t'<t}}\hspace{-22mm}\#\tild{\mathcal R}(\gamma^i;\gamma^i)\cdot\gamma.
 \]
 The coefficient $\#\tild{\mathcal R}(\gamma^i;\gamma^i)$ is of course $1$, but we include it for clarity. Examining the conditions on the sums, we see that the only $i$ which contributes is the smallest $i$ such that $n_\sigma(\gamma^i)\ne0$. For this $\gamma^i$, let $t_0\in(0,1)$ be the first time at which $\gamma^i$ intersects $\sigma(\hat F\times\R_+)$. Since $\gamma^i$ starts outside the image of $\sigma$, it crosses $Y_\sigma(1)$ topologically once before $t_0$, and hence the sum contributes a total coefficient of $1$.
\end{proof}

For $y\in CC_*(\Bss)$ a closed low action stabilization with $\mathcal{OC}(y)=f_\sigma$ a saddle unit, set
\[
 h_y:=h_y^1+h_y^2+h_{f_\sigma}^3.
\]
This is the last ingredient we need to prove the stop removal formula:

\begin{proof}[Proof of Proposition \ref{prop:retracting homotopy} and Theorem \ref{thm:quotient}]
 Since $\sigma$ is strongly nondegenerate, Corollary \ref{cor:nondegeneracy for all large n} provides us for all sufficiently large $\mathbf m$ with a cycle $y_{\mathbf m}\in CC_*(\Bss)$ of small action such that $\mathcal{OC}({y_{\mathbf m}})$ is a saddle unit for Floer datum $K^{\C,\mathbf m}$. Define a linear map $\Delta_{y_{\mathbf m}}\colon\hom^*_{\ssigma'}(L_0,L_1)\to\hom^{*-1}_{\ssigma'}(L_0,L_1)$ on generators $\gamma$ by
 \[
  \Delta_{y_{\mathbf m}}(\gamma)=\begin{cases}
            \left(\Delta_{y_{\mathbf m}}^0+h_{y_{\mathbf m}}\right)\gamma&\text{ if }(\mathfrak n,\mathfrak m)(\gamma)\ge(1,0)\\
            0&\text{ otherwise.}
           \end{cases}
 \]
 As noted in Remarks \ref{rmk:large m gives good coprod} and \ref{rmk:large m gives good CO}, this is well defined for sufficiently large $\mathbf m$. Write $y={y_{\mathbf m}}$ for such an $\mathbf m$, and let $R_y=\mathrm{id}+\mu^1_{\ssigma'}\Delta_y+\Delta_y\mu^1_{\ssigma'}$ be the basic retraction as in Section \ref{sec:filtration on quotient}. Then tautologically $R_y|_{A^*_{1,0}}=\mathrm{id}$, since $A^*_{1,0}=\hom^*_\ssigma(L_0,L_1)$ is a subcomplex and $\Delta_y|_{A^*_{1,0}}=0$. So it suffices to show that, for $\gamma$ a generator with $\mathfrak n(\gamma)>0$, $(\mathfrak n,\mathfrak m)(R_y\gamma)<(\mathfrak n,\mathfrak m)(\gamma)$. Working modulo terms which decrease the main filtration, we compute
 \begin{align*}
  R_y\gamma &= \gamma + \left(\mu^1_{\ssigma'}\Delta^0_y+\Delta^0_y\mu^1_{\ssigma'}\right)\gamma + \left(\mu^1_{\ssigma'}h^1_y+h^1_y\mu^1_{\ssigma'}\right)\gamma + \left(\mu^1_{\ssigma'}h^2_y+h^2_y\mu^1_{\ssigma'}\right)\gamma + \left(\mu^1_{\ssigma'}h^3_{f_\sigma}+h^3_{f_\sigma}\mu^1_{\ssigma'}\right)\gamma \\
  &= \gamma + \left(A^{0,0}_y+A^{0,1}_y\right)\gamma + \left(A_y^{0,0}+A_y^{0,1}+\tild A_y^{0,0}+\tild A_y^{0,1}\right)\gamma\\
  &\hspace{6cm}+\left(\tild A_y^{0,0}+\tild A_y^{0,1}+\mathcal{CO}^{filt}_{\mathcal{OC}(y)}\right)\gamma+\left(\mathcal{CO}^{filt}_{f_\sigma}(\gamma)+\mathrm{id}\right)\gamma.
 \end{align*}
 Here, we have used Lemmas \ref{lem:main htopy effect}, \ref{lem:first htopy effect}, \ref{lem:second htopy effect}, and \ref{lem:third htopy effect}. Note that they apply even though $\Delta_y|_{A^*_{1,0}}$ was defined to be zero. This is because all operations are given by a count of holomorphic curves, so they satisfy positivity of intersection. This means that in any count which fails to decrease the main filtration, all intermediate terms have the same intersection number $\mathfrak n$ as the input and the output. In particular, they never witness the discrepancy $\Delta_{y_{\mathbf m}}|_{A^*_{1,0}}\ne\left(\Delta_{y_{\mathbf m}}^0+h_{y_{\mathbf m}}\right)|_{A^*_{1,0}}$.
 
 Now, using that $\mathcal{OC}(y)=f_\sigma$, the terms on the right-hand-side all cancel pairwise, which proves the Proposition, and hence the Theorem.
\end{proof}

\end{document}